\documentclass[a4paper]{amsart}
\usepackage[active]{srcltx}
\usepackage[all]{xy}
\usepackage{subfig}
\usepackage{hyperref}
\usepackage{mathrsfs}

\usepackage{amssymb,latexsym}
\usepackage{mathrsfs}
\usepackage{graphics}
\usepackage{latexsym}
\usepackage{psfrag}
\usepackage{verbatim}
\usepackage{fancyhdr,fancyvrb}
\usepackage[dvips]{graphicx}
\usepackage{pifont,marvosym}
\usepackage{mathbbol}
\usepackage{enumerate}
\usepackage[usenames]{color}
\usepackage{pifont,marvosym}
\usepackage{relsize}

\newcounter{assu}

\theoremstyle{plain}
\newtheorem{lemma}{Lemma}[section]
\newtheorem{theorem}[lemma]{Theorem}
\newtheorem{proposition}[lemma]{Proposition}
\newtheorem{corollary}[lemma]{Corollary}

\theoremstyle{definition}
\newtheorem{assumption}[assu]{Assumption}
\newtheorem{definition}[lemma]{Definition}
\newtheorem{remark}[lemma]{Remark}
\newtheorem{example}{Example}

\numberwithin{equation}{section}

\newcommand{\R}{\mathbb{R}}
\newcommand{\N}{\mathbb{N}}
\newcommand{\Q}{\mathbb{Q}}

\newcommand{\TV}{\text{\rm Tot.Var.}}
\newcommand{\BV}{\text{\rm BV}}

\newcommand{\supp}{\text{\rm supp}}

\newcommand{\Id}{\mathbb{I}}
\newcommand{\Z}{\mathbb{Z}}

\newcommand{\gr}{\textrm{graph}}

\newcommand{\diam}{\textrm{diam\,}}
\newcommand{\haus}{\mathcal{H}}
\newcommand{\ve}{\varepsilon}

\newcommand{\erre}{\mathbb{R}}
\newcommand{\enne}{\mathbb{N}}

\newcommand{\f}{\varphi}

\newcommand{\weak}{\rightharpoonup}

\begin{document}

\pagestyle{fancy}%
\renewcommand{\sectionmark}[1]{\markright{\thesection\ #1}}
\fancyhf{}%
\fancyhead[LE,RO]{\bfseries\thepage}%
\fancyhead[LO,RE]{\bfseries\rightmark}%
\fancyfoot[C]{Preprint SISSA 50/2009/M (August 6, 2009)}%
\renewcommand{\footrulewidth}{.5pt}
\addtolength{\headheight}{0.5pt}
\fancypagestyle{plain}{
  \fancyhead{}%
  \renewcommand{\headrulewidth}{0pt}
}%

\title[Monge problem with distance cost]{
The Monge problem for distance cost in geodesic spaces}

\author{Stefano Bianchini and Fabio Cavalletti \\ August 6, 2009 \\ Preprint SISSA  50/2009/M}

\address{SISSA, via Bonomea 265, IT-34136 Trieste (ITALY)}
\date{}

%
%

\bibliographystyle{plain}

\begin{abstract}
We address the Monge problem in metric spaces with a geodesic distance: $(X,d)$ is a Polish space and $d_L$ is a geodesic Borel distance which makes $(X,d_L)$ a non branching geodesic space. We show that under the assumption that geodesics are $d$-continuous and locally compact, we can reduce the transport problem to $1$-dimensional transport problems along geodesics.

We introduce two assumptions on the transport problem $\pi$ which imply that the conditional probabilities of the first marginal on each geodesic are continuous or absolutely continuous w.r.t. the $1$-dimensional Hausdorff distance induced by $d_L$. It is known that this regularity is sufficient for the construction of a transport map.

We study also the dynamics of transport along the geodesic, the stability of our conditions and show that in this setting $d_L$-cyclical monotonicity is not sufficient for optimality.
\end{abstract}

\maketitle

\tableofcontents

\section{Introduction}
\label{S:intro}

This paper concerns the Monge transportation problem in geodesic spaces, i.e. metric spaces with a geodesic structure.
Given two Borel probability measure $\mu,\nu \in \mathcal{P}(X)$, where $(X,d)$ is a Polish space, we study the minimization of the functional
\[ 
\mathcal{I}(T) = \int d_{L}(x,T(x)) \mu(dy)  
\]
where $T$ varies over all Borel maps $T:X \to X$ such that $T_{\sharp}\mu = \nu$ and $d_{L}$ is a Borel distance that makes $(X,d_{L})$ a non branching geodesic space.

Before giving an overview of the paper and of the existence result, we recall which are the main results concerning the Monge problem. 

In the original formulation given by Monge in 1781 the problem was settled in $\erre^{n}$, 
with the cost given by the Euclidean norm and the measures $\mu, \nu$ were supposed to be absolutely 
continuous and supported on two disjoint compact sets.
The original problem remained unsolved for a long time. 
In 1978 Sudakov \cite{sudak} claimed to have a solution for any distance cost function induced 
by a norm: an essential ingredient in the proof was that if $\mu \ll \mathcal{L}^{d}$ and $\mathcal{L}^{d}$-a.e. $\erre^{d}$ can be decomposed into 
convex sets of dimension $k$, then
then the conditional probabilities are absolutely continuous with respect to the $\haus^{k}$ measure of the correct dimension. 
But it turns out that when $d>2$, $0<k<d-1$ the property claimed by Sudakov is not true. An example with $d=3$, $k=1$ can be found in \cite{larm}.

The Euclidean case has been correctly solved only during the last decade. L. C. Evans and W. Gangbo in \cite{evagangbo} 
solved the problem under the assumptions that 
$\textrm{spt}\,\mu \cap  \textrm{spt}\,\nu = \emptyset$,  $\mu,\nu \ll \mathcal{L}^{d}$ and their densities are Lipschitz functions with compact support.
The first existence results for general absolutely continuous measures $\mu,\nu$ with compact support have been independently obtained by 
L. Caffarelli, M. Feldman and R.J. McCann in \cite{caffafeldmc} and by N. Trudinger and X.J. Wang in \cite{trudiwang}. 
Afterwards M. Feldman and R.J. McCann \cite{feldcann:mani} extended the results to manifolds with geodesic cost. 
The case of a general norm as cost function on $\erre^{d}$, including also the case with non strictly convex unitary ball, 
has been solved first in the particular case of crystalline norm by L. Ambrosio, B. Kirchheim and A. Pratelli in 
\cite{ambprat:crist}, and then in fully generality independently by L. Caravenna in \cite{caravenna:Monge} and by T. Champion and L. De Pascale in 
\cite{champdepasc:Monge}. 

\subsection{Overview of the paper}
\label{Ss:over}
The presence of $1$-dimensional sets (the geodesics) along which the cost is linear is a strong degeneracy for transport problems. This degeneracy is equivalent to the following problem in $\R$: if $\mu$ is concentrated on $(-\infty,0]$, and $\nu$ is concentrated on $[0,+\infty)$, then every transference plan is optimal for the $1$-dimensional distance cost $|\cdot|$. In fact, every $\pi \in \Pi(\mu,\nu)$ is supported on the set $(-\infty,0] \times [0,+\infty)$, on which $|x-y| = y-x$ and thus
\[
\int |x-y| \pi(dxdy) = - \int x \mu(dx) + \int y \nu(dy).
\]
Nevertheless, for this easy case an explicit map $T : \R \to \R$ can be constructed if $\mu$ is continuous (i.e. without atoms): the easiest choice is the monotone map, a minimizer of the quadratic cost $|\cdot|^2$.

The approach suggested by the above simple case is the following:
\begin{enumerate}
\item reduce the problem to transportation problems along distinct geodesics;
\item show that the disintegration of the marginal $\mu$ on each geodesic is continuous;
\item find a transport map on each geodesic and piece them together.
\end{enumerate}
While the last point can be seen as an application of selection principles in Polish spaces, the first two points are more subtle.

The geodesics used by a given transference plan $\pi$ to transport mass can be obtained from a set $\Gamma$ on which $\pi$ is concentrated. If $\pi$ wants to be a minimizer, then it certainly chooses the shortest paths: however the metric space can be branching, i.e. geodesics can bifurcate. 

In this paper we assume that the space is non branching.

Under this assumption, a cyclically monotone plan $\pi$ 
yields a natural partition $R$ of a subset of the transport set $\mathcal T_e$, i.e. the set of points on the geodesics used by $\pi$: 
defining
\begin{itemize}
\item the set $\mathcal T$ made of inner points of geodesics,
\item the set $a \cup b := \mathcal T_e \setminus \mathcal T$ of initial points $a$ and end points $b$,
\end{itemize}
the non branching assumption and the cyclical monotonicity of $\Gamma$ imply that the geodesics used by $\pi$ are a partition on $\mathcal T$. 
In general in $a$
there are points from which more than geodesic starts and in $b$ there are points in which more than one geodesic ends, hence 
being on a geodesic can't be an equivalence relation on the set $a\cup b$.
For example one can think to the unit circle with $\mu = \delta_0$ and $\nu = \delta_\pi$. 

We note here that $\pi$ gives also a direction along each component of $R$, as the one dimensional example above shows.

Even if we have a natural partition $R$ in $\mathcal T$ and $\mu(a \cup b) = 0$, we cannot reduce the transport problem to one dimensional problems: a necessary and sufficient condition is that the disintegration of the measure $\mu$ is strongly consistent, which is equivalent to the fact that there exists a $\mu$-measurable quotient map $f : \mathcal T \to \mathcal T$ of the equivalence relation $R$. In this case, one can write
\[
m := f_\sharp \mu, \quad \mu = \int \mu_y m(dy), \quad \mu_y(f^{-1}(y)) = 1,
\]
i.e. the conditional probabilities $\mu_y$ are concentrated on the counterimages $f^{-1}(y)$ (which are single geodesics). We can obtain the one dimensional problems by 
partitioning $\pi$ w.r.t. the partition $R \times (X \times X)$,
\[
\pi = \int \pi_y m(dy), \quad \nu = \int \nu_y m(dy) \quad \nu_y := (P_2)_\sharp \pi_y,
\]
and considering the one dimensional problems along the geodesic $R(y)$ with marginals $\mu_y$, $\nu_y$ and cost $|\cdot|$, the length on the geodesic. At this point we can study the problem of the regularity of the conditional probabilities $\mu_y$.

The fact that there exists a strongly consistent disintegration is a property of the geodesics of the metric space. In the setting considered in this paper, $(X,d_L)$ is a non branching geodesic space, not necessarily Polish. To assure that standard measure theory can be used, there exists a second distance $d$ on $X$ which makes $(X,d)$ Polish, and $d_L$ is a Borel function on $X \times X$ with the metric $d \times d$. 

Note that we do not require $d_L$ to be l.s.c., so the existence of an optimal plan $\pi$ is not assured, but we consider a $d_L$-cyclically monotone transference plan $\pi$. 
It is worth notice that we do not use the existence of optimal potentials $(\phi,\psi)$, as well as the optimality of $\pi$.

Thus, let $\pi$ be a $d_L$-cyclically monotone transference plan. The strong consistency of the disintegration of $\mu$ along the geodesic used by $\pi$ 
is a consequence of the topological properties of the geodesics of $d_L$ considered as curves in $(X,d)$: in fact we require that they are $d$-continuous and locally compact. Under this assumption, on $\mathcal T$ (the transport set without end points) it is possible to disintegrate $\mu$. Moreover, a natural operation on sets can be considered: the translation along geodesics. If $A$ is a subset of $\mathcal T$, we denote by $A_t$ the set translated by $t$ in the direction determined by $\pi$.

It turns out that the fact that $\mu(a \cup b) = 0$ and the measures $\mu_y$ are continuous depends on how the function $t \mapsto \mu(A_t)$ behaves. We can now state the main result.

\begin{theorem}[Lemma \ref{L:puntini} and Proposition \ref{P:nonatoms}]
\label{T:-1}
If $\sharp \{t > 0: \mu(A_t) > 0\}$ is uncountable for all $A$ Borel such that $\mu(A) > 0$, then $\mu(a \cup b) = 0$ and the conditional probabilities $\mu_y$ are continuous.
\end{theorem}

This is sufficient to solve the Monge problem, i.e. to find a transport map which has the same cost as $\pi$. A second result concerns a stronger regularity assumption.

\begin{theorem}[Theorem \ref{teo:a.c.}]
\label{T:1}
Assume that $\mathcal L^1(\{t > 0: \mu(A_t) > 0\}) > 0$ for all $A$ Borel such that $\mu(A) > 0$. Then $\mu(a \cup b) = 0$ and $\mu_y$ is a.c. w.r.t. the $1$-dimensional Hausdorff measure $\mathcal H^1_{d_L}$ induced by $d_L$.
\end{theorem}

The assumption of the above theorem and the assumption $d_L \geq d$ allows to define a current in $(X,d)$ which represents the vector field corresponding to the translation $A \mapsto A_t$, and moreover to solve the equation
\[
\partial U = \mu - \nu
\]
is the sense of current in metric space.

The final results of the paper are the stability of these conditions under Measure-Gromov-Hausdorff like convergence of structures $(X_n,d_n,d_{L,n},\pi_{n})$. The conclusion is that a sort of uniform integrability condition on the conditional probability w.r.t. $\mathcal H^1_{d_{L,n}}$ passes to the limit, so that one can verify by approximation if Theorem \ref{T:1} holds.

In the case $d = d_L$, considering a reference measure $\eta \in \mathcal{P}(X)$ such that
$(X,d,\eta)$ is a geodesic measure space satisfying the $MCP(K,N)$ for $K\in \erre$ and $N\geq1$, 
the application of the above results together with $MCP$-condition implies that Assumption \ref{A:NDE} holds for 
$\eta$ w.r.t. the optimal flow induced by any $d$-monotone plan $\pi\in \Pi(\mu,\nu)$. 
Hence if $\mu \ll \eta$,
the existence of a minimizer for the Monge minimization problem with marginal $\mu$ and $\nu$ follows.

To conclude this introduction, we observe that it is probably possible to extend these results to the case where $-d_L$ is a Souslin function on $(X \times X,d \times d)$: 
this means that $d_{L}^{-1}(-\infty,t)$ is an analytic set in the sense of Souslin.

The interested reader can refer for example to the analysis of \cite{biacar:cmono}.


\subsection{Structure of the paper}
\label{Ss:struct}

The paper is organized as follows.

In Section \ref{S:preli}, we recall the basic mathematical results we use. In Section \ref{Ss:univmeas} the fundamentals of projective set theory are listed. In Section \ref{S:disintegrazione} we recall the Disintegration Theorem, using the version of \cite{biacar:cmono}. Next, the basic results of selection principles are in Section \ref{Ss:sele}, and in Section \ref{ss:Metric} we define the geodesic structure $(X,d,d_L)$ which is studied in this paper. Finally, Section \ref{Ss:General Facts} recalls some fundamental results in optimal transportation theory.

The next three sections are the key ones.

Section \ref{S:Optimal} shows how using only the $d_L$-cyclical monotonicity of a set $\Gamma$ we can obtain a partial order relation $G \subset X \times X$ as follows (Lemma \ref{L:analGR} and Proposition \ref{P:equiv}): $xGy$ iff there exists $(w,z) \in \Gamma$ and a geodesic $\gamma: [0,1] \to X$, with 
$\gamma(0)=w$, $\gamma(1)=z$, such that $x$, $y$ belongs to $\gamma$ and $\gamma^{-1}(x) \leq \gamma^{-1}(y)$. 
This set $G$ is analytic, and allows to define
\begin{itemize}
\item the transport ray set $R$ \eqref{E:Rray},
\item the transport sets $\mathcal T_e$, $\mathcal T$ (with and without and points) \eqref{E:TR0},
\item the set of initial points $a$ and final points $b$ \eqref{E:endpoint0}.
\end{itemize}
Moreover we show that $R \llcorner_{\mathcal T \times \mathcal T}$ is an equivalence relation (Proposition \ref{P:equiv}), we can assume that the set of final points $b$ can be taken $\mu$-negligible (Lemma \ref{L:finini0}), and in two final remarks we study what happens in the case more regularity on the cost $d_L$ is assumed, Remark \ref{R:lsccase1} and Remark \ref{R:TR}.\\
Notice that in the case $d=d_{L}$ the existence of a Lipschitz potential $\f$, one can take 
\[
\Gamma = G = \Big\{ (x,y) : \f(x)-\f(y)=d(x,y) \Big\}.
\]
Thus the main result of this section is that these sets can be defined even if the potential does not exist.

Section \ref{S:partition} proves that the continuity and local compactness of geodesics imply that the disintegration induced by $R$ on $\mathcal T$ is strongly consistent (Proposition \ref{P:sicogrF}): as Example \ref{Ex:nonsection} 
shows, the strong consistency of the disintegration is a non trivial property of the metric spaces we are considering.\\
Using this fact, we can define an order preserving map $g$ which maps our transport problem into a transport problem on $\mathcal S \times \R$, where $\mathcal S$ is a cross section of $R$ (Proposition \ref{P:gammaclass}). Finally we show that under this assumption there exists a transference plan with the same cost of $\pi$ which leaves the common mass $\mu \wedge \nu$ at the same place (note that in general this operation lowers the transference cost).

In Section \ref{S:regurlr} we prove Theorem \ref{T:-1} and Theorem \ref{T:1}. We first introduce the operation $A \mapsto A_t$, the translation along geodesics \eqref{E:At}, and show that $t \mapsto \mu(A_t)$ is a Souslin function if $A$ is analytic (Lemma \ref{L:measumuAt}). \\ 
Next we show that under the assumption
\[
\mu(A) > 0 \quad \Longrightarrow \quad \sharp \big\{ t > 0: \mu(A_t) > 0 \big\} > \aleph_0
\]
the set of initial points $a$ is $\mu$-negligible (Lemma \ref{L:puntini}) and the conditional probabilities $\mu_y$ are continuous. \\
Finally, we show that under the stronger assumption
\begin{equation}
\label{E:intro0}
\mu(A) > 0 \quad \Longrightarrow \quad \int_{\R^+} \mu(A_t) dt > 0,
\end{equation}
the conditional probabilities $\mu_y$ are a.c. w.r.t. $\mathcal H^1_{d_L}$ (Theorem \ref{teo:a.c.}). A final result shows that actually Condition \eqref{E:intro0} yields that $t \mapsto \mu(A_t)$ has more regularity than just integrability (Proposition \ref{P:peloso}) it is in fact continuous

After the above results, the solution of the Monge problem is routine, and it is done in Theorem \ref{T:mongeff} of Section \ref{S:Solution}.

Under Condition \ref{E:intro0} and $d \leq d_L$, in Section \ref{S:div} we give a dynamic interpretation to the transport along geodesics. In Definition \ref{D:dotgamma} we define the current $\dot g$ in $(X,d)$, which represents the flow induced by the transference plan $\pi$. Not much can be said of this flow, unless some regularity assumptions are considered. These assumptions are the natural extensions of properties of transportation problems in finite dimensional spaces. \\
If there exists a background measure $\eta$ whose disintegration along geodesics satisfies
\[
\eta = \int q_y \mathcal{H}^1_{d_L} m(dy), \quad q_y \in \BV, \ \int \TV(q_y) m(dy) < +\infty,
\]
then $\dot g$ is a normal current, i.e. its boundary is a bounded measure on $X$ (Lemma \ref{L:normalcurr}). \\
We can also consider the problem $\partial U = \mu - \nu$ in the sense of currents: Proposition \ref{P:bidual} gives a solution, and in the case $q_y(t) > 0$ for $\mathcal H^1_{d_L}$-a.e. $t$ we can write represent $U = \rho \dot g$, i.e. the flow $\dot g$ multiplied by a scalar density $\rho$ (Corollary \ref{C:regudual}).

In Section \ref{S:limite} we address the stability of the assumptions under Measure-Gromov-Hausdorff-like convergence of 
structures $(X_n,d_n,d_{L,n},\pi_{n})$. 
Under a uniform integrability condition of $\mu_{y,n}$ w.r.t. $\mathcal H^1_{d_{L,n}}$ and 
a uniform bound on the $\pi_{n}$ transportation cost 
(Assumption \ref{A:equi} of Section \ref{Ss:GHC}), we show that the marginal $\mu$ can be represented 
as the image of a measure $r m \otimes \mathcal L^1$ by a Borel function $h : \mathcal T \times \R \to \mathcal T_e$, 
with $r \in L^1(m \otimes \mathcal L^1)$ (Proposition \ref{P:graphnn}). 
The key feature of $h$ is that $t \mapsto h(y,t)$ is a geodesic of $\mathcal T$ for $m$-a.e. $y \in \mathcal T$. \\
Thus while $h(0,\mathcal T)$ is not a cross section for $R$ (in that case we would have finished the proof), in Proposition \ref{P:finalapr} we show which conditions on $h$ imply that $\mu$ can be disintegrated with a.c. conditional probabilities, and we verify that this is our case in Theorem \ref{T:aproxxgg}. \\
In two remarks we suggest how to pass also uniform estimates on the disintegration on $(X_n,d_n,d_{L,n})$ to the transference problem in $(X,d,d_L)$ (Remark \ref{R:morereg} and Remark \ref{R:reguadd}).

In Section \ref{S:mcp} we consider an application of the results obtained in the previous sections. 
We assume $d=d_{L}$ and the existence of background probability measure $\eta$ such that $(X,d,\eta)$ 
satisfies $MCP(K,N)$ (Definition \ref{D:mcp}).
In this framework we prove that for any $d$-cyclically monotone transference plan $\pi$, $\eta$ admits a disintegration along the geodesics 
used by $\pi$ with marginal probabilities absolutely continuous w.r.t. $\haus^{1}$ (Theorem \ref{T:regularity}). 
This implies directly (Corollary \ref{C:conclu}) that if $\mu\ll \eta$ the Monge minimization problem with marginals $\mu$ and $\nu$
admits a solution. 
The final result of the section (Lemma \ref{L:mhdregul}) shows that we can solve the dynamical problem $\partial U = \mu - \nu$ 
with $U=\rho \dot g$, and if the support of $\mu$ and $\nu$ are disjoint $U$ is a normal current.

The last section contains two important examples. In Example \ref{Ex:nonsection} we show that if the geodesics are not locally compact, then in general the disintegration along transport rays is not strongly supported. In Example \ref{Ex:nooptir} we show that under our assumptions the $c$-monotonicity is not sufficient for optimality.

We end with a list of notations, Section \ref{S:notation}.

\section{Preliminaries}
\label{S:preli}

In this section we recall some general facts about projective classes, the Disintegration Theorem for measures, measurable selection principles, geodesic spaces and optimal transportation problems.

\subsection{Borel, projective and universally measurable sets}
\label{Ss:univmeas}

The \emph{projective class $\Sigma^1_1(X)$} is the family of subsets $A$ of the Polish space $X$ for which there exists $Y$ Polish and $B \in \mathcal{B}(X \times Y)$ such that $A = P_1(B)$. The \emph{coprojective class $\Pi^1_1(X)$} is the complement in $X$ of the class $\Sigma^1_1(X)$. The class $\Sigma^1_1$ is called \emph{the class of analytic sets}, and $\Pi^1_1$ are the \emph{coanalytic sets}.

The \emph{projective class $\Sigma^1_{n+1}(X)$} is the family of subsets $A$ of the Polish space $X$ for which there exists $Y$ Polish and $B \in \Pi^1_n(X \times Y)$ such that $A = P_1(B)$. The \emph{coprojective class $\Pi^1_{n+1}(X)$} is the complement in $X$ of the class $\Sigma^1_{n+1}$.

If $\Sigma^1_n$, $\Pi^1_n$ are the projective, coprojective pointclasses, then the following holds (Chapter 4 of \cite{Sri:courseborel}):
\begin{enumerate}
\item $\Sigma^1_n$, $\Pi^1_n$ are closed under countable unions, intersections (in particular they are monotone classes);
\item $\Sigma^1_n$ is closed w.r.t. projections, $\Pi^1_n$ is closed w.r.t. coprojections;
\item if $A \in \Sigma^1_n$, then $X \setminus A \in \Pi^1_n$;
\item the \emph{ambiguous class} $\Delta^1_n = \Sigma^1_n \cap \Pi^1_n$ is a $\sigma$-algebra
and $\Sigma^1_n \cup \Pi^1_n \subset \Delta^1_{n+1}$.
\end{enumerate}
We will denote by $\mathcal{A}$ the $\sigma$-algebra generated by $\Sigma^1_1$: clearly $\mathcal{B} = \Delta^1_1 \subset \mathcal{A} \subset \Delta^1_2$.

We recall that a subset of $X$ Polish is \emph{universally measurable} if it belongs to all completed $\sigma$-algebras of all Borel measures on $X$:
it can be proved that every set in $\mathcal{A}$ is universally measurable.
We say that $f:X \to \erre \cup \{\pm \infty\}$ is a \emph{Souslin function} if $f^{-1}(t,+\infty] \in \Sigma^{1}_{1}$.

\begin{lemma}
\label{L:measuregP}
If $f : X \to Y$ is universally measurable, then $f^{-1}(U)$ is universally measurable if $U$ is.
\end{lemma}

\begin{proof}
If $\mu \in \mathcal{M}(X)$, then $f_\sharp \mu \in \mathcal{M}(Y)$, so for $U \subset Y$ universally measurable there exist Borel sets $B$, $B'$ such that $B \subset U \subset B'$ and
\[
0 = (f_\sharp \mu)(B' \setminus B) = \mu \big( f^{-1}(B') \setminus f^{-1}(B) \big).
\]
Since $f^{-1}(B), f^{-1}(B') \subset X$ are universally measurable, there exists Borel sets $C$, $C'$ such that
\[
C \subset f^{-1}(B) \subset f^{-1}(U) \subset f^{-1}(B') \subset C'
\]
and $\mu(C' \setminus C) = 0$. The conclusion follows.
\end{proof}

\subsection{Disintegration of measures}
\label{S:disintegrazione}

Given a measurable space $(R, \mathscr{R})$ and a function $r: R \to S$, with $S$ generic set, we can endow $S$ with the \emph{push forward $\sigma$-algebra} $\mathscr{S}$ of $\mathscr{R}$:
\[
Q \in \mathscr{S} \quad \Longleftrightarrow \quad r^{-1}(Q) \in \mathscr{R},
\]
which could be also defined as the biggest $\sigma$-algebra on $S$ such that $r$ is measurable. Moreover given a measure space 
$(R,\mathscr{R},\rho)$, the \emph{push forward measure} $\eta$ is then defined as $\eta := (r_{\sharp}\rho)$.

Consider a probability space $(R, \mathscr{R},\rho)$ and its push forward measure space $(S,\mathscr{S},\eta)$ induced by a map $r$. From the above definition the map $r$ is clearly measurable and inverse measure preserving.

\begin{definition}
\label{defi:dis}
A \emph{disintegration} of $\rho$ \emph{consistent with} $r$ is a map $\rho: \mathscr{R} \times S \to [0,1]$ such that
\begin{enumerate}
\item  $\rho_{s}(\cdot)$ is a probability measure on $(R,\mathscr{R})$ for all $s\in S$,
\item  $\rho_{\cdot}(B)$ is $\eta$-measurable for all $B \in \mathscr{R}$,
\end{enumerate}
and satisfies for all $B \in \mathscr{R}, C \in \mathscr{S}$ the consistency condition
\[
\rho\left(B \cap r^{-1}(C) \right) = \int_{C} \rho_{s}(B) \eta(ds).
\]
A disintegration is \emph{strongly consistent with respect to $r$} if for all $s$ we have $\rho_{s}(r^{-1}(s))=1$.
\end{definition}

The measures $\rho_s$ are called \emph{conditional probabilities}.

We say that a $\sigma$-algebra $\mathcal{H}$ is \emph{essentially countably generated} with respect to a measure $m$ if there exists a countably generated $\sigma$-algebra $\hat{\mathcal{H}}$ such that for all $A \in \mathcal{H}$ there exists $\hat{A} \in \hat{\mathcal{H}}$ such that $m (A \vartriangle \hat{A})=0$.

We recall the following version of the disintegration theorem that can be found on \cite{Fre:measuretheory4}, Section 452 (see \cite{biacar:cmono} for a direct proof).

\begin{theorem}[Disintegration of measures]
\label{T:disintr}
Assume that $(R,\mathscr{R},\rho)$ is a countably generated probability space, $R = \cup_{s \in S}R_{s}$ a partition of R, $r: R \to S$ the quotient map and $\left( S, \mathscr{S},\eta \right)$ the quotient measure space. Then  $\mathscr{S}$ is essentially countably generated w.r.t. $\eta$ and there exists a unique disintegration $s \mapsto \rho_{s}$ in the following sense: if $\rho_{1}, \rho_{2}$ are two consistent disintegration then $\rho_{1,s}(\cdot)=\rho_{2,s}(\cdot)$ for $\eta$-a.e. $s$.

If $\left\{ S_{n}\right\}_{n\in \enne}$ is a family essentially generating  $\mathscr{S}$ define the equivalence relation:
\[
s \sim s' \iff \   \{  s \in S_{n} \iff s'\in S_{n}, \ \forall\, n \in \enne\}.
\]
Denoting with p the quotient map associated to the above equivalence relation and with $(L,\mathscr{L}, \lambda)$ the quotient measure space, the following properties hold:
\begin{itemize}
\item $R_{l}:= \cup_{s\in p^{-1}(l)}R_{s} = (p \circ r)^{-1}(l)$ is $\rho$-measurable and $R = \cup_{l\in L}R_{l}$;
\item the disintegration $\rho = \int_{L}\rho_{l} \lambda(dl)$ satisfies $\rho_{l}(R_{l})=1$, for $\lambda$-a.e. $l$. In particular there exists a 
strongly consistent disintegration w.r.t. $p \circ r$;
\item the disintegration $\rho = \int_{S}\rho_{s} \eta(ds)$ satisfies $\rho_{s}= \rho_{p(s)}$ for $\eta$-a.e. $s$.
\end{itemize}
\end{theorem}

In particular we will use the following corollary.

\begin{corollary}
\label{C:disintegration}
If $(S,\mathscr{S})=(X,\mathcal{B}(X))$ with $X$ Polish space, then the disintegration is strongly consistent.
\end{corollary}

\subsection{Selection principles}
\label{Ss:sele}

Given a multivalued function $F: X \to Y$, $X$, $Y$ metric spaces, the \emph{graph} of $F$ is the set
\begin{equation}
\label{E:graphF}
\textrm{graph}(F) := \big\{ (x,y) : y \in F(x) \big\}.
\end{equation}
The \emph{inverse image} of a set $S\subset Y$ is defined as:
\begin{equation}
\label{E:inverseF}
F^{-1}(S) := \big\{ x \in X\ :\ F(x)\cap S \neq \emptyset \big\}.
\end{equation}
For $F \subset X \times Y$, we denote also the sets
\begin{equation}
\label{E:sectionxx}
F_x := F \cap \{x\} \times Y, \quad F^y := F \cap X \times \{y\}.
\end{equation}
In particular, $F(x) = P_2(\gr(F)_x)$, $F^{-1}(y) = P_1(\gr(F)^y)$. We denote by $F^{-1}$ the graph of the inverse function
\begin{equation}
\label{E:F-1def}
F^{-1} := \big\{ (x,y): (y,x) \in F \big\}.
\end{equation}

We say that $F$ is \emph{$\mathcal{R}$-measurable} if $F^{-1}(B) \in \mathcal{R}$ for all $B$ open. We say that $F$ is \emph{strongly Borel measurable} if inverse images of closed sets are Borel. A multivalued function is called \emph{upper-semicontinuous} if the preimage of every closed set is closed: in particular u.s.c. maps are strongly Borel measurable.

In the following we will not distinguish between a multifunction and its graph. Note that the \emph{domain of $F$} (i.e. the set $P_1(F)$) is in general a subset of $X$. The same convention will be used for functions, in the sense that their domain may be a subset of $X$.

Given $F \subset X \times Y$, a \emph{section $u$ of $F$} is a function from $P_1(F)$ to $Y$ such that $\textrm{graph}(u) \subset F$. We recall the following selection principle, Theorem 5.5.2 of \cite{Sri:courseborel}, page 198.

\begin{theorem}
\label{T:vanneuma}
Let $X$ and $Y$ be Polish spaces, $F \subset X \times Y$ analytic, and $\mathcal{A}$ the $\sigma$-algebra generated by the analytic subsets of X. Then there is an $\mathcal{A}$-measurable section $u : P_1(F) \to Y$ of $F$.
\end{theorem}

A \emph{cross-section of the equivalence relation $E$} is a set $S \subset E$ such that the intersection of $S$ with each equivalence class is a singleton. We recall that a set $A \subset X$ is saturated for the equivalence relation $E \subset X \times X$ if $A = \cup_{x \in A} E(x)$.

The next result is taken from \cite{Sri:courseborel}, Theorem 5.2.1.

\begin{theorem}
\label{T:KRN}
Let $Y$ be a Polish space, $X$ a nonempty set, and $\mathcal{L}$ a $\sigma$-algebra of subset of $X$. 
Every $\mathcal{L}$-measurable, closed value multifunction $F:X \to Y$ admits an $\mathcal{L}$-measurable section.
\end{theorem}

A standard corollary of the above selection principle is that if the disintegration is strongly consistent in a Polish space, then up to a saturated set of negligible measure there exists a Borel cross-section.

In particular, we will use the following corollary.

\begin{corollary}
\label{C:weelsupprr}
Let $F \subset X \times X$ be $\mathcal{A}$-measurable, $X$ Polish, such that $F_x$ is closed and define the equivalence relation $x \sim y \ \Leftrightarrow \ F(x) = F(y)$. Then there exists a $\mathcal{A}$-section $f : P_1(F) \to X$ such that $(x,f(x)) \in F$ and $f(x) = f(y)$ if $x \sim y$.
\end{corollary}

\begin{proof}
For all open sets $G \subset X$, consider the sets $F^{-1}(G) = P_1(F \cap X \times G) \in \mathcal{A}$, and let $\mathcal{R}$ be the $\sigma$-algebra generated by $F^{-1}(G)$. Clearly $\mathcal{R} \subset \mathcal{A}$.

If $x \sim y$, then
\[
x \in F^{-1}(G) \quad \Longleftrightarrow \quad y \in F^{-1}(G),
\]
so that each equivalence class is contained in an atom of $\mathcal{R}$, and moreover by construction $x \mapsto F(x)$ is $\mathcal{R}$-measurable.

We thus conclude by using Theorem \ref{T:KRN} that there exists an $\mathcal{R}$-measurable section $f$: this measurability condition implies that $f$ is constant on atoms, in particular on equivalence classes.
\end{proof}

\subsection{Metric setting}
\label{ss:Metric}

In this section we refer to \cite{burago}.

\begin{definition}
\label{D:lengthstr}
A \emph{length structure} on a topological space $X$ is a class $\mathtt{A}$ of admissible paths, which is a subset of all continuous paths in X, together with a map $L: \mathtt{A} \to [0,+\infty]$: the map $L$ is called \emph{length of path}. The class $\mathtt{A}$ satisfies the following assumptions:
\begin{description}
\item[closure under restrictions] if $\gamma : [a,b] \to X$ is admissible and $a \leq c\leq d \leq b$, then $\gamma \llcorner_{[c,d]}$ is also admissible.
\item[closure under concatenations of paths] if $\gamma : [a,b] \to X$ is such that its restrictions $\gamma_1, \gamma_2$ to $[a,c]$ and $[c,b]$ are both admissible, then so is $\gamma$.
\item[closure under admissible reparametrizations] for an admissible path $\gamma : [a,b] \to X$ and a for $\f: [c,d]\to[a,b]$, $\f \in B$, with $B$ class of admissible homeomorphisms that includes the linear one, the composition $\gamma(\f(t))$ is also admissible.
\end{description}

The map $L$ satisfies the following properties:
\begin{description}
\item[additivity] $L(\gamma \llcorner_{[a,b]}) = L(\gamma \llcorner_{[a,c]}) + L(\gamma \llcorner_{[c,b]})$ for any $c\in [a,b]$.
\item[continuity] $L(\gamma \llcorner_{[a,t]})$ is a continuous function of $t$.
\item[invariance] The length is invariant under admissible reparametrizations.
\item[topology] Length structure agrees with the topology of $X$ in the following sense: for a neighborhood $U_x$ of a point $x \in X$, the length of paths connecting $x$ with points of the complement of $U_x$ is separated from zero:
\[
\inf \big\{ L(\gamma) : \gamma(a)=x, \gamma(b) \in X\setminus U_{x} \big\} >0.
\]
\end{description}
\end{definition}

Given a length structure, we can define a distance
\[
d_L(x,y) = \inf \Big\{ L(\gamma): \gamma:[a,b]\to X, \gamma \in \mathtt{A}, \gamma(a)=x, \gamma(b)=y \Big\},
\]
that makes $(X,d_{L})$ a metric space (allowing $d_{L}$ to be $+\infty$). The metric $d_{L}$ is called \emph{intrinsic}. 

It follows from Proposition 2.5.9 of \cite{burago} that every admissible curve of finite length admits a constant speed parametrization, i.e.
$\gamma$ defined on $[0,1]$ and $L(\gamma\llcorner[t,t'])= v (t'-t)$, with $v$ velocity.


\begin{definition}
A length structure is said to be \emph{complete} if for every two points $x,y$ there exists an admissible path joining them whose length $L(\gamma)$ is equal to $d_{L}(x,y)$.
\end{definition}

In other words, a length structure is complete if there exists a shortest path between two points.

Intrinsic metrics associated with complete length structure are said to be \emph{strictly intrinsic}. The metric space $(X,d_L)$ with $d_L$ strictly intrinsic is called a \emph{geodesic space}. A curve whose length equals the distance between its end points is called \emph{geodesic}.


\begin{definition}\label{D:strettaconv}
Let $(X,d_L)$ be a metric space. The distance $d_L$ is said to be \emph{strictly convex} if, for all $r \geq 0$, $d_L(x,y) = r/2$ implies that 
\[
\{ z : d_L(x,z) = r \} \cap \{ z : d_L(y,z) = r/2 \}
\]
is a singleton.
\end{definition}

The definition can be restated in geodesics spaces as: geodesics cannot bifurcate in the interior, i.e. \emph{the geodesic space $(X,d_L)$ is not branching}. 
An equivalent requirement is that if $\gamma_1 \not=\gamma_2$ and $\gamma_1(0) = \gamma_2(0)$, $\gamma_1(1) = \gamma_2(1)$, then $\gamma_1((0,1)) \cap \gamma_2((0,1)) = \emptyset$ and such geodesics do not admit a geodesic extension i.e. they are not a part of a longer geodesic.

From now on we assume the following: \label{P:assumpDL}

\medskip
\begin{enumerate}
\item $(X,d)$ Polish space;
\item $d_L : X \times X \to [0,+\infty]$ Borel distance;
\item $(X,d_L)$ is a non-branching geodesic space;
\item geodesics are continuous w.r.t. $d$;
\item  \label{Cond:XdL5}geodesics are locally compact in $(X,d)$: if $\gamma$ is a geodesic for $(X,d_L)$, then for each $x \in \gamma$ there exists $r$ such that $\gamma^{-1}(\bar B_r(x))$ is compact in $\R$.
\end{enumerate}
\medskip

Since we have two metric structures on $X$, we denote the quantities relating to $d_L$ with the subscript $L$: for example
\[
B_r(x) = \big\{ y : d(x,y) < r \big\}, \quad B_{r,L}(x) = \big\{ y : d_L(x,y) < r \big\}.
\]
In particular we will use the notation
\[
D_L(x) = \big\{ y : d_L(x,y) < + \infty \big\},
\]
$(\mathcal{K},d_H)$ for the compact sets of $(X,d)$ with the Hausdorff distance $d_H$ and $(\mathcal{K}_L,d_{H,L})$ for the compact sets of $(X,d_L)$ with the Hausdorff distance $d_{H,L}$. We recall that $(\mathcal{K},d_H)$ is Polish.

We write
\begin{equation}
\label{E:gammarre}
\gamma_{[x,y]} := \Big\{ \gamma \in \textrm{Lip}_{d_L}([0,1];X): \gamma(0) = x, \gamma(1)=y, L(\gamma) = d_L(x,y) \Big\}.
\end{equation}
With a slight abuse of notation, we will write
\begin{equation}
\label{E:opeclgeo}
\gamma_{(x,y)} = \bigcup_{\gamma \in \gamma_{[x,y]}} \gamma((0,1)), \quad \gamma_{[x,y]} = \bigcup_{\gamma \in \gamma_{[x,y]}} \gamma([0,1]).
\end{equation}
We will also use the following definition.

\begin{definition}
\label{D:geoconvx}
We say that $A \subset X$ is \emph{geodesically convex} if for all $x,y \in A$ the minimizing geodesic $\gamma_{[x,y]}$ between $x$ and $y$ is contained in $A$:
\[
\Big\{ \gamma((0,1)) : \gamma(0) = x, \gamma(1) = y, L(\gamma) = d(x,y), x,y \in A \Big\} \subset A.
\]
\end{definition}

\begin{lemma}
\label{L:measclos}
If $A$ is analytic in $(X,d)$, then $\{x : d_L(A,x) < \epsilon \}$ is analytic for all $\epsilon > 0$.
\end{lemma}

\begin{proof}
Observe that
\[
\big\{ x : d_L(A,x) < \epsilon \big\} = P_1 \Big( X \times A \cap \big\{ (x,y) :d_L(x,y) < \epsilon \big\} \Big),
\]
so that the conclusion follows from the invariance of the class $\Sigma^1_1$ w.r.t. projections.
\end{proof}

In particular, $\overline{A}^{_{d_L}}$, the closure of $A$ w.r.t. $d_L$, is analytic if $A$ is analytic.

\begin{remark}
\label{R:compact}

During the paper, whenever more regularity is required, we will assume also the following hypothesis:
\begin{itemize}
\item[(2')] $d_L : X \times X \to [0,+\infty]$ l.s.c. distance,
\item[(4')] $d_L(x,y) \geq d(x,y)$,
\item[(5')] \label{Point:globalcomp} $\cup_{x \in K_1, y \in K_2} \gamma_{[x,y]}$ is $d$-compact if $K_1$, $K_2$ are $d$-compact, $d_L \llcorner_{K_1 \times K_2}$ uniformly bounded.
\end{itemize}

A simple computation shows that $d_L(x,y) \geq d(x,y)$ implies the following 
\begin{enumerate}
\item $d_L$-compact sets are $d$-compact;
\item $d$-Lipschitz functions  are $d_{L}$-Lipschitz with the same constant.
\end{enumerate}
\end{remark}

An application of Theorem \ref{T:vanneuma}, in the setting of Remark \ref{R:compact}, gives a Borel function which selects a single geodesic $\gamma\in \gamma_{[x,y]}$ 
for any couple $(x,y)$.

\begin{lemma}\label{L:geodetiche}
Assume that $d_{L}$ is l.s.c.. Then
there exists a Borel function $\varUpsilon: X \times X \to \textrm{\rm Lip}_{d}([0,1],X)$ such that
up to reparametrization $\varUpsilon(x,y)\in \gamma_{[x,y]}$.
\end{lemma}
\begin{proof}
Let 
\begin{equation*}
\begin{array}{ccccc}
F& : & X\times X& \to & \textrm{Lip}_{d}([0,1],X) \\
&& (x,y)& \mapsto & \gamma_{[x,y]}
\end{array}
\end{equation*}
with $\textrm{Lip}_{d}([0,1],X)$ endowed with the uniform topology and $\gamma_{[x,y]}$ defined in \eqref{E:gammarre}.

The result follows by Theorem \ref{T:vanneuma} observing that $\gr(F)$ is the set   
\[ 
\Big\{ (x,y,\gamma)\in X\times X \times \textrm{Lip}_{d}([0,1],X),   L(\gamma)= d_{L}(x,y) \Big\}. 
\]
which is Borel by the l.s.c. of the map $\gamma \mapsto L(\gamma)$, and this is implied by the l.s.c. of $d_{L}$. 
\end{proof}

\subsection{General facts about optimal transportation}
\label{Ss:General Facts}

Let $(X,\mathcal B,\mu)$ and $(Y,\mathcal B,\nu)$ be two Polish probability spaces and  $c : X \times Y \to \R$ be a Borel measurable function. Consider the set of \emph{transference plans}
\[
\Pi (\mu,\nu) := \Big\{ \pi \in \mathcal{P}(X\times Y) : (P_1)_\sharp \pi = \mu, (P_2)_\sharp \pi = \nu \Big\}.
\]
Define the functional
\begin{equation}
\label{E:Ifunct}
\begin{array}{ccccl} 
\mathcal{I} &:& \Pi(\mu,\nu) &\to& \erre^{+} \cr
&& \pi &\mapsto& \mathcal{I}(\pi):=\int c \pi.
\end{array}
\end{equation}
The \emph{Monge-Kantorovich minimization problem} is to find the minimum of $\mathcal{I}$ over all transference plans.

If we consider a $\mu$-measurable \emph{transport map} $T : X \to Y$ such that $T_{\sharp}\mu=\nu$, the functional \eqref{E:Ifunct} becomes
\[
\mathcal I(T):= \mathcal I \big( (Id \times T)_\sharp \mu \big) = \int c(x,T(x)) \mu(dx).
\]
The minimum problem over all $T$ is called \emph{Monge minimization problem}.

The Kantorovich problem admits a (pre) dual formulation.

\begin{definition}
A map $\varphi : X \to \erre \cup \{-\infty\} $ is said to be \emph{$c$-concave} if it is not identically $-\infty$ and there exists $\psi : Y \to \erre \cup \{-\infty\}$, $\psi \not\equiv -\infty$, such that
\[
\varphi(x) = \inf_{y \in Y} \big\{ c(x,y) - \psi(y) \big\}.
\]
The \emph{$c$-transform} of $\varphi$ is the function
\begin{equation}
\label{E:ctransf}
\varphi^c(y) := \inf_{x\in X}  \left\{ c(x,y) - \varphi (x) \right\}.
\end{equation}
The \emph{$c$-superdifferential $\partial^c \f$} of $\varphi$ is the subset of $X \times Y$ defined by
\begin{equation}
\label{E:csudiff}
\partial^{c}\f := \Big\{ (x,y) : c(x,y) - \f(x) \leq c(z,y) - \f(z) \ \forall z \in X \Big\} \subset X \times Y.
\end{equation}
\end{definition}

\begin{definition}\label{D:cicl}
A set $\Gamma \subset X \times Y$ is said to be \emph{$c$-cyclically monotone} if, for any $n \in \mathbb{N}$ and for any family $(x_0,y_0),\dots,(x_n,y_n)$ of points of $\Gamma$, the following inequality holds:
\[
\sum_{i=0}^nc(x_i,y_i) \leq \sum_{i=0}^nc(x_{i+1},y_i),
\]
where $x_{n+1} = x_0$.

A transference plan is said to be \emph{$c$-cyclically monotone} if it is concentrated on a $c$-cyclically monotone set.
\end{definition}

Consider the set
\begin{equation}
\label{E:Phicset}
\Phi_c := \Big\{ (\varphi,\psi) \in L^1(\mu) \times L^1(\nu): \varphi(x) + \psi(y) \leq c(x,y) \Big\}.
\end{equation}
Define for all $(\varphi,\psi)\in \Phi_c$ the functional
\begin{equation}
\label{E:Jfunct}
J(\varphi,\psi) := \int \varphi \mu + \int \psi \nu.
\end{equation}

The following is a well known result (see Theorem 5.10 of \cite{villa:Oldnew}).

\begin{theorem}[Kantorovich Duality]
\label{T:kanto}
Let X and Y be Polish spaces, let $\mu \in \mathcal{P}(X)$ and $\nu \in \mathcal{P}(Y)$, and let $c : X \times Y \to [0,+\infty]$ be lower semicontinuous. Then the following holds:
\begin{enumerate}
\item Kantorovich duality:
\[
\inf_{\pi \in \Pi(\mu,\nu)} \mathcal{I} (\pi) = \sup _{(\varphi,\psi)\in \Phi_{c}} J(\varphi,\psi).
\]
Moreover, the infimum on the left-hand side is attained and the right-hand side is also equal to
\[
\sup _{(\varphi,\psi)\in \Phi_{c}\cap C_{b}} J(\varphi,\psi),
\]
where $C_{b}= C_b(X, \erre) \times C_b(Y,\erre)$.
\item If $c$ is real valued and the optimal cost is finite, then there is a measurable $c$-cyclically monotone set $\Gamma \subset X\times Y$, closed if $c$ is continuous, such that for any $\pi \in \Pi(\mu,\nu)$ the following statements are equivalent:
\begin{enumerate}
\item $\pi$ is optimal;
\item $\pi$ is $c$-cyclically monotone;
\item $\pi$ is concentrated on $\Gamma$;
\item there exists a $c$-concave function $\f$ such that $\pi$-a.s. $\f(x)+\f^{c}(y)=c(x,y)$.
\end{enumerate}

\item If moreover
\[
c(x,y) \leq c_{X}(x) + c_{Y}(y), \quad \ c_{X}\ \mu\textrm{-integrable}, \ c_{Y}\ \nu\textrm{-integrable},
\]
then the supremum is attained:
\[
\sup_{\Phi_c} J = J(\f, \f^{c}) = \inf_{\pi \in \Pi(\mu,\nu)} \mathcal{I}(\pi).
\]
\end{enumerate}
\end{theorem}

We recall also that if $-c$ is Souslin, then every optimal transference plan $\pi$ is concentrated on a $c$-cyclically monotone set \cite{biacar:cmono}.

\section{Optimal transportation in geodesic spaces}
\label{S:Optimal}

Let $\mu, \nu \in \mathcal{P}(X)$ and consider the transportation problem with cost $c(x,y)= d_L(x,y)$, and let $\pi \in \Pi(\mu,\nu)$ be a $d_L$-cyclically monotone transference plan with finite cost. By inner regularity, we can assume that the optimal transference plan is concentrated on a $\sigma$-compact $d_L$-cyclically monotone set $\Gamma \subset \{d_L(x,y) < +\infty\}$. By Lusin Theorem, we can require also that $d_L \llcorner_{\Gamma}$ is $\sigma$-continuous:
\begin{equation}\label{E:gamman}
\Gamma = \cup_n \Gamma_n, \ \Gamma_n \subset \Gamma_{n+1} \ \text{compact}, \quad d_L \llcorner_{\Gamma_n} \ \text{continuous.}
\end{equation}

In this section, using only the $d_L$-cyclical monotonicity of $\Gamma$,  we obtain a partial order relation $G \subset X \times X$.
The set $G$ is analytic, and allows to define the transport ray set $R$, the transport sets $\mathcal T_e$, $\mathcal T$,
and the set of initial points $a$ and final points $b$.
Moreover we show that $R \llcorner_{\mathcal T \times \mathcal T}$ is an equivalence relation and that we can assume the set of final points $b$ to be $\mu$-negligible.

Consider the set
\begin{align}
\label{E:gGamma}
\Gamma' :=&~ \bigg\{ (x,y) : \exists I \in \enne_0, (w_i,z_i) \in \Gamma \ \text{for} \ i = 0,\dots,I, \ z_I = y \crcr
&~ \qquad \qquad w_{I+1} = w_0 = x, \ \sum_{i=0}^I d_L(w_{i+1},z_i) - d_L(w_i,z_i) = 0 \bigg\}.
\end{align}
In other words, we concatenate points $(x,z), (w,y) \in \Gamma$ if they are initial and final point of a cycle with total cost $0$.

\begin{lemma}
\label{L:gGamma}
The following holds:
\begin{enumerate}
\item $\Gamma \subset \Gamma' \subset \{d_L(x,y) < +\infty\}$;
\item if $\Gamma$ is analytic, so is $\Gamma'$;
\item if $\Gamma$ is $d_L$-cyclically monotone, so is $\Gamma'$.
\end{enumerate}
\end{lemma}

\begin{proof}
For the first point, set $I=0$ and $(w_{n,0},z_{n,0}) = (x,y)$ for the first inclusion. If $d_L(x,y) = +\infty$, then $(x,y) \notin \Gamma$ and all finite set of points in $\Gamma$ are bounded.

For the second point, observe that
\begin{align*}
\Gamma' =&~ \bigcup_{I \in \enne_0} P_{12} (A_I) \crcr
=&~ \bigcup_{I \in \enne_0} P_{12} \bigg( \prod_{i=0}^I \Gamma \cap \bigg\{ \prod_{i=1}^I (w_i,z_i) : \sum_{i=0}^I d_L(w_{i+1},z_i) - d_L(w_i,z_i) = 0, w_{I+1} = w_0 \bigg\} \bigg).
\end{align*}
For each $I \in \N_0$, since $d_L$ is Borel, it follows that
\begin{align*}
\bigg\{ \prod_{i=1}^I (w_i,z_i) : \sum_{i=0}^I d_L(w_{i+1},z_i) - d_L(w_i,z_i) = 0, w_{I+1} = w_0 \bigg\}
\end{align*}
is Borel in $\prod_{i=0}^I (X \times X)$, so that for $\Gamma$ analytic each set $A_{n,I}$ is analytic. Hence $P_{12}(A_I)$ is analytic, and since the class $\Sigma^1_1$ is closed under countable unions and intersections it follows that $\Gamma'$ is analytic.

For the third point, observe that for all $(x_j,y_j) \in \Gamma'$, $j=0,\dots,J$, there are $(w_{j,i},z_{j,i}) \in \Gamma$, $i = 0,\dots,I_j$, such that
\begin{align*}
d_L(x_j,y_j) + \sum_{i=0}^{I_j-1} d_L(w_{j,i+1},z_{j,i}) - \sum_{i=0}^{I_j} d_L(w_{j,i},z_{j,i}) = 0.
\end{align*}
Hence we can write for $x_{J+1} = x_0$, $w_{j,I_j+1} = w_{j+1,0}$, $w_{J+1,0} = w_{0,0}$
\begin{align*}
\sum_{j=0}^J d_L(x_{j+1},y_j) - d_L(x_j,y_j) =&~ \sum_{j=0}^J \sum_{i=0}^{I_j} d_L(w_{j,i+1},z_{j,i}) - d_L(w_{j,i},z_{j,i}) \geq 0,
\end{align*}
using the $d_L$-cyclical monotonicity of $\Gamma$.
\end{proof}

\begin{definition}[Transport rays]
\label{D:Gray}
Define the \emph{set of oriented transport rays}
\begin{equation}
\label{E:trG}
G := \Big\{ (x,y): \exists (w,z) \in \Gamma', d_L(w,x) + d_L(x,y) + d_L(y,z) = d_L(w,z) \Big\}.
\end{equation}

For $x \in X$, the \emph{outgoing transport rays from $x$} is the set $G(x)$ and the \emph{incoming transport rays in $x$} is the set $G^{-1}(x)$. Define the \emph{set of transport rays} as the set
\begin{equation}
\label{E:Rray}
R := G \cup G^{-1}.
\end{equation}
\end{definition}

\begin{lemma}
\label{L:analGR}
The following holds:
\begin{enumerate}
\item $G$ is $d_L$-cyclically monotone;
\item $\Gamma' \subset G \subset \{d_L(x,y) < +\infty\}$;
\item the sets $G$, $R := G \cup G^{-1}$ are analytic.
\end{enumerate}
\end{lemma}

\begin{proof}
The second point follows by the definition: if $(x,y) \in \Gamma'$, just take $(w,z) = (x,y)$ in the r.h.s. of \eqref{E:trG}.

The third point is consequence of the fact that
\[
G = P_{34} \Big( \big( \Gamma' \times X \times X \big) \cap \Big\{ (w,z,x,y) : d_L(w,x) + d_L(x,y) + d_L(y,z) = d_L(w,z) \Big\} \Big),
\]
and the result follows from the properties of analytic sets.

The first point follows from the following observation: if $(x_i,y_i) \in \gamma_{[w_i,z_i]}$, then from triangle inequality
\begin{align*}
d_L(x_{i+1},y_i) - d_L(x_i,y_i) + d_L(x_i,y_{i-1}) \geq&~ d_L(x_{i+1},z_i) - d_L(z_i,y_i) - d_L(x_i,y_i) + d_L(x_i,y_{i-1}) \crcr
=&~ d_L(x_{i+1},z_i) - d_L(x_i,z_i) + d_L(x_i,y_{i-1}) \crcr
\geq&~  d_L(x_{i+1},z_i) - d_L(x_i,z_i) + d_L(w_i,y_{i-1}) - d_L(w_i,x_i) \crcr
=&~  d_L(x_{i+1},z_i) - d_L(w_i,z_i) + d_L(w_i,y_{i-1}).
\end{align*}
Repeating the above inequality finitely many times one obtain
\[
\sum_i d_L(x_{i+1},y_i) - d_L(x_i,y_i) \geq \sum_i d_L(w_{i+1},z_i) - d_L(w_i,z_i) \geq 0.
\]
Hence the set $G$ is $d_L$-cyclically monotone.
\end{proof}

\begin{definition} Define the \emph{transport sets}
\begin{subequations}
\label{E:TR0}
\begin{align}
\label{E:TR}
\mathcal T :=&~ P_1 \big( \textrm{graph}(G^{-1}) \setminus \{x = y\} \big) \cap P_1 \big( \textrm{graph}(G) \setminus \{x = y\} \big), \\
\label{E:TRe}
\mathcal T_e :=&~ P_1 \big( \textrm{graph}(G^{-1}) \setminus \{x = y\} \big) \cup P_1 \big( \textrm{graph}(G) \setminus \{x = y\} \big).
\end{align}
\end{subequations}
\end{definition}

From the definition of $G$ it is fairly easy to prove that $\mathcal{T}$, $\mathcal{T}_e$ are analytic sets. The subscript $e$ refers to the endpoints of the geodesics: clearly we have
\begin{equation}
\label{E:RTedef}
\mathcal{T}_e = P_1(R \setminus \{x = y\}).
\end{equation}

The following lemma shows that we have only to study the Monge problem in $\mathcal{T}_e$.

\begin{lemma}
\label{L:mapoutside}
It holds $\pi(\mathcal{T}_e \times \mathcal{T}_e \cup \{x = y\}) = 1$.
\end{lemma}

\begin{proof}
If $x \in P_1(\Gamma \setminus \{x=y\})$, then $x \in G^{-1}(y) \setminus \{y\}$ for some $y\in X$. Similarly, $y \in P_2(\Gamma \setminus \{x=y\})$ implies that $y \in G(x) \setminus \{x\}$
for some $x\in X$. Hence $\Gamma \setminus \mathcal{T}_e \times \mathcal{T}_e \subset \{x = y\}$.
\end{proof}

As a consequence, $\mu(\mathcal{T}_e) = \nu(\mathcal{T}_e)$ and any maps $T$ such that for $\nu \llcorner_{\mathcal{T}_e} = T_\sharp \mu \llcorner_{\mathcal{T}_e}$ can be extended to a map $T'$ such that $\nu = T_\sharp \mu$ with the same cost by setting
\begin{equation}
\label{E:extere}
T'(x) =
\begin{cases}
T(x) & x \in \mathcal{T}_e \crcr
x & x \notin \mathcal{T}_e
\end{cases}
\end{equation}

We now use the non branching assumption.

\begin{lemma}
\label{L:uniqr}
If $x \in \mathcal{T}$, then $R(x)$ is a single geodesic.
\end{lemma}

\begin{proof}
Since $x \in \mathcal{T}$, there exists $(w,x), (x,z) \in G \setminus \{x=y\}$: from the $d_L$-cyclical monotonicity and triangular inequality, it follows that
\[
d_L(w,z) = d_L(w,x) + d_L(x,z),
\]
so that $(w,z) \in G$ and $x \in \gamma_{(w,z)}$. Hence from the non branching assumption the set
\[
R(x) = \bigcup_{y \in G(x)} \gamma_{[x,y]} \cup \bigcup_{z \in G^{-1}(x)} \gamma_{[z,x]}
\]
is a single geodesic.
\end{proof}

\begin{proposition}
\label{P:equiv}
The set $R \cap \mathcal{T} \times \mathcal{T}$ is an equivalence relation on $\mathcal{T}$. The set $G$ is a partial order relation on $\mathcal{T}_e$.
\end{proposition}

\begin{proof}
Using the definition of $R$, one has in $\mathcal{T}$:
\begin{enumerate}
\item $x \in \mathcal{T}$ implies that
\[
\exists y \in G(x) \setminus \{x = y\},
\]
so that from the definition of $G$ it follows $(x,x) \in G$;
\item if $y \in R(x)$, $x,y \in \mathcal{T}$, then from Lemma \ref{L:uniqr} there exists $(w,z) \in G$ such that $x , y \in \gamma_{(w,z)}$. Hence $x \in R(y)$;
\item if $y \in R(x)$, $z \in R(y)$, $x,y,z \in \mathcal{T}$, then from Lemma \ref{L:uniqr} it follows again there exists $(w,z) \in G$ such that $x , y , z \in \gamma_{(w,z)}$. Hence $z \in R(x)$.
\end{enumerate}

The second part follows similarly:
\begin{enumerate}
\item $x \in \mathcal{T}_e$ implies that
\[
\exists (x,y) \in \big( G \setminus \{x=y\} \big) \cup \big( G^{-1} \setminus \{x=y\} \big),
\]
so that in both cases $(x,x) \in G$;
\item as in Lemma \ref{L:uniqr}, $(x,y), (y,z) \in G \setminus \{x=y\}$ implies by $d_L$-cyclical monotonicity that $(x,z) \in G$.
\end{enumerate}
\end{proof}

\begin{remark}
\label{R:orderX}
Note that $G \cup \{x=y\}$ is a partial order relation on $X$.
\end{remark}

\begin{definition}
\label{D:endpoint}
Define the multivalued \emph{endpoint graphs} by:
\begin{subequations}
\label{E:endpoint0}
\begin{align}
\label{E:endpointa}
a :=&~ \big\{ (x,y) \in G^{-1}: G^{-1}(y) \setminus \{y\} = \emptyset \big\}, \\
\label{E:endpointb}
b :=&~ \big\{ (x,y) \in G: G(y) \setminus \{y\} = \emptyset \big\}.
\end{align}
\end{subequations}
We call $P_2(a)$ the set of \emph{initial points} and $P_2(b)$ the set of \emph{final points}.
\end{definition}

Even if $a$, $b$ are not in the analytic class, still they belong to the $\sigma$-algebra $\mathcal{A}$.

\begin{figure}
\label{Fi:mongemetri3}
\psfrag{Gamma}{\color{MidnightBlue}$\Gamma$}
\psfrag{Gamma'}{\color{Brown}$\Gamma'$}
\psfrag{b}{\color{Purple}$b$}
\psfrag{a}{\color{Goldenrod}$a$}
\psfrag{G}{\color{Cyan}$G$}
\psfrag{G-1}{\color{CarnationPink}$G^{-1}$}
\psfrag{mu}{\color{Red}$\mu$}
\psfrag{nu}{\color{Green}$\nu$}
\psfrag{T}{\color{Cyan}$\mathcal{T}$}
\centerline{\resizebox{12cm}{12cm}{\includegraphics{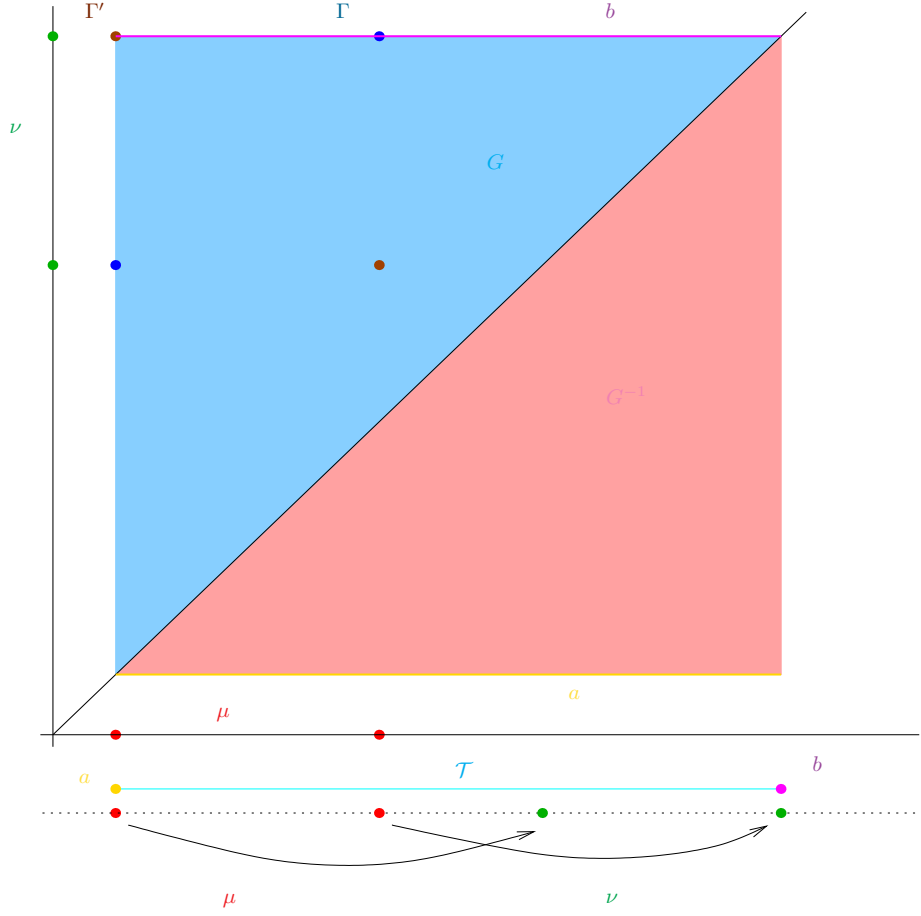}}}
\caption{Construction of the sets $\Gamma$, $\Gamma'$, $G$, $G^{-1}$, $a$, $b$ in a 1-dimensional example with $d_{L}=|\cdot|$.}
\end{figure}

\begin{proposition}
The following holds:
\begin{enumerate}
\item the sets
\[
a,b \subset X \times X, \quad a(A), b(A) \subset X,
\]
belong to the $\mathcal{A}$-class if $A$ analytic;
\item $a \cap b \cap \mathcal{T}_e \times X = \emptyset$;
\item $a(x)$, $b(x)$ are singleton or empty when $x \in \mathcal{T}$;
\item $a(\mathcal{T}) = a(\mathcal{T}_e)$, $b(\mathcal{T}) = b(\mathcal{T}_e)$;
\item $\mathcal{T}_e = \mathcal{T} \cup a(\mathcal{T}) \cup b(\mathcal{T})$, $\mathcal{T} \cap (a(\mathcal{T}) \cup b(\mathcal{T})) = \emptyset$.
\end{enumerate}
\end{proposition}

\begin{proof}
Define
\[
C := \Big\{ (x,y,z) \in \mathcal{T}_e \times \mathcal{T}_e \times \mathcal{T}_e: y \in G(x), z \in G(y) \Big\} = (G \times X) \cap (X \times G) \cap \mathcal{T}_e \times \mathcal{T}_e \times \mathcal{T}_e,
\]
that is clearly analytic. Then
\[
b = \Big\{ (x,y) \in G : y \in G(x), G(y) \setminus \{y\} = \emptyset \Big\} = G \setminus P_{1,2}(C \setminus X \times \{y=z\}),
\]
\[
b(A) = \Big\{ y: y \in G(x), G(y) \setminus \{y\} = \emptyset, x \in A \Big\} = P_2(G \cap A \times X) \setminus P_2(C \setminus X \times \{y=z\}).
\]
A similar computation holds for $a$:
\[
a = G^{-1} \setminus P_{23}(C \setminus \{x=y\} \times X), \quad a(A) = P_1(G \cap X \times A) \setminus P_1(C \setminus \{x=y\} \times X)
\]
Hence $a,b \in \mathcal{A}(X \times X)$, $a(A), b(A) \in \mathcal{A}(X)$, being the intersection of an analytic set with a coanalytic one.

If $x \in \mathcal{T}$, then from Lemma \ref{L:uniqr} it follows that $a(x)$, $b(x)$ are empty or singletons and $a(x) \not= b(x)$. If $x \in \mathcal{T}_e \setminus \mathcal{T}$, then it follows that the geodesic $\gamma_{[w,z]}$, $(w,z) \in G$, to which $x$ belongs cannot be prolonged in at least one direction: hence $x \in a(x) \cup b(x)$.

The other point follows easily.
\end{proof}

We finally show that we can assume that the $\mu$-measure of final points and the $\nu$-measure of the initial points are $0$.

\begin{lemma}
\label{L:finini0}
The sets $G \cap b(\mathcal{T}) \times X$, $G \cap X \times a(\mathcal{T})$ is a subset of the graph of the identity map.
\end{lemma}

\begin{proof}
From the definition of $b$ one has that
\[
x \in b(\mathcal{T}) \quad \Longrightarrow \quad G(x) \setminus \{x\} = \emptyset,
\]
A similar computation holds for $a$.
\end{proof}

Hence we conclude that
\[
\pi (b(\mathcal{T}) \times X) = \pi(G \cap b(\mathcal{T}) \times X) = \pi(\{x = y\}),
\]
and following \eqref{E:extere} we can assume that
\[
\mu(b(\mathcal{T})) = \nu(a(\mathcal{T})) = 0.
\]

\begin{remark}
\label{R:lsccase1}

In the case considered in Remark \ref{R:compact}, it is possible to obtain more regularity for the sets introduced so far.
Recall that we are now assuming
\begin{itemize}
\item[(2')] $d_L : X \times X \to [0,+\infty]$ l.s.c. distance,
\item[(4')] $d_L(x,y) \geq d(x,y)$,
\item[(5')] $\cup_{x \in K_1, y \in K_2} \gamma_{[x,y]}$ is $d$-compact if $K_1$, $K_2$ are $d$-compact, $d_L \llcorner_{K_1 \times K_2}$ uniformly bounded.
\end{itemize}

The set $\Gamma'$ is $\sigma$-compact: in fact, if one restrict to each $\Gamma_n$
given by \eqref{E:gamman}, then the set of cycles of order $I$ is compact, and thus
\begin{align*}
\Gamma'_{n,\bar I} :=&~ \bigg\{ (x,y) : \exists I \in \{0,\dots,\bar I\}, (w_i,z_i) \in \Gamma_n \ \text{for} \ i = 0,\dots,I, \ z_I = y \crcr
&~ \qquad \qquad w_{I+1} = w_0 = x, \ \sum_{i=0}^I d_L(w_{i+1},z_i) - d_L(w_i,z_i) = 0 \bigg\}
\end{align*}
is compact. Finally $\Gamma' = \cup_{n,I} \Gamma'_{n,I}$.

Moreover, $d_L \llcorner_{\Gamma'_{n,I}}$ is continuous. If $(x_n,y_n) \to (x,y)$, then from the l.s.c. and
\[
\sum_{i=0}^I d_L(w_{n,i+1},z_{n,i}) = \sum_{i=0}^I d_L(w_{n,i},z_{n,i}), \quad w_{n,I+1} = w_{n,0} = x_{n}, \ z_{n,I} = y_{n},
\]
it follows also that each $d_L(w_{n,i+1},z_{n,i})$ is continuous.

Similarly the sets $G$, $R$, $a$, $b$ are $\sigma$-compact: assumption (5') and the above computation in fact shows that
\[
G_{n,I} := \Big\{ (x,y): \exists (w,z) \in \Gamma'_{n,I}, d_L(w,x) + d_L(x,y) + d_L(y,z) = d_L(w,z) \Big\}
\]
is compact. For $a$, $b$, one uses the fact that projection of $\sigma$-compact sets is $\sigma$-compact.

So if we are in the case of Remark \ref{R:compact}, $\Gamma$, $\Gamma'$, $G$, $G^{-1}$, $a$ and $b$ are $\sigma$-compact sets.
\end{remark}

\begin{remark}
\label{R:TR}
Many simplifications occur in the case the disintegration w.r.t. the partition $\{D_L(x)\}_{x \in X}$ is strongly consistent. Recall that 
$D_L(x) = \big\{ y : d_L(x,y) < + \infty \big\}$. Let
\[
\pi = \int_0^1 \pi_\alpha m(d\alpha), \quad \mu = \int_0^1 \mu_\alpha m(d\alpha), \ \nu = \int_0^1 \nu_\alpha m(d\alpha)
\]
be strongly consistent disintegrations such that
\[
\mu_\alpha(D_L(x_\alpha)) = \nu_\alpha(D_L(x_\alpha)) = 1, \quad \pi_\alpha \in \Pi(\mu_\alpha,\nu_\alpha).
\]
We have used the fact that the partition $\{D_L(x) \times D_L(x)\}_{x \in X}$ has the crosswise structure, and then we can apply the results of \cite{biacar:cmono}.

\medskip

\noindent {\it 1) Optimality of $\pi_\alpha$.} Since $\pi$ is $d_L$-cyclically monotone, also the $\pi_\alpha$ are $d_L$-cyclically monotone: precisely they are concentrated on the sets
\[
\Gamma_\alpha = \Gamma \cap D_L(x_\alpha) \times D_L(x_\alpha),
\]
if $\Gamma$ is $d_L$-cyclically monotone and $\pi(\Gamma) = 1$.

Using the fact that $(D_L(x_\alpha),d_L)$ is a metric space, then we can construct a potential $\varphi(x,x_\alpha)$ using the formula
\[
\varphi(x,x_\alpha) = \inf \bigg\{ \sum_{i=0}^I d_L(x_{i+1},y_i) - d_L(x_i,y_i), (x_i,y_i) \in \Gamma_\alpha, x_{I+1} = x, (x_0,y_0) = (x_\alpha,x_\alpha) \bigg\}.
\]
and since this is bounded on $(D_L(x_\alpha),d_L)$, we see that $\pi_\alpha$ and hence $\pi$ are optimal.

\medskip

\noindent {\it 2. Potential for $\pi$.} Extend $\varphi(x,x_\alpha)$ to $X$ by setting $\varphi(x,x_\alpha) = +\infty$ if $x \notin D_L(x_\alpha)$. If $\{(x_\alpha,x_\alpha)\}_{\alpha \in [0,1]}$ is a Borel section, then the function
\[
\varphi(x) = \inf_\alpha\{\varphi(x,\alpha)\}
\]
is easily seen to be analytic. This function is clearly a potential for $\pi$. In particular, it follows again from \cite{biacar:cmono} that $\pi$ is optimal if it is $d_L$-cyclically monotone.

\medskip

\noindent {\it 3.Transport set.} We can then define the set of oriented transport rays as the set
\[
G = \Big\{ (x,y) \in X\times X: \varphi (x) - \varphi(y) = d_L(x,y) \Big\}.
\]
In general, this sets is larger than the one of definition \ref{D:Gray}.
%
%
\end{remark}

\section{Partition of the transport set $\mathcal{T}$}
\label{S:partition}

In this section we use the continuity and local compactness of geodesics to show that the disintegration induced by $R$ on $\mathcal T$ is strongly consistent. 
Using this fact, we can define an order preserving map $g$ which maps our transport problem into a transport problem on $\mathcal S \times \R$, where $\mathcal S$ is a cross section of $R$.


Let $\{x_i\}_{i \in \N}$ be a dense sequence in $(X,d)$.

\begin{lemma}
\label{L:reguloclco}
The sets
\[
W_{ijk} := \Big\{ x \in \mathcal{T} \cap \bar B_{2^{-j}}(x_i): L(G(x)),L(G^{-1}(x)) \geq 2^{2-k}, L \big( R(x) \cap \bar B_{2^{1-j}}(x_i) \big) \leq 2^{-k} \Big\}
\]
form a countable covering of $\mathcal{T}$ of class $\mathcal{A}$.
\end{lemma}

\begin{proof}
We first prove the measurability. We consider separately the conditions defining $W_{ijk}$.

{\it Point 1.} The set
\[
A_{ij} := \mathcal{T} \cap \bar B_{2^{-j}}(x_i)
\]
is clearly analytic.

{\it Point 2.} The set
\[
B_k := \big\{ x \in \mathcal{T}: L(G(x)) \geq 2^{2-k} \big\} = P_1 \Big( G \cap \big\{ d_L(x,y) \geq 2^{2-k} \big\} \Big)
\]
is again analytic, being the projection of an analytic set. Similarly, the set
\[
C_k := \big\{ x \in \mathcal{T}: L(G^{-1}(x)) \geq 2^{2-k} \big\} = P_1 \Big( G^{-1} \cap \big\{ d_L(x,y) \geq 2^{2-k} \big\} \Big)
\]
is again analytic.

{\it Point 3.} The set
\begin{align*}
D_{jk} :=&~ \big\{ x \in \mathcal{T}: L \big( R(x) \cap \bar B_{2^{-j}}(x_i) \big) \leq 2^{-k} \big\} \crcr
=&~ \mathcal{T} \setminus P_1 \Big( R \cap \big( \{(x,y): d(x_i,y) \leq 2^{1-j}\} \cap \{d_L(x,y) > 2^{-k}\} \big) \Big)
\end{align*}
is in the $\mathcal{A}$-class, being the difference of two analytic sets.

We finally can write
\begin{align*}
W_{ijk} = A_{ij} \cap B_k \cap C_k \cap D_{jk},
\end{align*}
and the fact that $\mathcal{A}$ is a $\sigma$-algebra proves that $W_{ijk} \in \mathcal{A}$.

To show that it is a covering, notice that for all $x \in \mathcal{T}$ it holds
\[
\min \big\{ L(G(x)), L(G^{-1}(x)) \big\} \geq 2^{2-\bar k}
\]
for some $\bar k \in \N$.

From the local compactness of geodesics, Assumption \eqref{Cond:XdL5} of page \pageref{P:assumpDL}, it follows that if $\gamma^{-1}(\bar B_r(x))$ is compact, then the continuity of $\gamma$ implies that $\gamma^{-1}(\bar B_{r'}(x))$ is also compact for all $r' \leq r$, and $\diam_{d_L}(\gamma \cap \bar B_{r'}(x)) \to 0$ and $r' \to 0$. In particular there exists $\bar j \in \N$ such that
\[
L \big( R(x) \cap \bar B_{2^{1-\bar j}}(x) \big) \leq 2^{-\bar k},
\]
with $\bar k$ the one chosen above.

Finally, one choose $x_{\bar i}$ such that $d(x,x_{\bar i}) < 2^{-1-\bar j}$, so that $x \in \bar B_{2^{-\bar j}}(x_{\bar i}) \subset \bar B_{2^{1-\bar j}}(x)$ and thus
\[
L \big( R(x) \cap \bar B_{2^{-\bar j}}(x_{\bar i}) \big) \leq 2^{-\bar k}.
\]
\end{proof}

\begin{lemma}
\label{L:partitilW}
There exist $\mu$-negligible sets $N_{ijk} \subset W_{ijk}$ such that the family of sets
\begin{align*}
\mathcal{T}_{ijk} = R^{-1}(W_{ijk} \setminus N_{ijk})
\end{align*}
is a countable covering of $\mathcal{T} \setminus \cup_{ijk} N_{ijk}$ into saturated analytic sets.
\end{lemma}

\begin{proof}
First of all, since $W_{ijk} \in \mathcal{A}$, then there exists $\mu$-negligible set $N_{ijk} \subset W_{ijk}$ such that $W_{ijk} \setminus N_{ijk} \in \mathcal{B}(X)$. Hence $\{W_{ijk} \setminus N_{ijk}\}_{i,j,k \in \N}$ is a countable covering of $\mathcal{T} \setminus \cup_{ijk} N_{ijk}$. It follows immediately that $\{\mathcal{T}_{ijk}\}_{i,j,k \in \N}$ satisfies the lemma.
\end{proof}

\begin{remark}
\label{R:compFF}
Observe that $\bar B_{2^{-j}}(x_i) \cap R(x)$ is compact for all $x \in \mathcal{T}_{ijk}$: in fact, during the proof of Lemma \ref{L:reguloclco} we have already shown that $\gamma^{-1}(\bar B_{2^{-j}} (x_i))$ is compact.
%
\end{remark}

From any analytic countable covering, we can find a countable partition into $\mathcal{A}$-class saturated sets by defining
\begin{equation}
\label{E:Zkije}
\mathcal{Z}_{m,e} := \mathcal{T}_{i_mj_mk_m} \setminus \bigcup_{m' = 1}^{m-1} \mathcal{T}_{i_{m'}j_{m'}k_{m'}}, \quad \mathcal{Z}_{0,e} := \mathcal{T}_e \setminus \bigcup_{m \in \N} \mathcal{Z}_{m,e},
\end{equation}
where
\[
\N \ni m \mapsto (i_m,j_m,k_m) \in \N^3
\]
is a bijective map. Intersecting the above sets with $\mathcal{T}$, we obtain the countable partition of $\mathcal{T}$ in $\mathcal{A}$-sets
\begin{equation}
\label{E:Zkij}
\mathcal{Z}_m := \mathcal{Z}_{m,e} \cap \mathcal{T}, \quad m \in \N_0.
\end{equation}
Now we use this partition to prove the strong consistency of the disintegration.

On $\mathcal{Z}_m$, $m > 0$, we define the closed values map
\begin{equation}
\label{E:mapTijkF}
\mathcal{Z}_m \ni x \mapsto F(x) := R(x) \cap \bar B_{2^{-j_m}}(x_{i_m}) \in \mathcal{K} \big( \bar B_{2^{-{j_m}}}(x_{i_m}) \big),
\end{equation}
where $\mathcal{K}(\bar B_{2^{-{j_m}}}(x_{i_m}))$ is the space of compact subsets of $\bar B_{2^{-{j_m}}}(x_{i_m})$. 

\begin{proposition}
\label{P:sicogrF}
There exists a $\mu$-measurable cross section $f : \mathcal{T} \to \mathcal{T}$ for the equivalence relation $R$. 
\end{proposition}

\begin{proof}
First we show that $F$ is $\mathcal{A}$-measurable: for $\delta > 0$,
\begin{align*}
F^{-1}(B_\delta(y)) =&~ \Big\{ x \in \mathcal{Z}_m: R(x) \cap B_{\delta}(y) \cap \bar B_{2^{-{j_m}}}(x_{i_m}) \not= \emptyset \Big\} \crcr
=&~ \mathcal{Z}_m \cap P_1 \Big( R \cap \big( X \times B_{\delta}(y) \cap \bar B_{2^{-{j_m}}}(x_{i_m}) \big) \Big).
\end{align*}
Being the intersection of two $\mathcal{A}$-class sets, $F^{-1}(B_\delta(y))$ is in $\mathcal{A}$.

By Corollary \ref{C:weelsupprr} there exists a $\mathcal{A}$-class section $f_m : \mathcal Z_m \to \bar B_{2^{-{j_m}}}(x_{i_m})$. The proposition follows by setting $f \llcorner_{\mathcal Z_m} = f_m$ on $\cup_m \mathcal Z_m$, and defining it arbitrarily on $\mathcal T \setminus \cup_m \mathcal Z_m$: the latter being negligible, $f$ is $\mu$-measurable.
\end{proof}

Up to a $\mu$-negligible saturated set $\mathcal{T}_N$, we can assume it to have $\sigma$-compact range: just let $S \subset f(\mathcal{T})$ be a $\sigma$-compact set where 
$f_\sharp \mu \llcorner_{\mathcal{T}}$ is concentrated, and set
\begin{equation}
\label{E:TNngel}
\mathcal{T}_S := R^{-1}(S) \cap \mathcal{T}, \quad \mathcal{T}_N := \mathcal{T} \setminus \mathcal{T}_S, \quad \mu(\mathcal{T}_N) = 0.
\end{equation}

Having the $\mu\llcorner_{\mathcal{T}}$-measurable cross-section
\[
\mathcal{S} := f(\mathcal{T})= S \cup f(\mathcal{T}_N) = (\textrm{Borel}) \cup (f(\text{$\mu$-negligible})),
\]
we can define the parametrization of $\mathcal{T}$ and $\mathcal{T}_e$ by geodesics.

\begin{definition}[Ray map]
\label{D:mongemap}
Define the \emph{ray map $g$} by the formula
\begin{align*}
g :=&~ \Big\{ (y,t,x): y \in \mathcal{S}, t \in [0,+\infty), x \in G(y) \cap \{d_L(x,y) = t\} \Big\} \crcr
&~ \cup \Big\{ (y,t,x): y \in \mathcal{S}, t \in (-\infty,0), x \in G^{-1}(y) \cap \{d_L(x,y) = -t\} \Big\} \crcr
=&~ g^+ \cup g^-.
\end{align*}
\end{definition}

\begin{proposition}
\label{P:gammaclass}
The following holds.
\begin{enumerate}
\item The restriction $g \cap S \times \R \times X$ is analytic.
\item The set $g$ is the graph of a map with range $\mathcal{T}_e$.
\item $t \mapsto g(y,t)$ is a $d_L$ $1$-Lipschitz $G$-order preserving for $y \in \mathcal{T}$.
\item $(t,y) \mapsto g(y,t)$ is bijective on $\mathcal{T}$, and its inverse is
\[
x \mapsto g^{-1}(x) = \big( f(y),\pm d_L(x,f(y)) \big)
\]
where $f$ is the quotient map of Proposition \ref{P:sicogrF} and the positive/negative sign depends on $x \in G(f(y))$/$x \in G^{-1}(f(y))$.
\end{enumerate}
\end{proposition}

\begin{proof}
For the first point just observe that
\begin{align*}
g^+ =&~ \Big\{ (y,t,x): y \in S, t \in \R^+, x \in G(y) \cap \{d_L(x,y) = t\} \Big\} \crcr
=&~ S \times R^+ \times X \cap \{(y,t,x):(y,x) \in G\} \cap \{(y,t,x):d_L(x,y) = t\} \in \Sigma^1_1.
\end{align*}
Similarly
\[
g^- = \Big\{ (y,t,x): y \in S, t \in \R^-, x \in G^{-1}(y) \cap \{d_L(x,y) = -t\} \Big\} \in \Sigma^1_1.
\]

Since $\mathcal{S} \subset \mathcal{T}$ and $R(y)$ is a subset of a single geodesic for $y \in \mathcal{S} \subset \mathcal{T}$, 
$g$ is the graph of a map. 
Note that for any $x \in\mathcal{T}_{e}$ there exists $z \in \mathcal{T}$ such that $x\in R(z)$: hence $x\in R(f(z))$, and therefore the range of the map is 
the whole $\mathcal{T}_{e}$.

The third point is a direct consequence of the definition.The fourth point follows by substitution.
\end{proof}

\begin{figure}
\label{Fi:mongemetri4}
\psfrag{W}{$W_{ijk}$}
\psfrag{S}{\color{Red}$\mathcal{S}$}
\psfrag{g}{$g$}
\psfrag{R}{$\erre$}
\centerline{resizebox{15cm}{5cm}{\includegraphics{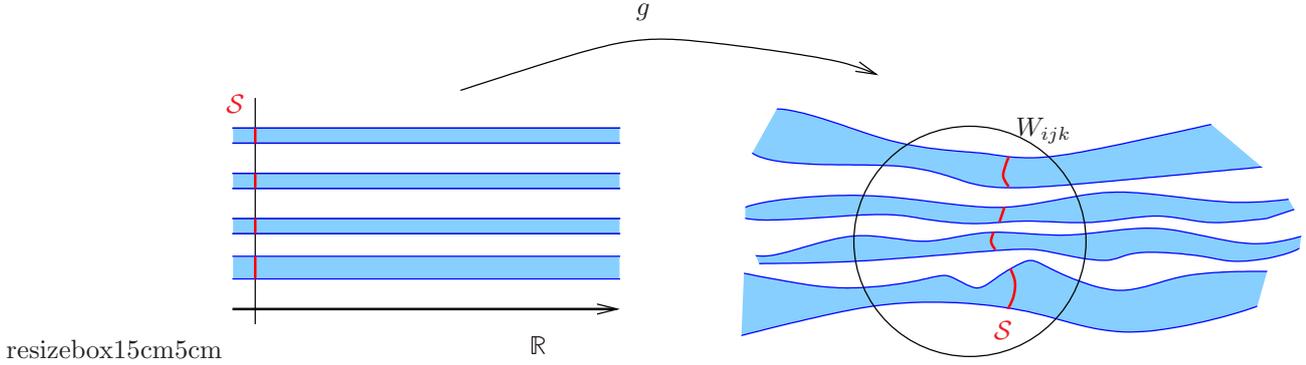}}}
\caption{The ray map $g$.}
\end{figure}

We finally prove the following property of $d_{L}$-cyclically monotone transference plans.

\begin{proposition}
\label{P:ortho}
For any $\pi$ $d_{L}$-monotone there exists a $d_L$-cyclically monotone transference plan $\tilde \pi$ with the same cost of $\pi$ such that it coincides with the identity on $\mu \wedge \nu$.
\end{proposition}

We will use the disintegration technique exploited also in the next section. We observe that another proof can be the direct composition of the transference plan with itself, using the fact that the mass moves along geodesics and the disintegration makes the problem one dimensional.

\begin{proof}
We have already shown that we can take
\[
\mu(P_2(b)) = \nu(P_2(a)) = 0,
\]
so that $\mu \wedge \nu$ is concentrated on $\mathcal{T}_S$.

{\it Step 1.} On $\mathcal{T}$ we can use the Disintegration Theorem to write
\begin{equation}
\label{E:disintT1}
\mu \llcorner_{\mathcal T} = \int_S \mu_y m(dy), \quad m = f_\sharp (\mu \llcorner_{\mathcal T}), \ \mu_y \in \mathcal{P}(R(y) \cap \mathcal T).
\end{equation}
In fact, the existence of a Borel section is equivalent to the strong consistency of the disintegration. Since $\{R(y) \times X\}_{y \in \mathcal T}$ is also a partition on $\mathcal T \times X$, we can similarly write
\[
\pi \llcorner_{\mathcal T \times X} = \int_S \pi_y m(dy), \quad \pi_y(R(y)\times R(y)) = 1.
\]
We write moreover
\begin{equation}
\label{E:nuy1}
\nu_y := (P_2)_\sharp (\pi \llcorner_{\mathcal T \times X}), \quad \tilde \nu := \int_S \nu_y m(dy) = \int_S (P_2)_\sharp \pi_y m(dy).
\end{equation}
Clearly the rest of the mass starts from $a(\mathcal T)$, so we have just to show how to rearrange the transference plan in $\mathcal T$ in order to obtain $\mu \perp \nu$. Using $g$, we can reduce the problem to a transport problem on $S \times \R$ with cost
\[
c((y,t),(y',t')) =
\begin{cases}
|t - t'| & y = y' \crcr
+ \infty & y \not= y'
\end{cases}
\]
By standard regularity argument, we can assume that $S \ni y \mapsto \pi_y \in \mathcal{P}(R(y) \times R(y))$ is $\sigma$-continuous, i.e. its graph is $\sigma$-compact.

{\it Step 2.} Using the fact that $(\mu,\nu) \mapsto \mu \wedge \nu$ is Borel w.r.t. the weak topology \cite{biacar:cmono}, we can assume that $S \ni y \mapsto \mu_y \wedge \nu_y \in \mathcal{P}(R(y))$ is $\sigma$-continuous, so that also the map
\[
S \ni y \mapsto (\mu_y - \mu_y \wedge \nu_y, \nu_y - \mu_y \wedge \nu_y) \in \mathcal{P}(R(y)) \times \mathcal{P}(R(y))
\]
is $\sigma$-continuous.

{\it Step 3.} Since in each $R(y)$ the problem is one dimensional, one can take the unique transference plan
\[
\tilde \pi_y \in \Pi \big( \mu_y - \mu_y \wedge \nu_y, \nu_y - \mu_y \wedge \nu_y \big)
\]
concentrated on a monotone set: clearly
\[
\int d_L \tilde \pi_y = \int d_L \pi_y.
\]

{\it Step 4.} If we define the left-continuous distribution functions
\[
H(y,s) := \big( \mu_y - \mu_y \wedge\nu_y \big)(-\infty,s), \quad F(y,t) := \big( \nu_y - \mu_y \wedge\nu_y \big)(-\infty,t),
\]
and
\[
G(y,s,t) := \tilde \pi_y \big( (-\infty,s) \times (-\infty,t) \big),
\]
then the measure $\tilde \pi_y$ is uniquely determined by $G(y,s,t) = \min \{ H(y,s), F(y,t) \}$.

The $\sigma$-continuity of $y \mapsto (\mu_y - \mu_y \wedge \nu_y, \nu_y - \mu_y \wedge \nu_y)$ yields that $H$, $F$ are again $\sigma$-l.s.c., so that $G$ is Borel, and finally $y \mapsto \tilde \pi_y$ is $\sigma$-continuous up to a $f_\sharp \mu$-negligible set.

{\it Step 5.} Define
\[
\hat \pi_y := \tilde \pi_y + (\Id,\Id)_\sharp (\mu_y \wedge \nu_y) \in \Pi(\mu_y,\nu_y).
\]
The above steps show that $\hat \pi$ is $m$-measurable, and thus we can define the measure
\[
\hat \pi := \pi \llcorner_{(\mathcal T_e \setminus \mathcal T) \times X} + \int \hat \pi_y m(dy).
\]
It is routine to check that $\hat \pi$ has the required properties.
\end{proof}

\section{Regularity of the disintegration}
\label{S:regurlr}

This section is divided in two parts.

In the first one we consider the translation of Borel sets by the optimal geodesic flow, we introduce a first regularity assumption (Assumption \ref{A:NDEatom}) on the measure $\mu$ and we show that an immediate consequence is that the set of initial points is negligible. A second consequence is that the disintegration of $\mu$ w.r.t. $R$ has continuous conditional probabilities.

In the second part we consider a stronger regularity assumption (Assumption \ref{A:NDE}) which gives that the conditional probabilities are absolutely continuous with respect to $\haus^{1}$ along geodesics.

\subsection{Evolution of Borel sets}
\label{Ss:evolution}

Let $A \subset \mathcal{T}_e$ be an analytic set and define for $t \in \R$ the \emph{$t$-evolution $A_t$ of $A$} by
\begin{equation}
\label{E:At}
A_t := g \big( g^{-1}(A) + (0,t) \big).
\end{equation}

\begin{lemma}
\label{L:evolution}
The set $A_t \cap g(S \times \R)$ is analytic, and $A_t$ is $\mu$-measurable for $t \geq 0$.
\end{lemma}

\begin{proof}
Divide $A$ into two parts:
\[
A_S := A \cap g(S \times \R) \quad \text{and} \quad A_N : = A \setminus A_S.
\]
From Point (1) of Proposition \ref{P:gammaclass} it follows that $A_S$ is analytic. We consider the evolution of the two sets separately.

Again by Point (1) of Proposition \ref{P:gammaclass}, the set $(A_S)_t$ is analytic, hence universally measurable for all $t \in \R$.

Since $\mathcal{T}_N$ is $\mu$-negligible (see \eqref{E:TNngel}), it follows that $(A_N)_t$ is $\mu$-negligible for all $t>0$, and by the assumptions it is clearly measurable for $t=0$.
\end{proof}

We can show that $t \mapsto \mu(A_t)$ is measurable.

\begin{lemma}
\label{L:measumuAt}
Let $A$ be  analytic. The function $t \mapsto \mu(A_t)$ is Souslin for $t \geq 0$. 
If $A \subset g(S \times \R)$, then $t \mapsto \mu(A_t)$ is Souslin for $t \in \R$.
\end{lemma}

\begin{proof}
As before, we split the $A$ into the sets
\[
A_S := A \cap g(S \times \R) \quad \text{and} \quad A_N : = A \setminus A_S.
\]

The function
\[
t \mapsto \mu(A_{N,t}) =
\begin{cases}
\mu(A_N) & t = 0 \crcr
0 & t > 0
\end{cases}
\]
is clearly Borel. Observe that since $\mathcal{T}_{N}\subset \mathcal{T}$ and the $\mu$-measure of final points is 0, 
the value of $\mu(A_{N,t})$ is known only for $t>0$.

Since $A_S$ is analytic, then $g^{-1}(A_S)$ is analytic, and the set
\[
\tilde A_S := \big\{ (y,\tau,t) : (y,\tau-t) \in g^{-1}(A_S) \big\}
\]
is easily seen to be again analytic.
Define the analytic set $\hat A_{S} \subset X \times \erre$ by 
$$
\hat{A}_{S} : = (g,\Id) (\tilde A_{S}).
$$
Clearly $(A_{S})_{t} = \hat A_{S}(t)$. We now show in two steps that the function $t \mapsto \mu((A_{S})_{t})$ is analytic. 

{\it Step 1.} Define the closed set in $\mathcal{P}(X \times [0,1])$
$$
\Pi(\mu) : = \big\{ \pi \in \mathcal{P}(X\times [0,1]) : (P_{1})_{\sharp}(\pi)= \mu \big\}
$$
and let $B \subset X\times \R \times [0,1]$ a Borel set such that $P_{12}(B)= \hat A_{S}$.

Consider the function 
\[ 
\R \times \Pi(\mu) \ni (t, \pi) \mapsto \pi(B(t)).
\]
A slight modification of Lemma 4.12 in \cite{biacar:cmono} shows that this function is Borel.

{\it Step 2.}
Since supremum of Borel function are Souslin, pag. 134 of \cite{Sri:courseborel},  
the proof is concluded once we show that 
\[
\mu((A_{S})_{t}) = \mu (\hat A_{S}(t)) = \sup_{ \pi \in \Pi(\mu) } \pi(B(t)). 
\]

From the Disintegration Theorem, for all $\pi \in \Pi(\mu)$ we have 
\[ 
\pi(B(t)) = \int \pi_{x}(B(t)) \mu(dx) \leq \int_{P_{1}(B(t))} \mu(dx) = \mu (\hat A_{S}(t)).
\]
On the other hand from Theorem \ref{T:vanneuma}, there exists an $\mathcal{A}$-measurable section $u: \hat A_{S}(t) \to B(t)$.
Clearly for $\pi_{u} = (\Id,u)_{\sharp}(\mu)$ it holds $\pi_{u}(B(t))= \mu (\hat A_{S}(t))$.
\end{proof}

The next assumption is the first fundamental assumption of the paper. 

\begin{assumption}[Non-degeneracy assumption]
\label{A:NDEatom}
For all Borel sets $A$ such that $\mu(A) > 0$ the set $\{t \in \R^+: \mu(A_t) > 0\}$ has cardinality $> \aleph_0$.
\end{assumption}

By inner regularity, it is clearly enough to verifies Assumption \ref{A:NDEatom} only for compact sets. 
Note that since for analytic set Cantor Hypothesis holds true, Theorem 4.3.5, pag. 142 of \cite{Sri:courseborel} , 
Assumption \ref{A:NDEatom} implies that the cardinality of $\{t \in \R^+: \mu(A_t) > 0\}$ is  $\mathfrak{c}$. 

An immediate consequence of the Assumption \ref{A:NDEatom} is that the measure $\mu$ is concentrated on $\mathcal{T}$.

\begin{lemma}
\label{L:puntini}
If $\mu$ satisfies Assumption \ref{A:NDEatom} then
\[
\mu(\mathcal{T}_e\setminus \mathcal{T}) = 0.
\]
\end{lemma}

\begin{proof}
If $A \subset a(X)$, then $A_t \cap A_s = \emptyset$ for $0 \leq s < t$. Hence
\[
\sharp \big\{ t \in \R^+: \mu(A_t) > 0 \big\} \leq \aleph_0,
\]
because of the boundedness of $\mu$. This contradicts the assumptions.
\end{proof}

Once we know that $\mu(\mathcal{T}) = 1$, we can use the Disintegration Theorem \ref{T:disintr} to write
\begin{equation}
\label{E:disintT}
\mu = \int_S \mu_y m(dy), \quad m = f_\sharp \mu, \ \mu_y \in \mathcal{P}(R(y)).
\end{equation}
The disintegration is strongly consistent since the quotient map $f : \mathcal T \to \mathcal T$ is $\mu$-measurable and $(\mathcal T,\mathcal B(\mathcal T))$ is countably generated.

\begin{figure}
\label{Fi:mongem5}
\psfrag{mu}{$\mu$}
\psfrag{nu}{$\nu$}
\psfrag{B}{$B$}
\psfrag{Bt}{$B_t$}
\psfrag{gamma}{$\gamma$}
\resizebox{11cm}{6cm}{\includegraphics{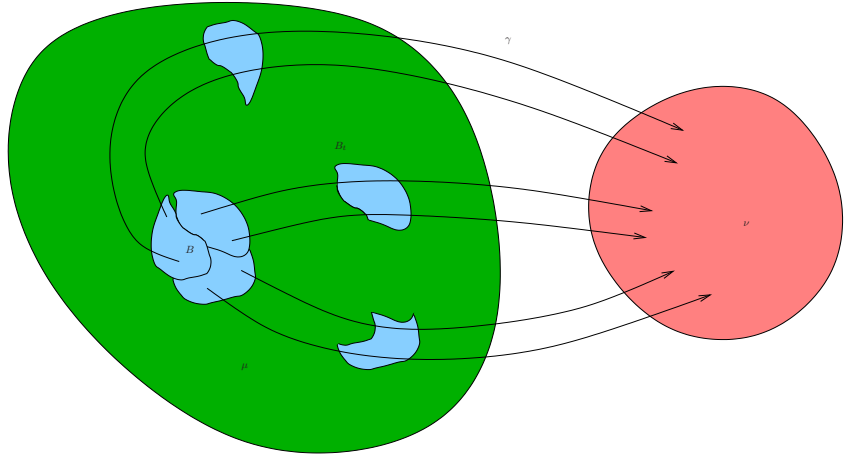}}
\caption{The evolution of a set $B$ through the optimal flow.}
\end{figure}

The second consequence of Assumption \ref{A:NDEatom} is that $\mu_y$ is continuous, i.e. $\mu_y(\{x\}) = 0$ for all $x \in X$.

\begin{proposition}
\label{P:nonatoms}
The conditional probabilities $\mu_y$ are continuous for $m$-a.e. $y \in S$.
\end{proposition}

\begin{proof}
From the regularity of the disintegration and the fact that $m(S) = 1$, we can assume that the map $y \mapsto \mu_y$ is weakly continuous on a compact set $K \subset S$ of comeasure $<\epsilon$ such that $L(R(y)) > \epsilon$ for all $y \in K$. It is enough to prove the proposition on $K$.

{\it Step 1.} From the continuity of $K \ni y \mapsto \mu_y \in \mathcal{P}(X)$ w.r.t. the weak topology, it follows that the map
\[
y \mapsto A(y) := \big\{ x \in R(y): \mu_y(\{x\}) > 0 \big\} = \cup_n \big\{ x \in R(y): \mu_y(\{x\}) \geq 2^{-n} \big\}
\]
is $\sigma$-closed: in fact, if $(y_m,x_m) \to (y,x)$ and $\mu_{y_m}(\{x_m\}) \geq 2^{-n}$, then $\mu_y(\{x\}) \geq 2^{-n}$ by u.s.c. on compact sets.

Hence it is Borel, and by Lusin Theorem (Theorem 5.8.11 of \cite{Sri:courseborel}) it is the countable union of Borel graphs: setting in case $c_i(y) = 0$, we can consider them as Borel functions on $S$ and order them w.r.t. $G$,
\[
\mu_{y,\textrm{atomic}} = \sum_{i \in \Z} c_i(y) \delta_{x_i(y)}, \quad x_{i+1}(y) \in G(x_i(y)), \ i \in \Z.
\]

{\it Step 2.} Define the sets
\[
S_{ij}(t) := \Big\{ y \in K: x_i(y) = g \big( g^{-1}(x_j(y)) + t \big) \Big\} \cap \mathcal T.
\]
Since $K \subset S$, to define $S_{ij}$ we are using the graph $g \cap S \times \R \times \mathcal{T}$, which is analytic: hence $S_{ij} \in \Sigma^1_1$.

For $A_j := \{x_j(y), y \in K\}$ and $t \in \R^+$ we have that
\begin{align*}
\mu((A_j)_t) =&~ \int_K \mu_y((A_j)_t) m(dy) = \int_K \mu_{y,\textrm{atomic}}((A_j)_t) m(dy) \crcr
=&~ \sum_{i \in \Z} \int_K c_i(y) \delta_{x_i(y)} \big( g(g^{-1}(x_j(y)) + t) \big) m(dy) = \sum_{i \in \Z} \int_{S_{ij}(t)} c_i(y) m(dy).
\end{align*}
We have used the fact that $A_j \cap R(y)$ is a singleton.

{\it Step 3.} For fixed $i,j \in \N$, again from the fact that $A_j \cap R(y)$ is a singleton
\[
S_{ij}(t) \cap S_{ij}(t') =
\begin{cases}
S_{ij}(t) & t = t' \crcr
\emptyset & t \not= t'
\end{cases}
\]
so that
\[
\sharp \big\{ t : m(S_{ij}(t)) > 0 \big\} \leq \aleph_0.
\]
Finally
\[
\mu((A_j)_t) > 0 \quad \Longrightarrow \quad t \in \bigcup_i \big\{ t : m(S_{ij}(t)) > 0 \big\},
\]
whose cardinality is $\leq \aleph_0$, contradicting Assumption \ref{A:NDEatom}.
\end{proof}

\subsection{Absolute continuity}
\label{Ss:regolarita'}

We next assume a stronger regularity assumption.

\begin{assumption}[Absolute continuity assumption]
\label{A:NDE}
For every Borel set $A \subset \mathcal{T}_{e}$
\[
\mu(A) > 0 \quad \Longrightarrow \quad \int_0^{+\infty} \mu(A_t) dt > 0.
\]
\end{assumption}

Again by inner regularity, Assumption \ref{A:NDE} can be verified only for compact sets.
Note that the condition is meaningful by Lemma \ref{L:measumuAt}. Observe moreover that Assumption \ref{A:NDE} implies Assumption \ref{A:NDEatom}, so that in the following we will restrict the map $g$ to the set $g^{-1}(\mathcal{T})$, where it is analytic. Moreover, we can consider shift $t \mapsto A_t$ for $t \in \R$, because of Lemma \ref{L:measumuAt}.

\begin{remark}\label{R:NDE}
An equivalent form of the Assumption \ref{A:NDE} is the following:
\[
\mu(A)>0 \quad \Longrightarrow \quad \int_{t,s \geq 0} \mu(A_t \cap A_s) dtds > 0.
\]
In fact, due to $\mu(X) = 1$, in the set $I_n := \{t:\mu(A_t) > 2^{-n}\}$ the set $\{s \in I_n : \mu(A_s \cap A_t) = 0, t \in I_n\}$ has cardinality at most $2^{-n}$. Hence, since for some $n$ $\mathcal{L}^1(I_n) > 0$ by Assumption \ref{A:NDE}, it follows that
\[
\mathcal{L}^2(I_n \times I_n) = \big( \mathcal{L}^1(I_n) \big)^2 > 0.
\]
The opposite implication is a consequence of Fubini theorem.
\end{remark}

The next results show regularity of the Radon-Nikodym derivative of $\mu_y$ w.r.t. $(\haus^1_L)\llcorner_{ f^{-1}(y)}$, where $\haus^1_L$ is the $1$-dimensional Hausdorff measure w.r.t. the $d_L$-distance. Note that along $d_L$ $1$-Lipschitz geodesics, $\haus^1_L$ is equivalence to $g(y,\cdot)_\sharp \mathcal{L}^1$: in the following we will use both notations.
%

\begin{lemma}
\label{Lem:dec}
Let $\mu$ be a Radon measure and
\[
\mu_y = r(y,\cdot) g(y,\cdot)_\sharp \mathcal{L}^1 + \omega_y, \quad \omega_y \perp g(y,\cdot)_\sharp \mathcal{L}^1
\]
be the Radon-Nikodym decomposition of $\mu_y$ w.r.t. $g(y,\cdot)_\sharp \mathcal{L}^1$. Then there exists a Borel set $C\subset X$ such that
\[
\mathcal{L}^{1} \big( g^{-1} (C) \cap (\{y\} \times \R) ) \big) = 0
\]
and $\omega_y = \mu_y \llcorner_C$ for $m$-a.e. $y \in [0,1]$.
\end{lemma}

\begin{proof}
Consider the measure
\[
\lambda = g_\sharp (m \otimes \mathcal L^1),
\]
and compute the Radon-Nikodym decomposition
\[
\mu  = \frac{D \mu}{D \lambda} \lambda + \omega.
\]
Then there exists a Borel set $C$ such that $\omega = \mu \llcorner_C$ and $\lambda(C)=0$. The set $C$ proves the Lemma. Indeed $C = \cup_{y \in [0,1]} C_{y}$ where $C_{y} = C \cap f^{-1}(y)$ is such that $\mu_y \llcorner_{C_{y}} = \omega_{y} $ and $g(y,\cdot)_{\sharp}\mathcal{L}^{1}(C_{y})=0$ for $m$-a.e. $y \in [0,1]$.
\end{proof}

\begin{theorem}
\label{teo:a.c.}
If $\mu$ satisfies Assumption \ref{A:NDE}, then for $m$-a.e. $y \in [0,1]$ the conditional probabilities $\mu_y$ are absolutely continuous w.r.t. $g(y,\cdot)_\sharp \mathcal{L}^1$.
\end{theorem}

The proof is based on the following simple observation.

\medskip
\noindent Let $\eta$ be a Radon measure on $\erre$. Suppose that for all $A \subset \erre$ Borel with $\eta(A)>0$ it holds
\[
\int_{\R^+} \eta(A+t) dt = \eta \otimes\mathcal{L}^1 \big( \{ (x,t): t \geq 0, x - t \in A \} \big) > 0.
\]
Then $\eta \ll \mathcal{L}^1$.

\begin{proof}
The proof will use Lemma \ref{Lem:dec}: take $C$ the set constructed in Lemma \ref{Lem:dec} and suppose by contradiction that
\[
\mu(C) > 0 \quad \text{and} \quad m \otimes \mathcal{L}^1 (g^{-1}(C)) = 0.
\]
In particular, for all $t \in \R$ it follows that
\[
m \otimes \mathcal{L}^1 (g^{-1}(C_t)) = m \otimes \mathcal{L}^1 (g^{-1}(C) + (0,t)) = 0.
\]
By Fubini-Tonelli Theorem
\begin{align*}
0< &~ \int_{\R^+} \mu(C_t) dt  =  \int_{\R^+} \bigg( \int_{g^{-1}(C_t)} (g^{-1})_\sharp \mu(dyd\tau) \bigg) dt \crcr
=&~ \big( (g^{-1})_\sharp \mu \otimes \mathcal{L}^1 \big) \Big( \Big\{ (y,\tau,t): (y,\tau) \in g^{-1}(\mathcal{T}), (y,\tau-t) \in g^{-1}(C) \Big\} \Big) \crcr
\leq&~ \int_{S \times \R} \mathcal{L}^1 \big( \big\{\tau - g^{-1}(C \cap f^{-1}(y)) \big\} \big) (g^{-1})_{\sharp} \mu (dyd\tau) \crcr
=&~ \int_{S \times \R} \mathcal{L}^1 \big( g^{-1}(C \cap f^{-1}(y)) \big) (g^{-1})_{\sharp} \mu (dyd\tau) \crcr
=&~ \int_{S} \mathcal{L}^1 \big( g^{-1}(C \cap f^{-1}(y)) \big) m(dy) = 0.
\end{align*}
That gives a contradiction.
\end{proof}

Now we will study the regularity of the map $t \mapsto \mu(A_t)$ under Assumption \ref{A:NDE}. We will use the following notation:
\[
\mu(A) = \int_S \mu_y(A) m(dy) = \int_S \bigg( \int_{g(y,\cdot)^{-1}(A)} r(y,\tau) d\tau \bigg) m(dy) = g_\sharp (r m \otimes \mathcal L^1).
\]

\begin{proposition}
\label{P:peloso}
$\mu$ satisfies Assumption \ref{A:NDE} if and only if for all $A$ Borel $t \mapsto \mu(A_t)$ is continuous. Moreover if $A$ is geodesically convex then $\mu(A_t)$ is absolutely continuous.
\end{proposition}

\begin{proof}
It is enough to prove the continuity for $t=0$. Since
\[
\mu(A_t) = \int_S \bigg( \int_{g(y,\cdot)^{-1}(A_t)} r(y,\tau) d\tau \bigg) m(dy),
\]
its continuity is a direct consequence of Lebesgue dominated convergence theorem applied to the function:
\[
t \mapsto \mu_y(A_t) = \int_{g(y,\cdot)^{-1}(A_t)} r(y,\tau) d\tau.
\]

Suppose now $A$ geodesically convex. Each $g(y,\cdot)^{-1}(A)$ is an interval $(\alpha(y),\omega(y))$, so that the map
\[
t \mapsto \int_{g(y,\cdot)^{-1}(A_t)} r(y,\tau) d\tau
\]
is absolutely continuous with derivative
\[
h(y,t) = r(y,\omega(y)+t) - r(y,\alpha(y)+t).
\]
Since $h(y,t) \in L^1(m \otimes \mathcal{L}^1)$ the result follows by a standard computation.
\end{proof}

\section{Solution to the Monge problem}
\label{S:Solution}

In this section we show that Theorem \ref{teo:a.c.} allows to construct an optimal map $T$. We recall the one dimensional result for the Monge problem \cite{villa:Oldnew}.

\begin{theorem}
\label{T:oneDmonge}
Let $\mu$, $\nu$ be probability measures on $\erre$, $\mu$ continuous, and let
\[
H(s) := \mu((-\infty,s)), \quad F(t) := \nu((-\infty,t)),
\]
be the left-continuous distribution functions of $\mu$ and $\nu$ respectively. Then the following holds.
\begin{enumerate}
\item The non decreasing function $T : \erre \to \erre \cup [-\infty,+\infty)$ defined by
\[
T(s) := \sup \big\{ t \in \erre : F(t) \leq H(s) \big\}
\]
maps $\mu$ to $\nu$. Moreover any other non decreasing map $T'$ such that $T'_\sharp \mu = \nu$ coincides with $T$ on the support of $\mu$ up to a countable set.
\item If $\phi : [0,+\infty] \to \erre$ is non decreasing and convex, then $T$ is an optimal transport relative to the cost $c(s,t) = \phi(|s-t|)$. Moreover $T$ is the unique optimal transference map if $\phi$ is strictly convex.
\end{enumerate}
\end{theorem}

Assume that $\mu$ satisfies Assumption \ref{A:NDEatom}. Then we can disintegrate $\mu$ and $\pi$ respect to the ray equivalence relation $R$ and $R \times X$ as in \eqref{E:disintT},
\begin{equation}
\label{E:muy}
\mu = \int \mu_y m(dy), \ \pi = \int \pi_y m(dy), \quad \mu_y \ \text{continuous}, \ (P_1)_\sharp \pi_y = \mu_y.
\end{equation}
We write moreover
\begin{equation}
\label{E:nuy}
\nu = \int \nu_y m(dy) = \int (P_2)_\sharp \pi_y m(dy).
\end{equation}
Note that $\pi_y \in \Pi(\mu_y,\nu_y)$ is $d_L$-cyclically monotone (and hence optimal, because $R(y)$ is one dimensional) for $m$-a.e. $y$. If $\nu(\mathcal{T}) = 1$, then \eqref{E:nuy} is the disintegration of $\nu$ w.r.t. $R$.

\begin{theorem}
\label{T:mongeff}
Let $\pi \in \Pi(\mu,\nu)$ be a $d_L$-cyclically monotone transference plan, and assume that Assumption \ref{A:NDEatom} holds. Then there exists a Borel map $T: X \to X$ with the same transport cost as $\pi$.
\end{theorem}

\begin{proof}
By means of the map $g^{-1}$, we reduce to a transport problem on $S \times \R$, with cost
\[
c((y,s),(y',t)) =
\begin{cases}
|t - s| & y = y' \crcr
+ \infty & y \not= y'
\end{cases}
\]
It is enough to prove the theorem in this setting under the following assumptions: $S$ compact and $S \ni y \mapsto (\mu_y,\nu_y)$ weakly continuous. We consider here the probabilities $\mu_y$, $\nu_y$ on $\R$.

{\it Step 1.} From the weak continuity of the map $y \mapsto (\mu_y,\nu_y)$, it follows that the maps
\[
(y,t) \mapsto H(y,t) := \mu_y((-\infty,t)), \ \ (y,t) \mapsto F(y,t) := \nu_y((-\infty,t))
\]
are easily seen to be l.s.c..  Both are clearly increasing in $t$. Note also that $H$ is continuous in $t$.

{\it Step 2.} The map $T$ defined as Theorem \ref{T:oneDmonge} by
\[
T(y,s) := \Big( y, \sup \big\{ t : F(y,t) \leq H(y,s) \big\} \Big)
\]
is Borel. In fact, for $A$ Borel,
\[
T^{-1}(A \times [t,+\infty)) = \big\{ (y,s) : y \in A, H(y,s) \geq F(y,t) \big\} \in \mathcal{B}(S \times \R).
\]

{\it Step 3.} Note that $\pi_{y}$ and $T(y,\cdot)$ are both optimal for the transport problem between $\mu_{y}$ and $\nu_{y}$ with cost $d_{L}$ restricted to $R(y)$. 
Indeed $d_{L}$ restricted to $R(y)\times R(y)$ is finite. Therefore $\pi_{y}$ and $T(y,\cdot)$ have the same cost.
\end{proof}

\begin{remark}\label{R:monot}
By the definition of the set $G$, it follows that along each geodesic $\mu_y(g(y,(-\infty,t))) \geq \nu_y(g(y,(-\infty,t)))$, because in the opposite case $G$ is not $d_L$-cyclically monotone. Hence $T(s) \geq s$, and $c((y,s),T(y,s)) =P_{2}( T(y,s)) - s$. Hence 
\begin{equation}\label{E:costo}
\int d_{L} \pi =\int d_{L}(x,T(x))\mu(dx) =  \int_{S\times \erre} s \big(g(y,\cdot)^{-1}_{\sharp}(\nu_{y} - \mu_{y})\big)(ds)m(dy)= \int P_{2}(g^{-1}(x)) (\nu - \mu)(dx).
\end{equation}
\end{remark}

\section{Dynamic interpretation}
\label{S:div}

In this section we show how the regularity of the disintegration yields a correct definition of the current $\dot g$ representing the flow along the geodesics of an optimal transference plan. This allows to solve the PDE
\[
\partial U = \mu - \nu
\]
in the sense of currents in metric spaces. In particular, under additional regularity assumptions, one can prove that the boundary $\partial \dot g$ is well defined and satisfies an ODE along geodesics. This gives a dynamic interpretation to the transport problem.

The setting here is slightly different from the previous sections: 
\begin{enumerate}
\item $d(x,y) \leq d_L(x,y)$;
\item there exists a probability measure $\eta$, such that it (or more precisely $\eta \llcorner_{\mathcal{T}_e}$) satisfies Assumption \ref{A:NDE} along the transport rays of the transportation problem with marginals $\mu$, $\nu$;
\item $\mu \ll \eta$, so that also $\mu$ satisfies Assumption \ref{A:NDE}.
\end{enumerate}
In particular, $\textrm{Lip}(X) \subset \textrm{Lip}_{d_L}(X)$.

The main reference for this chapter is \cite{ambkir:currentmetric}.

\subsection{Definition of $\dot g$}
\label{Ss:dotgam}

For any Lipschitz function $\omega : X \to \erre$ we can define the derivative $\partial_t \omega$ along the geodesic $g(t,y)$ for a.e. $t \in \R$,
\[
\partial_t \omega(g(y,t)) := \frac{d}{dt} \omega(g(t,y)).
\]
Using the disintegration formula
\[
\eta \llcorner_{\mathcal T} = \int (g(y,\cdot))_\sharp (q(y,\cdot) \mathcal L^1) m(dy) = g_\sharp (q m \otimes \mathcal L^1)
\]
for some $q \in L^{1}(m \otimes \mathcal{L}^{1})$ (Theorem \ref{teo:a.c.}), we can define the measure $\partial_t \omega \eta$ as
\[
\int \phi(x) (\partial_{t} \omega \eta)(dx) :=  \int_S \int_\R \phi(g(y,t)) \partial_t \omega(g(y,t)) q(y,t) dt m(dy).
\]
where $\phi \in C_b(X,\erre)$.

\begin{definition}
\label{D:dotgamma}
We define the \emph{flow $\dot g$} as the current
\[
\langle \dot{g}, (h, \omega) \rangle = \int_{S \times \R}  h(g(y,t)) \partial_t \omega(g(y,t)) q(y,t) dt m(dy)
\]
where $h$, $\omega$ are Lipschitz functions of $(X,d)$ with $h$ bounded.
\end{definition}

It is fairly easy to see that $\dot g$ is a current: in fact,
\begin{enumerate}
\item $\dot g$ has finite mass, namely
\[
\big| \langle \dot g, (h,\omega) \rangle \big| \leq \textrm{Lip}(\omega) \int h \eta;
\]
\item $\dot g$ is linear in $h$, $\omega$;
\item if $\omega_n \to \omega$ pointwise in $X$ with uniformly bounded Lipschitz constant, then by Lebesgue Dominated Convergence Theorem if follows that
\[
\lim_{n \to +\infty} \langle \dot g, (h_n,\omega_n) \rangle = \langle \dot g, (h,\omega) \rangle;
\]
\item $\langle \dot g, (h,\omega) \rangle = 0$ if $\omega$ is constant in $\{h \not= 0\}$.
\end{enumerate}

In general, $\dot g$ is only a current, with boundary $\partial \dot g$ defined by the duality formula
\begin{equation}
\label{E:boundgamma}
\langle \partial \dot g, \omega \rangle = \langle \dot g, (1,\omega) \rangle.
\end{equation}
Under additional assumptions, the current $\dot g$ is a normal current, i.e. $\partial \dot g$ is also a scalar current, in particular it is a bounded measure on $(X,d)$.

\begin{lemma}
\label{L:normalcurr}
Assume that $q(y,\cdot) : \R \to \R$ belongs to $\BV(\R)$ for $m$-a.e. $y$ and
\[
\sigma_y := - \frac{d}{dt} q(y,t), \quad \int_S |\sigma_y(\R)| m(dy) = \int_S \TV(q(y,\cdot) )m(dy) < + \infty.
\]
Then $\dot g$ is a normal current and its boundary is given by
\[
\langle \partial \dot g, \omega \rangle = \int_S \int_\R \omega(g(y,t)) \sigma_y(dt) m(dy).
\]
\end{lemma}

Note that in the above formula we cannot restrict $\sigma_y$ to $g^{-1}(\mathcal T)$: in fact, in general
\[
\int_S (g(y,\cdot)_\sharp \sigma_y)(\mathcal T_e \setminus \mathcal T) m(dy) > 0.
\]

\begin{proof}
First of all, by using the formula $q(y,t) = \sigma_y((t,+\infty))$, it follows that $\sigma_y$ is $m$-measurable, i.e. for all $\phi \in C_b(X,\erre)$ the integral
\[
\int \bigg( \int \phi(g(y,t)) \sigma_y(dt) \bigg) m(dy)
\]
is meaningful and then
\[
\int \big( g(y,\cdot)_\sharp \sigma_y \big) m(dy)
\]
is a finite measure on $(X,d)$.

A direct computation yields
\begin{align*}
\langle \partial \dot g, \omega \rangle = &~ \langle \dot g, (1,\omega) \rangle = \int_S \int_\R \partial_t \omega(g(t,y)) \sigma_y((t,+\infty)) dt m(dy) = \int_S \int_\R \omega(g(t,y)) \sigma_y(dt) m(dy).
\end{align*}
\end{proof}

\begin{remark}
\label{R:abscurr}
In many cases the measure $\int (g(y,\cdot)_\sharp \sigma_y) \llcorner_{\mathcal{T}} m(dy)$ is absolutely continuous w.r.t. $\eta$, i.e. for $m$-a.e. $y$
\[
\sigma_{y}\llcorner_{\mathcal{T}} = h(g(t,y)) q(y,t) \mathcal{L}^1.
\]
for some $h \in L^1(\eta)$. In that case we obtain that
\begin{align*}
\langle \partial \dot g, \omega \rangle =&~ \int \omega(b(y)) \sigma_y \big( P_2(\{g^{-1}(b(y))\}) \big) m(dy) \crcr
&~ - \int \omega(a(y)) \sigma_y \big( P_2(\{g^{-1}(a(y))\}) \big) m(dy) + \int \omega(x) h(x) \eta(dx).
\end{align*}
\end{remark}

\subsection{Transport equation}
\label{Ss:transpT}

We now consider the problem $\partial U = \mu - \nu$ in the sense of currents:
\[
\langle U, (1,\omega) \rangle = \langle \mu - \nu, \omega \rangle = \int \omega(x) (\mu - \nu)(dx).
\]
Using the disintegration formula and \eqref{E:muy}, \eqref{E:nuy} we can write
\[
\langle U, (1,\omega) \rangle = \int_S \bigg\{ \int_\R \omega(g(y,t)) (g^{-1}(y,\cdot)_\sharp \mu_y)(dt) - \int_\R \omega(g(y,t)) (g^{-1}(y,\cdot)_\sharp \nu_{y})(dt)  \bigg\} m(dy).
\]
By integrating by parts we obtain
\begin{align*}
\int_\R \omega(g(y,t)) (g^{-1}(y,\cdot)_\sharp \mu_y)(dt) =&~ - \int_\R \mu_y(g(y,(-\infty,t))) \partial_t \omega(g(y,t)) dt = - \int_\R H(y,t) \partial_t \omega(g(y,t)) dt,
\end{align*}
\begin{align*}
\int_\R \omega(g(y,t)) (g^{-1}(y,\cdot)_\sharp \nu_y)(dt) = - \int_\R \nu_y(g(y,(-\infty,t))) \partial_t \omega(g(y,t)) dt = - \int_\R F(y,t) \partial_t \omega(g(y,t)) dt.
\end{align*}

Observe that the map 
\[
S \times \R \ni (y,t) \mapsto F(y,t) - H(y,t) \in \R
\]
is in $L^1(m \otimes \mathcal L^1)$ if the transport cost $\mathcal{I}(\pi)$ is finite: in fact, using the fact that $F(y,t) \leq H(y,t)$ and integrating by parts,
\begin{equation}
\label{E:elle1}
\int_\R H(y,t) - F(y,t) dt = \int_\R  (g^{-1}(y,\cdot)_\sharp \mu_y - g^{-1}(y,\cdot)_\sharp \nu_y)(-\infty,t) (dt) = \int_{\R^2} (t - s) \tilde \pi_{y}(ds,dt),
\end{equation}
where $\tilde \pi_{y}$ is the monotone rearrangement.

We deduce the following proposition.

\begin{proposition}
\label{P:bidual}
Under Assumption \ref{A:NDEatom}, a solution to $\partial U = \mu - \nu$ is given by the current $U$ defined as
\[
\langle U, (h,\omega) \rangle = \int_S \bigg( \int_\R (F(y,t) - H(y,t)) h(g(y,t)) \partial_t \omega(g(y,t)) dt \bigg) m(dy).
\]
\end{proposition}

In general, the solution is not unique: just add a boundary free current to our solution.

Some further assumptions allow to represent our solution $U$ as the product of a scalar $\rho$ with the current $\dot g$.

\begin{proposition}
\label{P:dualregu}
Assume that $q(y,t) > 0$ whenever $H(y,t) - F(y,t) > 0$. Then $R = \rho \dot g$, where
\[
\rho(g(y,t)) = \frac{F(y,t) - H(y,t)}{q(y,t)}.
\]
\end{proposition}

\begin{proof}
It is enough to observe that
\begin{align*}
\int_{S \times \R} F(y,t) - H(y,t) dt m(dy) =&~ \int_{S \times \R} \frac{F(y,t) - H(y,t)}{q(y,t)} q(y,t) dt m(dy) \crcr
=&~ \int_{S \times \R} \rho(g(y,t)) q(y,t) dt m(dy) = \int_X \rho(x) \eta(dx),
\end{align*}
and from \eqref{E:elle1} we conclude that $\rho \in L^1(\eta)$.
\end{proof}

\begin{corollary}
\label{C:regudual}
If $q(y,t) \not= 0$ for $m \otimes \mathcal{L}^1$-a.e. $(y,t) \in g^{-1}(\mathcal{T})$, then there exists a scalar function $\rho$ such that $\partial (\rho \dot g) = \mu - \nu$.
\end{corollary}

\section{Stability of the non degeneracy condition}
\label{S:limite}

In this section we prove a general approximation theorem, which will be then applied to the Measure-Gromov-Hausdorff (MGH) convergence: if a uniform estimate holds for the disintegration in the approximating spaces, we deduce the regularity of the disintegration also in the limit.

\subsection{A general stability result}
\label{Ss:genstb}

We consider the following setting:
\begin{enumerate}
\item $\mu_n$ is a sequence of measure converging to $\mu$ weakly;
\item there exists functions $g_n : S_n \times \R \to X$, $S_n \subset X$ Borel, and measures $r_n m_n \otimes \mathcal{L}^1 \in \mathcal P(S_n \times \R)$ such that

\begin{equation}\label{enum:erren}
\mu_n = (g_n)_\sharp \big( r_n m_n \otimes \mathcal{L}^1 \big).
\end{equation}
\end{enumerate}


The following is the basic tool for our stability result.
\begin{proposition}
\label{P:limite}
Let $Y$ be a Polish space, $\{\xi_n\}_{n \in \enne} \subset \mathcal{P}(Y)$ such that $\xi_n \rightharpoonup \xi$. 
Consider $\{r_n\}_{n \in \enne}$, $r_n \geq 0$, such that $r_n \in L^{1}( \xi_n )$, $r_n \xi_n \weak \zeta$ and the following equintegrability condition holds:
\[
\forall \ve > 0 \ \exists \delta >0  \ \bigg( \forall A \in \mathcal B, \xi_n(A) < \delta \quad \Longrightarrow \quad \int_A r_n \xi_n < \ve \bigg).
\]
Then there exists $r \in L^1(\xi)$ such that $\zeta= r \xi$.
\end{proposition}

\begin{proof}
We will show that $\zeta(B) = 0$ for all $B$ such that $\xi(B) = 0$. Clearly by inner and outer regularity, it is enough to prove the following statement:
\[
\forall \ve > 0 \ \exists \delta >0 \ \bigg( \phi \in C_{b}(Y), \phi\geq 0, \ \int \phi \xi < \delta \quad \Longrightarrow \int \phi \zeta < \ve \bigg).
\]

Fix $\ve > 0$ and take the corresponding $\delta$ given by the equintegrability condition on $r_n$. Clearly w.l.o.g. $\delta\leq\ve$.
Consider $\phi \in C_b(Y)$ positive such that
\[
\int \phi \xi \leq \delta^2/2.
\]
From the weak convergence for $n$ great enough
\[
\int \phi \xi_n \leq \delta^2,
\]
so that we can estimate
\[
\int \phi r_n \xi_n \leq \int_{\phi > \delta} r_n \xi_n + \delta <  \ve+\delta . 
\]
Hence $\int \phi \zeta < 2 \ve$.
\end{proof}

\begin{theorem}
\label{T:apprograph}
Assume that the family of functions $\{r_n\} \subset L^1(m_n \otimes \mathcal{L}^1)$ given by \eqref{enum:erren} is such that 
\[
(\Id,\Id,g_n)_\sharp \big( r_n m_n \otimes \mathcal{L}^1 \big) \rightharpoonup (\Id,\Id,g)_\sharp \zeta
\]
with $\zeta \in \mathcal{P}(S \times \R)$ and $g$ being the ray map (Definition \ref{D:mongemap}). Assume moreover
\[
\forall T\geq 0 \ \forall \ve > 0 \ \exists \delta >0  \ \bigg( A \in \mathcal B(S \times[-T,T]), m_{n}\otimes\mathcal{L}^{1}(A) < \delta \quad \Longrightarrow \quad 
\int_A r_n m_n\otimes\mathcal{L}^{1} < \ve \bigg).
\]
Then $\zeta = r m \otimes \mathcal{L}^1$ for some function $r \in L^1(m \otimes \mathcal{L}^1)$, measure $m \in \mathcal{P}(S)$ and the disintegration of $\mu$ is a.c. w.r.t. $\mathcal H^1$ on each geodesic.
\end{theorem}

\begin{proof}
Define for $k \in \enne$
\[ 
\phi_{k} \in C_{c}(\erre),\ \phi_{k}\geq 0, \quad  
\phi_{k}(t):=
\begin{cases}
1 & |t| \leq k, \crcr
0 & |t| \geq k+1.
\end{cases}
\]
Let $\xi_{n,k}=m_{n}\otimes\mathcal{L}^{1}\llcorner_{[-k-1,k+1]}$ and consider the functions $\tilde r_{n,k}:= r_{n}(y,t)\phi_{k}(t)$. 
Since $m_{n}=(P_{1})_{\sharp}(r_{n}m_{n}\otimes\mathcal{L}^{1})$ and hence $m_{n}\weak m = (P_{1})_{\sharp}\zeta$,  then 
\[
\xi_{n,k} \weak m\otimes \mathcal{L}^{1}\llcorner_{[-k-1,k+1]}  
\]
and the hypothesis of Proposition \ref{P:limite} are verified up to rescaling. So $\zeta = r m \otimes \mathcal{L}^1$.

The fact that $g_\sharp \zeta$ is a disintegration is a consequence of the a.c. of $\zeta$ along each geodesic: in this case the initial points have $\zeta$-measure $0$ 
and therefore $g$ is invertible on a set of full $\mu$-measure. 
\end{proof}

In general the convergence of the graph of $g_n$ is too strong: the next result considers a more general case.

\begin{proposition}
\label{P:finalapr}
Assume that $\tilde \zeta \in \Pi(r m \otimes \mathcal{L}^1, \mu)$ is concentrated on the graph of a Borel function $h : \mathcal{T} \times \R \to \mathcal{T}_e$ such that
\begin{enumerate}
\item \label{Cond:1final} $(y,t) \mapsto e(y) := f(h(y,t)) \in S$ is constant w.r.t. $t$,
\item it holds
\[
h(y,\cdot)_\sharp \big( r(y,\cdot) \mathcal{L}^1 \big) \ll \mathcal{H}^1 \llcorner_{g(e(y),\R)}.
\]
\end{enumerate}
Then the disintegration w.r.t. $g$ has absolutely continuous conditional probability.
\end{proposition}

\begin{proof}
We can disintegrate the measure $m$ as follows:
\[
m = \int_S m_z (e_\sharp m)(dz),
\]
and by the second assumption
\[
h(y,\cdot)_\sharp (r(y,\cdot) \mathcal{L}^1) = g(e(y),\cdot)_\sharp (\tilde r(y,\cdot) \mathcal{L}^1),
\]
for $m$-a.e. $y \in \mathcal{T}$. Hence by explicit computation,
\begin{align*}
\mu =&~  \int_{S} h(y,\cdot)_{\sharp}( r(y,\cdot) \mathcal{L}^{1} ) m(dy) = \int_{S} g(e(y),\cdot)_{\sharp}( \tilde r(y,\cdot) \mathcal{L}^{1} ) m(dy) \crcr
=&~ \int_{S} \bigg( \int_{e^{-1}(z)} g(z,\cdot)_{\sharp} (\tilde r(y,\cdot) \mathcal{L}^{1}) m_{z}(dy)\bigg)  e_{\sharp}m(dz).
\end{align*}
To conclude the proof observe that 

\begin{align*}
\int_{e^{-1}(z)} g(z,\cdot)_{\sharp} (\tilde r(y,\cdot) \mathcal{L}^{1}) m_{z}(dy) = &~ 
g(z,\cdot)_{\sharp}\bigg( \int_{e^{-1}(z)} \tilde r(y,\cdot) \mathcal{L}^{1} m_{z}(dy) \bigg) \crcr
=&~ g(z,\cdot)_{\sharp}\bigg( \int_{e^{-1}(z)} \tilde r(y,\cdot) m_{z}(dy) \bigg) \mathcal{L}^{1}.
\end{align*}
\end{proof}

\begin{remark}
\label{R:morereg}
Observe that some properties of $r_n$ are preserved passing to the limit $r$. 
In relation with the previous section, we consider the following cases: for $A \subset X \times \erre$ open
\begin{enumerate}
\item for some $\ve>0$
\[
\big((r_n -\ve)m_{n}\otimes \mathcal{L}^{1}\big)\llcorner_{A} \geq0;
\]
\item there exists $L> 0$ such that
\[
r_n(y,\cdot) \in  \textrm{Lip}_{L}(A_{y});
\]
\item  there exists $M>0$ such that
\[
TV( r_n(y,\cdot)\llcorner_{A}) \leq M.
\]
\end{enumerate}
The first condition yields that the assumptions of Corollary \ref{C:regudual} holds in $A$.
The second and third conditions imply that we are under the conditions for Remark \ref{R:abscurr} in $A$.
\end{remark}

\subsection{Approximations by metric spaces}
\label{Ss:GHC}

In this section we explain a procedure to verify if the transport problem under consideration satisfies Assumption \ref{A:NDE}. The basic references for this sections are \cite{villott:curv} and \cite{sturm:MGH1,sturm:MGH2}.

We consider the following setting:
\begin{enumerate}
\item $(X,d,d_{L})$, $(X_n,d_n,d_{L,n})$, $n \in \N$, are metric structures satisfying the assumptions of page \pageref{P:assumpDL} and Remark \ref{R:compact}: more precisely,
$d_{L},d_{L,n}$ l.s.c.,  $d_{L}\geq d, d_{L,n}\geq d_{n}$ and
\[
\bigcup_{x \in K_1, y \in K_2} \gamma_{[x,y]}\ \  \textrm{is $d_{n}$($d$)-compact if $K_1$, $K_2$ are $d_{n}$($d$)-compact, $d_{L,n}(d_{L}) \llcorner_{K_1 \times K_2}$ uniformly bounded}.
\]
\item $\mu_n, \nu_n \in \mathcal{P}(X_n)$, $\mu_n \perp \nu_n$;
\item $\pi_n \in \Pi(\mu_n,\nu_n)$ is a $d_{L,n}$-cyclically monotone transference plan with finite cost.
\end{enumerate}
For $\mu,\nu \in \mathcal{P}(X)$ let $\pi \in \Pi(\mu,\nu)$ be a generic transference plan.

\begin{definition}\label{D:mgh}
We say that the structures $(X_n,d_n,d_{L,n}, \pi_n)$ \emph{converge} to $(X,d,d_L,\pi)$ if the following holds: 
there exists $C>0$ such that for all $n \in \N$ 
\begin{equation*}
\int d_{L,n} \pi_{n}\leq C
\end{equation*}
and there exist Borel sets $A_n \subset X_n$ and Borel maps $\ell_n : A_n \to X$ such that
\begin{equation}
\label{E:conMHD0}
(\ell_n\otimes \ell_n)_\sharp \pi_n \llcorner_{A_n \times A_n} \rightharpoonup \pi,
\end{equation}
\begin{equation}
\label{E:conMHD}
 \big| d_L(\ell_n(x),\ell_n(y)) - d_{L,n}(x,y) \big| \leq 2^{-n},
\end{equation}
and if $(\ell_n(x_n),\ell_n(y_n)) \to (x,y)$, then
\begin{equation}
\label{E:conMHD1}
d_L(x,y) = \lim_n d_{L,n}(x_n,y_n).
\end{equation}
\end{definition}

As a first result, we show that also $\pi$ is $d_{L}$-cyclically monotone with finite cost.

\begin{proposition}\label{P:bddmono}
If $(X_n,d_n,d_{L,n},\pi_n)$ converges to $(X,d,d_L,\pi)$ and the plans $\pi_{n}$ have uniformly bounded cost then 
also $\pi$ has finite cost and is $d_{L}$-cyclically monotone.
\end{proposition}
\begin{proof}
Since $d_{L}$ is l.s.c.
\begin{align*}
\int d_{L} \pi \leq~&  \liminf_{n\to +\infty} \int d_{L} (\ell_{n}\otimes \ell_{n})_{\sharp}\pi_{n}  =  
\liminf_{n\to +\infty} \int d_{L}(\ell_{n}(x), \ell_{n}(y)) \pi_{n}(dxdy) \crcr
\stackrel{\textrm{\eqref{E:conMHD}}}{\leq}&  \liminf_{n\to +\infty} \bigg\{ \int d_{L,n}(x,y) \pi_{n}(dxdy) + 2^{-n} \bigg\} \leq C,
\end{align*}
for some $C<+\infty$.

Now let $\Gamma_{n}$ be a $d_{L,n}$-cyclically monotone set with $\pi_{n}(\Gamma_{n})=1$: by standard regularity of Borel function and by Prokhorov Theorem 
we can assume that 
\begin{enumerate}
\item $\Gamma_{n}$ is $\sigma$-compact, $\Gamma_{n}=\cup_{m\in \enne} \Gamma_{n,m}$ with $\Gamma_{n,m}\subset \Gamma_{n,m+1}$;
\item $(\ell_{n}\otimes \ell_{n})(\Gamma_{n,m})$ is compact and $(\ell_{n}\otimes \ell_{n})(\Gamma_{n,m}) \to \Gamma_{m}$ 
in the Hausdorff distance $d_{H}$; 
\item $\pi_{n}(\Gamma_{n,m})\geq 1- 2^{-m}$.
\end{enumerate}
It follows that: $\pi(\Gamma_{m})\geq 1 - 2^{-m}$, hence 
\[
\pi \bigg(\bigcup_{m \in \enne}\Gamma_{m} \bigg)=1.
\]
Since each $\Gamma_{m}$ is  the limit in Hausdorff distance of $(\ell_{n}\otimes \ell_{n})(\Gamma_{n,m})$, 
\eqref{E:conMHD1} implies that $\Gamma_{m}$ (and thus $\cup_{m}\Gamma_{m}$, because $\Gamma_{m}\subset \Gamma_{m+1}$) 
is $d_{L}$-cyclically monotone.
\end{proof}

Note that since $\pi$ is $d_{L}$-cyclically monotone, we can define 
the sets $\Gamma, \Gamma', G, G^{-1}, R, a, b$ of Section \ref{S:Optimal} as well as 
the quotient map $f$ and the ray map $g$ constructed in Section \ref{S:partition}. 
The same sets and maps can be given for the structures $(X_{n},d_{n},d_{L,n})$: we 
will denote them with the subscript $n$.

For the transport problems in $(X_n,d_n)$ with measures $\mu_n$, $\nu_n$, we assume the following.

\begin{assumption}[Non degeneracy]
\label{A:punti}
The $d_{L,n}$-cyclically monotone plan $\pi_n$ satisfies Assumption \ref{A:NDE} for all $n \in \N$.
\end{assumption}

This allows to write the disintegration of $\mu_n$  w.r.t. 
the ray equivalence relation $R_{n}$:
\[
\mu_n = (g_n)_\sharp (r_n m_{n} \otimes \mathcal L^1) = \int g_n(y,\cdot)_\sharp (r_{n}(y,\cdot) \mathcal L^1) m_{n}(dy),
\]
with $f_{n\,\sharp}\mu_{n}=m_{n}$ and $r_n \in L^1(m_{n} \otimes \mathcal L^1)$.

\begin{lemma}\label{L:gnid}
If $(X_n,d_n,d_{L,n},\pi_n)$ converges to $(X,d,d_L,\pi)$ then the structures $(S_{n}\times \erre,\tilde d_{n}, \tilde d_{L,n}, \tilde \pi_{n} )$, where 
\[
\tilde d_{n} = d_{n} \circ (g_{n}\otimes g_{n}), \quad \tilde \pi_{n}=(g_{n}^{-1}\otimes g_{n}^{-1})_{\sharp} (\pi_{n}),
\quad \tilde d_{L,n}((y,t),(y',t')) = 
\begin{cases}
|t-t'| & y=y' \crcr
+\infty & y\neq y',
\end{cases}
\]
converges to $(X,d,d_{L},\pi)$.
\end{lemma}
\begin{proof}
It is enough to observe that $\pi_{n}(G_{n})=1$, $\tilde d_{L,n} = d_{L,n} \circ (g_{n}\otimes g_{n})$ on $G_{n}$ and to replace the map $\ell_{n}$
with the map $\ell_{n}\circ g_{n}$.
\end{proof}

By Lemma \ref{L:gnid}, in the following we  assume that the ray map $g_{n}$ is the identity map.

The next assumption is the fundamental one.

\begin{assumption}[Equintegrability]
\label{A:equi}
The $L^{1}$-functions  $r_n$ are equintegrable w.r.t. the measure $m_n \otimes \mathcal L^1$: 
\[
\forall \ve>0 \ \exists \delta>0 \bigg( (m_n \otimes \mathcal L^1)(A) < \delta \ \Rightarrow \ \int_A r_n m_n \otimes \mathcal L^1 < \ve \bigg).
\]
\end{assumption}

From now on we will assume that $(X_n,d_n,d_{L,n}, \pi_n) \to (X,d,d_L,\pi)$ in the sense of Definition \ref{D:mgh},
$(X_n,d_n,d_{L,n}, \pi_n)$ verifies Assumption \ref{A:punti} and Assumption \ref{A:equi}. 

Our aim is to prove that the structure $(X,d,d_L,\pi)$ satisfies Assumption \ref{A:NDE}, which is equivalent to the fact that
the marginal probabilities of the disintegration of $\mu$ w.r.t the ray equivalence relation $R$ are a.c. w.r.t. $\haus^{1}$.

The next lemma shows that in order to obtain our purpose we can perform some reductions without losing generality.
We will write $\mu_{k}\nearrow \mu$ for $\mu_{k}\leq \mu_{k+1}$ and  $\mu = \sup_{k} \mu_{k}$.

\begin{lemma}\label{L:approx}
Let $\{\mu_{k}\}_{k\in\enne} \subset \mathcal{M}(X)$, $\mu_k \geq 0$, be such that $\mu_{k} \nearrow \mu$ and assume that
\[
\mu_{k}=g_{\sharp}( r_{k}m_{k}\otimes \mathcal{L}^{1}), \quad r_{k}\geq0,
\]
where $g$ is the ray map on $\mathcal{T}$. Then there exist $m \in \mathcal{P}(X)$, $r \in L^{1}(m\otimes \mathcal{L}^{1})$, $r\geq0$ such that the same formula holds for $\mu$: 
\[
\mu=g_{\sharp}( r m \otimes \mathcal{L}^{1}).
\]
\end{lemma}

\begin{proof}
Since $\int r_{k}(y,t)dt=1$ it follows that $P_{1\,\sharp}(r_{k}m_{k}\otimes \mathcal{L}^{1})= m_{k}$ and therefore $m_{k} \nearrow m$ with 
$m=f_{\sharp}\mu$ (recall that $f$ is a section for the ray equivalence relation $R$).
The convergence $\mu_{k}\nearrow \mu$ yields
\[
\bigg( r_{k}\frac{dm_{k}}{d m} \bigg) m\otimes \mathcal{L}^{1} \nearrow \zeta,
\]
where $\mu=g_{\sharp}\zeta$. We conclude $\zeta =r m\otimes \mathcal{L}^{1}$ with $r:= \sup_{k} r_{k}\frac{dm_{k}}{d m}$.
\end{proof}

A first reduction is given by the following lemma.

\begin{lemma} \label{L:ptdist}
To prove that there exist $m \in \mathcal{P}(X)$, $r \in L^{1}(m\otimes \mathcal{L}^{1})$, $r\geq0$ such that 
\[
\mu=g_{\sharp}( r m \otimes \mathcal{L}^{1}),
\]
we can assume w.l.o.g. that there exist  $\bar x, \bar y \in X$ and $q \geq0$ such that 
\[
\pi \Big( \Big\{ (x,y): d(\bar x,\bar y) > 8q, d(x,\bar x), d(y,\bar y) \leq q \Big\} \Big) =1.
\]
Moreover the $d_{L}$-cyclically monotone set $\Gamma$ and the set of oriented transport rays $G$ can be assumed to be compact subsets of $X\times X$.
\end{lemma}

\begin{proof}
{\it Step 1.}
Since $\pi(\{x=y \})=0$ we can assume that $\Gamma \cap \{x=y \}= \emptyset$. 
Take two dense sequences $\{x_{i}\}_{i\in \enne} \subset X$, $\{q_{i} \}_{i\in \enne}\subset \erre^+$ and consider the family of closed sets
\[ 
\Gamma_{ijk}:=\Big\{ (x,y): d( x_{i},x_{j}) \geq 8q_{k}, d(x,x_{i}), d(y,x_{j}) \leq q_{k} \Big\}. 
\]  
Then $\Gamma_{ijk}$ is a countable covering of $X\times X \setminus \{x=y\}$. 

Suppose now to have proven that for all $\mu_{ijk} = P_{1\,\sharp}( \pi\llcorner_{\Gamma_{ijk}})$ the disintegration formula holds with 
$\mathcal H^{1}$-a.c. marginal probabilities, then the same $\mathcal H^1$-a.c. property is true if we replace $\Gamma_{ijk}$ with the finite union of 
sets $\Gamma_{i'j'k'}$.

Define
\[ 
\tilde \Gamma_{m} := \bigcup_{n < m} \Gamma_{i_{n}j_{m}k_{m}},
\]
where
\[
\N \ni m \mapsto (i_m,j_m,k_m) \in \N^3
\]
is a bijective map, and consider $\mu_{m} = P_{1\,\sharp}(\pi\llcorner_{\tilde \Gamma_{m}})$, then $\{\mu_{m}\}_{m\in\enne}$ 
verifies the hypothesis of Lemma \ref{L:approx}.

{\it Step 2.} It remains to show how to construct the approximating structure $\tilde \pi_{n} \in \mathcal{P}(X_{n})$ 
converging in the sense of Definition \ref{D:mgh} to $\pi \llcorner_{\Gamma_{ijk}}$.
Since $\Gamma_{ijk}$ is closed, there exists a sequence
$\phi_{l} \in C_{c}(X\times X,[0,1])$ such that $\phi_{l}\searrow \chi_{\Gamma_{ijk}}$. 
Now $\phi_{l} \pi \searrow \pi\llcorner_{\Gamma_{ijk}}$ as $l\to +\infty$ and  
\[
\phi_{l} (\ell_{n}\otimes \ell_{n} )_{\sharp}\pi_{n} \weak \phi_{l} \pi.
\]
Hence there exists a subsequence $\{\phi_{l_i} (\ell_{n_i}\otimes \ell_{n_i} )_{\sharp}\pi_{n_i} \}_{i \in \N}$ satisfying \eqref{E:conMHD0} with weak limit $\pi\llcorner_{\Gamma_{ijk}}$. If one defines
\[
\tilde \pi_{i} = \big( \phi_{l_i} \circ \ell_{n_i} \big) \pi_{n_i},
\]
then it is straightforward to show that $(X_i,d_i,d_{L,i},\tilde \pi_i)$ converges to $(X,d,d_L,\pi\llcorner_{\Gamma_{ijk}})$ in the sense of Definition \ref{D:mgh}.

{\it Step 3.}
Since Remark \ref{R:lsccase1} yields that $\Gamma$, $G$ are $\sigma$-compact, let $\Gamma= \cup_{k}\Gamma_{k}$, $G=\cup_{k}G_{k}$
with $\Gamma_{k}$, $G_{k}$ compact and consider $\pi\llcorner_{\Gamma_{k}}$.
The same reasoning done in Step 1 and Step 2. yields that it is enough to prove the a.c. of disintegration for $\pi\llcorner_{\Gamma_{k}}$.
%
\end{proof}

Therefore from now on we will assume that $\pi$ is concentrated on the set
\[ 
\Big\{ (x,y): d(\bar x,\bar y) > 8q, d(x,\bar x), d(y,\bar y) \leq q \Big\}. 
\]
Using the same reasoning of Lemma \ref{L:ptdist} one can also prove the following.
\begin{lemma} \label{L:cpt}
We can assume w.l.o.g. that the sets $A_{n}\subset X_{n}$ are compact and the maps $\ell_{n}:A_{n} \to X$ are continuous.
Moreover $\ell_{n}(A_{n})$ converges in Hausdorff distance to a compact set $K$ on which $\mu$ and $\nu$ are concentrated.
\end{lemma}

\begin{proof}
By Lusin Theorem and inner regularity of measures it follows that there exist $B_{n}\subset A_{n}$ such that
\begin{itemize}
\item $A_{n}\setminus B_{n}$ is compact;
\item $\mu_{n}(B_{n})\leq 1/n$;
\item the map $\ell_{n}: A_{n}\setminus B_{n} \to X$ is continuous.
\end{itemize}
To prove the first part of the claim just observe that $(\ell_{n} \otimes \ell_{n} )_{\sharp} \pi_{n}\llcorner_{A_{n}\setminus B_{n} \times A_{n}\setminus B_{n}} \weak \pi$. 

The second part of the statement can be proven following the line of the second part of the proof of the Proposition \ref{P:bddmono}.
\end{proof}


By Lemma \ref{L:cpt}  it is straightforward that for all $n$ great enough we have
\[ 
\big(\ell_{n\,\sharp}\mu_{n}\big)(B_{2q}(\bar x) )= \big(\ell_{n\,\sharp}\nu_{n}\big)(B_{2q}(\bar y) ) =1. 
\]


\begin{lemma}\label{L:rette}
We can assume that the measure $m_{n}$ is concentrated on a compact subset of 
\[
\Big\{ y \in S_{n} :  \exists t,s>\delta : (y,-s),(y,0),(y,t) \in P_{12}( \textrm{\rm graph} (\ell_{n}) ) \Big\}
\]
for some fixed $\delta>0$ and 
\begin{equation}\label{E:rette}
\mu_{n}( S_{n}\times (-\infty,0]) =1, \quad \nu_{n}( S_{n}\times (4q,+\infty)) =1.
\end{equation}
\end{lemma}

\begin{proof}

{\it Step 1.} Defining
\[
A_\delta := \Big\{ (y,t)\in S_{n} \times \erre :  |a(y)-t|<\delta \Big\},
\]
by Fubini Theorem 
\[ 
m_{n}\otimes\mathcal{L}^{1}(A_{\delta})
\leq \delta,
\]
hence by Assumption \ref{A:equi}, for any $\ve>0$ we can choose $\delta>0$ such that  
$r_{n} m_{n} \otimes \mathcal{L}^{1}(A_{\delta})<\ve$ for al $n \in \N$.

Therefore we can assume that  $r_{n} m_{n} \otimes \mathcal{L}^{1}$ is concentrated on a compact subset $B_{n}$
of $\ell_{n}^{-1}(\bar B (\bar x, \frac{3}{2}q) ) \setminus A_{\delta}$.

{\it Step 2.}
Define the u.s.c. selection of $B_{n}$ $t_{n}: S_{n} \to \erre $ in the following way:
\[
y \mapsto t_{n}(y):= \max \big\{ t \in \erre, (y,t) \in B_{n} \big\}.
\]
By removing a set of arbitrarily small measure we can assume that for all $y \in P_{1}(\gr (t_{n}))$ there exists $t>4q$ such that 
\[
(y,t_{n}(y)+t) \in \ell_{n}^{-1}\bigg( \bar B\bigg(\bar y, \frac{3}{2}q \bigg) \bigg).
\]

{\it Step 3.}
The Borel transformation 
\begin{align*}
B_{n} \ni (y,t) \mapsto (y, t-t_{n}(y))
\end{align*}
maps $m_{n}\otimes \mathcal{L}^1$ into itself and in the new coordinates the section $S_{n}$ satisfies the first part of the claim.

By the definition of $G_{n}$ and $\mu_{n}\perp\nu_{n}$ it follows that $\mu_{n}$ and $\nu_{n}$ satisfy \eqref{E:rette}, see Remark \ref{R:monot}.
\end{proof}

%
%
%
%
%
%
%
%
%
%
%
%
%
%
%
%
%
%

Define the map
\[
\begin{array}{ccccc}
h_n &:& S_{n} \times \R &\to& X \times \erre \crcr
&& (y,t) &\mapsto& (\ell_{n}  (y,0),t) .
\end{array}
\]
and the measure $h_{n\,\sharp} ( r_{n}m_{n}\otimes \mathcal{L}^{1} ) = \tilde r_{n} \tilde m_{n} \otimes \mathcal{L}^{1}$, 
with $\tilde m_{n} = \ell_{n}(\cdot,0)_{\sharp} m_{n}$.

\begin{lemma} \label{L:tight} 
The family of measures  $\{\tilde r_{n} \tilde m_{n} \otimes \mathcal{L}^{1}\}_{n\in \enne} \subset \mathcal{P}( \ell_{n}(S_{n}\times \{0\}) \times \erre)$ is tight and $\tilde r_{n}$ 
is equintegrabile w.r.t.  $\tilde m_{n} \otimes \mathcal{L}^{1}$.
\end{lemma}

\begin{proof}
Performing the same calculation of \eqref{E:costo}  
\[
C\geq \int d_{L,n} \pi_{n} = \int s  \nu_{n} - \int s \tilde r_{n} \tilde m_{n}\otimes \mathcal{L}^{1}.
\]
From \eqref{E:rette}, Lemma \ref{L:rette}, it follows that $s \leq 0$, $\tilde r_{n} \tilde m_{n}\otimes \mathcal{L}^{1}$-a.e..
Hence $\tilde r_{n} \tilde m_{n}\otimes \mathcal{L}^{1} \in \mathcal{P}( \ell_{n}(S_{n}\times \{0\}) \times (-\infty,0]))$ and 
\[
0 \leq - \int s \tilde r_{n} \tilde m_{n}\otimes \mathcal{L}^{1} \leq C,
\]
therefore $\tilde r_{n} \tilde m_{n}\otimes \mathcal{L}^{1}$ is tight. Recall in fact that $\{S_n\}_{n \in \N}$ is a precompact sequence w.r.t. the Hausdorff distance by Lemma \ref{L:cpt}.

The equintegrability is straightforward: 
\[
\int_{A} \tilde r_{n} \tilde m_{n}\otimes \mathcal{L}^{1} = \int_{(h_{n})^{-1}(A)} r_{n} m_{n}\otimes \mathcal{L}^{1}
\]
and $m_{n}\otimes \mathcal{L}^{1}((h_{n})^{-1}(A)) = \tilde m_{n}\otimes \mathcal{L}^{1} (A)$.
\end{proof}

Consider the following measure
\[
\zeta_n := (h_{n}, \ell_{n})_\sharp ( r_n  m_n \otimes \mathcal L^1) \in \Pi \Big( \tilde r_n \tilde m_n \otimes \mathcal L^1, 
(\ell_n )_{\sharp} (\mu_{n}) \Big) \in \mathcal{P}(X\times \erre \times X).
\]

\begin{proposition}
\label{P:graphnn}
Up to subsequences, $\zeta_n \rightharpoonup \zeta$, where $\zeta \in \Pi(r m \otimes \mathcal L^1, \mu)$ is supported on a Borel graph $h : \mathcal T \times \R \to \mathcal T_e$ such that $t \mapsto h(y,t)$ is the $d_L$ $1$-Lipschitz curve $R(y)$ for $m$-a.e. $y \in X$.
\end{proposition}

\begin{proof}
{\it Step 1.}
The convergence to the correct marginals is a consequence of \eqref{E:conMHD0}
\[
(P_2)_\sharp \zeta_n = (\ell_n \circ g_n)_\sharp (r_n m_n \otimes \mathcal L^1) = 
(\ell_n)_\sharp \mu_n \rightharpoonup \mu, 
\]
and by Lemma \ref{L:tight} 
\[
(P_1)_\sharp \zeta_n= \tilde r_n  \tilde m_n \otimes \mathcal L^1 \rightharpoonup r m \otimes \mathcal L^{1}.
\]

{\it Step 2.} Since up to subsequence $\zeta_{n} \weak \zeta$, using the same technique of Lemma \ref{P:bddmono}, 
we can assume that $K_{n}:= (h_{n},\ell_{n})(S_{n}\times \erre)$ is compact  and   $d_{H}(K_{n}, \gr(h)) \to 0$ where $\gr(h)$ is a compact set supporting $\zeta$ and $h$ is the associated multivalued function.

{\it Step 3.}
Let $(y,t,x) \in \gr(h)$, then by the definition of convergence in the Hausdorff metric, 
there exists a sequence $(\ell_{n}(y_{n},0), t_{n},\ell_{n} (y_{n},t_{n})) \to (y,t,x)$. 
Hence from
\[ 
d_{L,n}\big( (y_{n},t_{n}),(y_{n},0) \big)=|t_{n}| \to |t|,
\]
we deduce by \eqref{E:conMHD1} that $d_{L}(x,y)=|t|$. In particular this implies that if $t=0$ then $x=y$.

{\it Step 4.} Let $(y,t,x), (y,t',x') \in \gr(h)$ with $t < 0$ and $t' > 0$. 
Again by the Hausdorff convergence there exist two sequences satisfying
\[
\big(\ell_{n}(y_{n},0), t_{n},\ell_{n}(y_{n},t_{n})\big) \to(y,t,x), \quad\big(\ell_{n}(y'_{n},0), t'_{n}, \ell_{n}(y'_{n},t'_{n})\big) \to(y,t',x').
\]
Since $d_{L}(y,y)=0$, from \eqref{E:conMHD1} we deduce
\[
d_{L,n}\big((y_{n},0),(y'_{n},0)\big) \to 0
\]
hence by the definition of $d_{L,n}$, for $n$ great enough $y_{n}=y'_{n}$.
Therefore
\[
d_{L}(x,x')= \lim_{n \to + \infty} d_{L,n}((y_{n},t_{n}),(y'_{n},t'_{n})) = |t|+t',
\]
and by Step 3 we conclude that $d_{L}(x,x')=d_{L}(x,y)+d_{L}(y,x')$.

{\it Step 5.} Let $(y,t,x) \in \gr(h)$: we now show that
\[
t\geq0\ \Rightarrow\ (y,x)\in G, \quad -t\geq0\ \Rightarrow\ (y,x)\in G^{-1}.
\] 
We will prove only the first implication for $t>0$.
Since following Lemma \ref{L:cpt} we can take $G_{n}$ compact such that
\begin{enumerate}    
\item $(\ell_{n}\otimes \ell_{n})(G_{n}) \to \hat G$ in the Hausdorff metric;
\item $\hat G \subset G$,
\end{enumerate}
it is enough to show that there exists a sequence $(\ell_{n}(y_{n},0),t_{n},\ell (y_{n},t_{n}))\to (y,t,x)$
so that $(y_{n},x_{n}) \in G_{n}$ for all $n$, but this last implication is straightforward.

{\it Step 6.} We next show that for any $y \in P_{1}(\gr (h))$ there exist $t_{-},t_{+} \geq \delta$  
and $x_{-},x_{+}$ such that $(y,-t_{-},x_{-}),(y,t_{+},x_{+}) \in \gr(h)$.
In fact 
we recall that for all $y_{n} \in S_{n}$ there exist $t_{-,n},t_{+,n}\geq \delta$,
for some strictly positive constant $\delta$,  
such that 
\[
\big( (y_{n}, -t_{-,n}),(y_{n},0) \big), \big((y_{n},0),(y_{n},t_{+,n})\big) \in G_{n}.
\] 
Hence chose $y_{n} \in S_{n}$ such that $\ell_{n}(y_{n})\to y$ 
and pass to converging subsequences to obtain the claim. 

{\it Step 7.} Since for $y \in P_{1}(\gr(h))$ there exist  $x,x'$ such that $(x,y), (y,x')\in G \setminus \{x=y\}$, then $(x,x')\in G$, $y \in \mathcal{T}$
and $h$ is single valued. 
The same computation of Point 5 yields that
\[
\big\{ (y,h(t,y)), t\geq 0\big\} \cup \big\{ (h(t,y),y), t \leq 0\big\}\subset G,
\]
and from this it follows that $h(y,\erre)\subset R(y)$. 

Again from Point 5 one obtains that $d_{L}(y,h(t,y))=|t|$ and therefore
$t\mapsto g^{-1}(h(y,t))= g^{-1}(y)+t$.
\end{proof}

\begin{theorem}
\label{T:aproxxgg}
Let  $(X_{n},d_{n},d_{L,n},\pi_{n}) \to (X,d,d_{L},\pi)$  and $(X_{n},d_{n},d_{L,n},\pi_{n})$, $n \in \N$,
verifies Assumption \ref{A:punti} and Assumption \ref{A:equi}. Then the marginal measure $\mu = P_{1\,\sharp}(\pi)$ satisfies Assumption \ref{A:NDE}.
\end{theorem}

\begin{proof}
The measure $\zeta$ constructed in the Proposition \ref{P:graphnn} satisfies the hypothesis of Proposition \ref{P:finalapr}. 
Therefore the marginal probabilities of the disintegration of $\mu$ are absolutely continuous with respect to $\haus^{1}$ and therefore $\mu$  
verifies Assumption \ref{A:NDE}.
\end{proof}

\begin{remark}
\label{R:reguadd}
As in Remark \ref{R:morereg}, if we know more regularity of the disintegrations for the approximating problems, we can pass them to the limit. Here the key observation is that geodesics converge to geodesics, so uniform continuous functions on them converge pointwise to continuous functions.
\end{remark}

A special case is when $d_L = d$: a natural approximation is by transport plans where $\nu$ is atomic, with a finite number of atoms. 
This case can be studied with more standard techniques, we refer to the analysis contained in \cite{biaglo:HJ1}.

%


\section{Applications}
\label{S:mcp}

In this section we recall the definition of Measure Contraction Property ($MCP$) and then we 
prove that for a metric measure space $(X,d,\eta)$ satisfying $MCP$, the Monge minimization problem with
marginal measures 
$\mu$ and $\nu$ with $\mu \ll \eta$ and cost $d$ admits a solution. We show moreover that the hypotheses of Corollary \ref{C:regudual} hold, and if $\supp \mu$ and $\supp \nu$ are at positive distance then the assumptions of Lemma \ref{L:normalcurr} are satisfied, i.e. the current $\dot g$ is normal. The main reference for this section is \cite{ohta:mcp}.

From now on $d=d_{L}$ and $\eta \in \mathcal{M}^{+}(X)$ is a locally finite measure on $X$. 
Since $d_{L}=d$ there exists a Lipschitz function $\f$ potential for the transport problem: 
hence in the following we will set
\[
\Gamma = \Gamma' = G = \Big\{ (x,y) \in X\times X : \f(x)-\f(y) = d(x,y) \Big\},
\]
where $\phi$ is a potential for the transport problem.

Let $H$ be the set of all geodesics: we regard $H$ as a subset of $\textrm{Lip}_1([0,1],X)$ with the uniform topology. 
Define the evaluation map $e_{t}(\gamma)$ by
\begin{equation}\label{E:evalua}
\begin{array}{ccccc}
e &:& [0,1] \times H &\to& X \crcr
&& e_t(\gamma) &\mapsto& \gamma(t)
\end{array}
\end{equation}
It is immediate to see that $e_t(\gamma)$ is continuous.

A \emph{dynamical transference plan} $\Xi$ is a Borel probability measure on $H$, and 
the path $\{ \xi_{t}\}_{t \in [0,1]} \subset \mathcal{P}^{2}(X)$ given by $\xi_{t}= (e_{t})_{\sharp}\Xi$
is called \emph{displacement interpolation} associated to $\Xi$. 
We recall that $\mathcal{P}^{2}(X)$ is the set of Borel probability measures $\xi$
satisfying $\int_{X}d^{2}(x,y)\xi(dy)< \infty$ for some (and hence all) $x \in X$.

Define for $K\in \erre$ the function $s_{K} : [0,+\infty) \to \erre$ (on $[0,\pi/\sqrt{K})$ if $K>0$) 
\begin{equation}\label{E:sk}
s_{K}(t):= 
\begin{cases}
(1/\sqrt{K})\sin(\sqrt{K}t) & \rm{if}\ K>0, \crcr
t & \rm{if}\ K=0, \crcr
(1/\sqrt{-K})\sinh(\sqrt{-K}t) &\rm{if}\ K<0,
\end{cases}
\end{equation}
and let $N \in \enne$.

\begin{definition}\label{D:mcp}
A metric measure space $(X,d,\eta)$ is said to satisfies the $(K,N)$-\emph{measure contraction property} ($MCP(K,N)$)
if for every point $x \in X$ and $\eta$-measurable set $A \subset X$ with $\eta(A)>0$
there exists a displacement interpolation $\{\xi_{t}\}_{t\in [0,1]}$ associated to 
a dynamical transference plan $\Xi = \Xi_{x,A}$ satisfying the following: 
\begin{enumerate}
\item We have $\xi_{0}=\delta_{x}$ and $\xi_{1}=\eta(A)^{-1} \eta_{\llcorner A}$; 
\item \label{E:mcp} for $t \in [0,1]$ 
\[ 
\eta \geq (e_{t})_{\sharp} \bigg( t \bigg\{  \frac{s_{K}(t d(x,\gamma(1)) )}{s_{K}( d(x,\gamma(1)) )} \bigg\}^{N-1} \eta(A) \Xi  \bigg),
\]
where we set $0/0=1$.
\end{enumerate}
\end{definition}

From now on we will assume the metric measure space $(X,d,\eta)$ to satisfies $MCP(K,N)$ for some $K\in \erre$ and $N\in \enne$.
Recall that $MCP(K,N)$ implies that $(X,d)$ is locally compact, Lemma 2.4 of \cite{ohta:mcp}.

The strategy to prove Assumption \ref{A:NDE} for any $d$-cyclically monotone plan is the following: 
first we prove that for any $\pi \in \Pi(\mu,\delta_{x})$ $d$-monotone with $x$ arbitrary, 
the marginal probabilities of $\eta$
obtained by the disintegration induced by the ray map $g$ are absolutely continuous w.r.t. $\haus^{1}$ and their densities 
satisfy some uniform estimates. 
Then we observe that these estimates hold true also for any $\pi\in \Pi(\mu,\sum_{i\leq I} c_{i}\delta_{x_{i}})$ $d$-monotone. 
Finally we show that the same estimates hold for general transference plans and therefore 
we deduce that the densities of the marginals obtained by disintegrating $\eta$ w.r.t. any $d$-monotone plan $\pi$
are absolutely continuous w.r.t. $\haus^{1}$. 

By Lemma \ref{L:approx}, it is enough to assume that there exists $K_{1}, K_{2} \subset X$ compact set,
such that  $\mu(K_{1}) =\nu(K_{2})=1$ and $d_{H}(K_{1},K_{2})< + \infty$. 
Hence we can assume that $\diam(X)< + \infty$ and $\eta(X)=1$.



\begin{lemma}
\label{L:punto}
Consider $\bar x\in X$ and let $\pi \in \Pi(\mu,\delta_{\bar x})$ be the unique $d$-cyclically monotone transference plan. 
Then $\eta$ and the optimal flow induced by $\pi$ verify Assumption \ref{A:NDE}: 
more precisely, $\eta =   g_{\sharp} ( q m\otimes \mathcal{L}^{1} )$
and the density $q$ satisfies the estimate
\begin{equation}\label{E:marg} 
q(y,t) \geq \bigg\{   \frac{s_{K}( d(g(y, t) , \bar x))}{s_{K}( d(g(y,s) , \bar x)) } \bigg\}^{N-1} q(y,s)
\end{equation}
for $m$-a.e. $y \in S$, for any $s\leq t$ such that $d(g(y,t),x)>0$.
\end{lemma}

We recall that $S$ is a section for the ray equivalence relation. Since $\mu \ll \eta$, \eqref{E:marg} implies that $\mu = g_{\sharp}( r m \otimes \mathcal{L}^{1})$ with $r \leq q$.

\begin{proof}

First observe that the potential for the transport problem is 
\[ 
\f(x) : = \f(\bar x) + d(x,\bar x), 
\]
so that the geodesics used by $\pi$ are exactly $H_{\bar x}:= H \cap e_{0}^{-1}(\bar x)$, 
in the sense that  
\[ 
G = \Big\{ \big(\gamma(1-s),\gamma(1-t)\big), s \leq t, \gamma \in H_{\bar x} \Big\}.
\]

{\it Step 1.} 
We first prove that the set of initial points $A=a(X)$ has $\eta$-measure zero. 
Suppose by contradiction that $\eta(A)>0$ and let $\Xi_{\bar x,A}$ be the dynamical transference plan associated:
we can assume that $\Xi_{\bar x,A}$ is supported on the set 
$H_{\bar x,A}:= H_{\bar x} \cap e_{1}^{-1}(A)$.
Then the evolution of $A$ by the geodesics of $H_{\bar x, A}$ can be defined as
\[
A^{s}: = e_{1-s}(H_{\bar x,A}).
\]
By Condition \ref{E:mcp} of Definition \ref{D:mcp} and the fact that $e_{1-s}^{-1}(A^{s})=H_{\bar x, A}$
\begin{equation}\label{E:acaso}
\eta(A^{s}) \geq \eta(A) \int_{H_{\bar x,A}} 
(1- s) \bigg\{  \frac{s_{K}((1- s) d(\bar x,\gamma(1)) ) }{s_{K}( d(\bar x,\gamma(1)) )} \bigg\} ^{N-1} \Xi_{x,A}(d\gamma) > 0,
\end{equation}
for all $s \in [0,1)$.
Since all $A^{s}$ are disjoint being the space non branching, it follows that $\eta(A)=0$.

{\it Step 2.} For $A$ with $\eta(A)>0$ let $\Xi_{\bar x,A}$ be the dynamical transference plan concentrated on a set 
$H_{\bar x,A}: =H_{\bar x} \cap e_{1}^{-1}(A)$.
Denote as before $A^{s}: = e_{1-s}( H_{\bar x, A})$.

Observe that since the set initial point has $\eta$-measure zero, we can disintegrate $\eta$ w.r.t. the ray equivalence relation:
using the disintegration formula $\eta = \int \eta_{y} m(dy)$ the same estimate as in \eqref{E:acaso} yields
\[
\int \eta_{y}(A^{s}) m(dy) \geq \int \eta_{y}(A)m(dy) \bigg( \int_{ H_{\bar x,A}} (1- s) \bigg\{  \frac{s_{K}((1- s) d(\bar x,\gamma(1)) ) }{s_{K}( d(\bar x,\gamma(1)) )} \bigg\} ^{N-1} \Xi_{\bar x,A}(d\gamma) \bigg).
\]
By evaluating the above formula on sets of the form $A=g( S \times [t_{1},t_{2}] )$, where $g$ is the ray map
such that $g(y,0)=\bar x$ for all $y$, gives 
\begin{align*}
\int_{S} \eta_{y} \big( g(y,[t_{1},t_{2}] (1-s)) \big) m(dy) \geq &~ \int_{S} \eta_{y} \big( g(y,[t_{1},t_{2}]) \big) m(dy) \crcr
&~ \qquad \cdot \bigg( \int_{ H_{\bar x,A}} (1- s) \bigg\{  \frac{s_{K}((1- s) d(\bar x,\gamma(1)) ) }{s_{K}( d(\bar x,\gamma(1)) )} \bigg\} ^{N-1} \Xi_{\bar x,A}(d\gamma) \bigg) \crcr 
\geq &~ \int_{S} \eta_{y}\big(g(y,[t_{1},t_{2}])\big) m(dy)     \min_{c\in [t_{1},t_{2}]} \bigg\{  (1-s) \frac{s_{K}((1- s) |c| ) }{s_{K}( |c| )} \bigg\} ^{N-1}.
\end{align*}
and therefore for $m$-a.e. $y$ and every $t_{1},t_{2}$
\begin{equation}\label{E:evo}
\eta_{y} \big( g(y,[t_{1},t_{2}] (1-s)) \big) \geq \eta_{y} \big( g(y,[t_{1},t_{2}]) \big) \min_{c\in [t_{1},t_{2}]} \bigg\{  (1-s) \frac{s_{K}((1- s) |c| ) }{s_{K}( |c| )} \bigg\} ^{N-1}.
\end{equation}

{\it Step 3.}
For $t_{1}<0$ consider the family of disjoint open sets 
\[
t_{1} \bigg(  1 - \frac{k}{2n},  1 - \frac{k+1}{2n}\bigg), \quad k = \{0,1,\dots, n-1 \}.
\]
The above estimate and the fact that $\eta_{y}$ is probability yield  
\[ 
\eta_{y} \bigg\{  g\bigg(y,  t_{1}\bigg(1,1-\frac{1}{2n}\bigg) \bigg) \bigg\} \leq \frac{1}{n} 
\max_{c\in [t_{1},t_{1}/2]} \bigg\{  2 \frac{s_{K}( |c| ) }{s_{K}( 2|c| )} \bigg\} ^{N-1}.
\]
Hence $\eta_{y} = q \haus^{1}\llcorner_{g(y,\erre)}$ and $q$ satisfies \eqref{E:marg}.
\end{proof}

\begin{lemma}\label{L:piupunti}
Let $\pi\in \Pi(\mu,\sum_{i \leq I} c_{i}\delta_{x_{i}})$ $d$-cyclically monotone. 
Then the conditional probabilities of the disintegration of $\eta$ w.r.t. the ray equivalence relation induced by $\pi$ 
are absolutely continuous w.r.t. $\haus^{1}$ and  the density $q(y,\cdot)$ satisfies
\[
q(y,t) \geq \bigg\{   \frac{s_{K}( d(g(y, t) , b(y)))}{s_{K}( d(g(y,s) , b(y))) } \bigg\}^{N-1} q(y,s).
\]
\end{lemma}

\begin{proof}
Let $\f$ be a potential for the transport problem with marginal $\mu$ and $\nu$.
Define 
\[
E_{i}:= \bigg\{z \in \mathcal{T}_{e} : \varphi(z)-\varphi(x_{i})=d(z,x_{i})\bigg\}.
\]
Now each $E_{i}$ is sent by the optimal geodesic flow to $x_{i}$, so we can perform exactly the same calculations done in Lemma \ref{L:punto}. Indeed $E_{i}\cap E_{j} \subset a(X)$ which has $\eta$-measure zero, 
$\eta_{\llcorner E_{i}}$ verifies \eqref{E:mcp} of Definition \ref{D:mcp} along the geodesic flow connecting $E_{i}$ to $x_{i}$. 
\end{proof}

Given $\tilde H \subset \textrm{Lip}_1([0,1],X)$ a set of geodesics and $A\subset X$, define 
\begin{equation}\label{E:evodue}
A^{s,\tilde H}:= e_{1-s}( e^{-1}_{1}(A) \cap \tilde H).
\end{equation}

\begin{lemma}\label{L:fine}
Assume that there exists two compact sets $K_{1}, K_{2}\subset X$ such that 
\begin{enumerate} 
\item $\mu(K_{1})=\nu(K_{2})=1$; 
\item  there exist $0<a \leq b < + \infty$ such that  
\[ 
a =  \min_{x_{1} \in K_{1}, x_{2}\in K_{2}} d(x_{1},x_{2})  \leq  \max_{x_{1} \in K_{1}, x_{2}\in K_{2}} d(x_{1},x_{2})    ;
\]
\item $K_{2}$ is a section of $R$.
\end{enumerate}
Then if 
\begin{equation}\label{E:geodue}
H(G) : = \Big\{ \gamma \in H : \exists y \in K_{2} \Big( \f(\gamma(0)) - \f(\gamma(1))= d(\gamma(0),\gamma(1))\, \wedge \,\gamma(0)=y  \Big) \Big\},
\end{equation}
where $\f$ is the potential for the transport problem with marginal $\mu$ and $\nu$, then
\[
\eta(K_{1}^{s,H(G)}) \geq \eta(K_{1})     \min_{a \leq c \leq b}  \bigg\{  (1- s) \frac{s_{K}((1- s) c ) }{s_{K}( c )} \bigg\} ^{N-1}.
\]
\end{lemma}

\begin{proof}
{\it Step 1.} 
It follows directly from Lemma \ref{L:piupunti} that
the statement holds for $\nu = \sum_{i\leq I}c_{i} \delta_{y_{i}}$.

We thus consider the sequence of approximating problem constructed as follows: 
let $\{y_{i}\}_{i \in \enne}$ be a dense sequence in $K_{2}$ and for $I \in \enne$
define

$$
\f_{I} (x) : = \min \Big\{ \f(y) + d(x,y), y \in  \{y_{1},\cdots, y_{I}\} \Big\},
$$

$$
E_{i,I} := \Big\{ x \in X : \f_{I}(x) -\f_{I}(y_{i}) = d(x,y_{i}),  i \leq I \Big\},
$$

$$
\nu_{I}= \sum_{i\leq I} c_{i,I} \delta_{y_{i}}, \quad \textrm{where} \quad c_{i,I}= \mu \bigg(E_{i,I}\setminus \bigcup_{j\neq i} E_{j,I} \bigg).
$$

Clearly $\f_{I}$ is a potential for the transport problem with marginal $\mu$ and $\nu_{I}$ and let 
$$
H(G_{I}) : = 
\Big\{ \gamma \in H :  \f_{I}(\gamma(0)) - \f_{I}(\gamma(1))= d(\gamma(0),\gamma(1))\, \wedge \,\gamma(0) \in 
\{y_{1},\cdots, y_{I}\}   \Big\}.
$$

{\it Step 2.} Observe that $K_{1}^{s,H(G)}$ is compact. 
In fact, since $K_{1}$ and $K_{2}$ are compact, 
$H(G)\cap e_{1}^{-1}(K_{1})$ is compact and 
since $e_{1-s}$ is continuous $K_{1}^{s,H(G)} = e_{1-s}(H(G))$ is compact.
For the same reasons the sets $K_{1}^{s,H(G_{I})}$ are compact.

{\it Step 3.} $K_{1}^{s,H(G_{I})}$ is contained in a compact set and $\f_{I} \to \f$ as $I \to + \infty$, so that
up to subsequences $K_{1}^{s,H(G_{I})}$ converges in Hausdorff distance to a compact subset of $K_{1}^{s,H(G)}$.
By the upper semicontinuity of Borel bounded measures with respect to Hausdorff convergence for compact sets
the claim follows.
%
%
%
%
%
\end{proof}

\begin{theorem}\label{T:regularity}
If $\pi\in \Pi(\mu,\nu)$ $d$-monotone then $\eta \llcorner_{\mathcal T_e} = g_{\sharp} (q m \otimes \mathcal{L}^{1})$, where $\mathcal T_e$ is the transport set with end points \eqref{E:TRe}, and for $m$-a.e. $y$ and $s\leq t$ it holds
\begin{equation}\label{E:bigreg}
\bigg\{   \frac{s_{K}( d(g(y, t) , b(y)))}{s_{K}( d(g(y,s) , b(y))) } \bigg\}^{N-1} \leq \frac{q(y,t)}{q(y,s)} \leq 
\bigg\{   \frac{s_{K}( d(g(y, t) , a(y)))}{s_{K}( d(g(y,s) , a(y))) } \bigg\}^{N-1} 
\end{equation}
\end{theorem}

\begin{proof}
{\it Step 1.}
We first show that the set of initial points has $\eta$-measure zero. 
In fact suppose by contradiction that $\eta(a(S))>0$, where $S$ is a section for the ray equivalence relation of $\pi$. 
Hence we can assume that $S$ and $a(S)$ are compact and at strictly positive distance. 

Applying Lemma \ref{L:fine} to the transport problem with marginals $\eta\llcorner_{a(S)}$ and $f_{\sharp}\eta$, 
where $f$ is the quotient map, it follows that $\eta(a(S))=0$.

{\it Step 2.} 
Since the initial points have $\eta$-measure zero, 
we can disintegrate $\eta \llcorner_{\mathcal T_e}$ w.r.t. the ray equivalence relation obtaining
$\eta \llcorner_{\mathcal T_e} = \int \eta_{y} m(dy)$. 
By a standard covering argument, it is enough to prove the statement on the set 
\[ 
D_{\ve}:=\big\{x: d(x,b(x)) \geq \ve\big\}.
\]
For any $0<\delta<\ve$ we can take the section $S$ compact such that $d(f(x),b(x))=\delta$, in particular we have $g(y,\delta)=b(y)$.

For $S'\subset S$ and $t_{1}<t_{2}$ consider $\eta\llcorner_{g^{-1} (S'\times [t_{1},t_{2}])}$.
Applying Lemma \ref{L:fine} with 
$$
\mu = \frac{\eta\llcorner_{g^{-1} (S'\times [t_{1},t_{2}])}}{  \eta(g^{-1} (S'\times [t_{1},t_{2}]))}, \quad
\nu =  f_{\sharp}\mu
$$
where $f$ is the quotient map for the ray equivalence relation $R$, it holds 
\begin{align*}
\int_{S'} \eta_{y} \big( g(y,[t_{1},t_{2}] (1-s)) \big) m(dy) \geq \min_{c\in [t_{1},t_{2}]} \bigg\{  (1-s) \frac{s_{K}((1- s) |c| ) }{s_{K}( |c| )} \bigg\} ^{N-1} \int_{S'} \eta_{y} \big( g(y,[t_{1},t_{2}]) \big) m(dy).
\end{align*}
As in Step 2 of the proof of Lemma \ref{L:punto}, the estimate \eqref{E:evo} holds for $m$-a.e. $y$ and every $t_{1}<t_{2}$ 
and we deduce 
\[
\bigg\{   \frac{s_{K}( d(g(y, t) , b(y))   -\delta )    }{s_{K}( d(g(y,s) , b(y))-\delta) } \bigg\}^{N-1} \leq \frac{q(y,t)}{q(y,s)}.  
\]
Letting $\delta \to 0$, we obtain the left hand side of \eqref{E:bigreg}.

{\it Step 3.}
The right hand side of of \eqref{E:bigreg} is obtained by the same procedure taking 
$$
F_{\ve}:= \big\{ x : d(d,a(x)) \geq \delta \big\} 
$$
and the section $S$ such that $d(y,a(y))=\delta$ for all $y \in S$.
\end{proof}

\begin{figure}
\label{Fi:mongemetric5}
\psfrag{q}{$q(y,t)$}
\psfrag{sa}{$q(y,\bar t) \left( \frac{s_K(t+d(y,a(y))}{s_K(\bar t+d(y,a(y))} \right)^{N-1}$}
\psfrag{sb}{$q(y,\bar t) \left( \frac{s_K(d(y,b(y))-t}{s_K(d(y,b(y))-\bar t} \right)^{N-1}$}
\psfrag{qyt}{$q(y,\bar t)$}
\psfrag{da}{$-d(y,a(y))$}
\psfrag{db}{$d(y,b(y))$}
\psfrag{t}{$\bar t$}
\centerline{\resizebox{12cm}{8cm}{\includegraphics{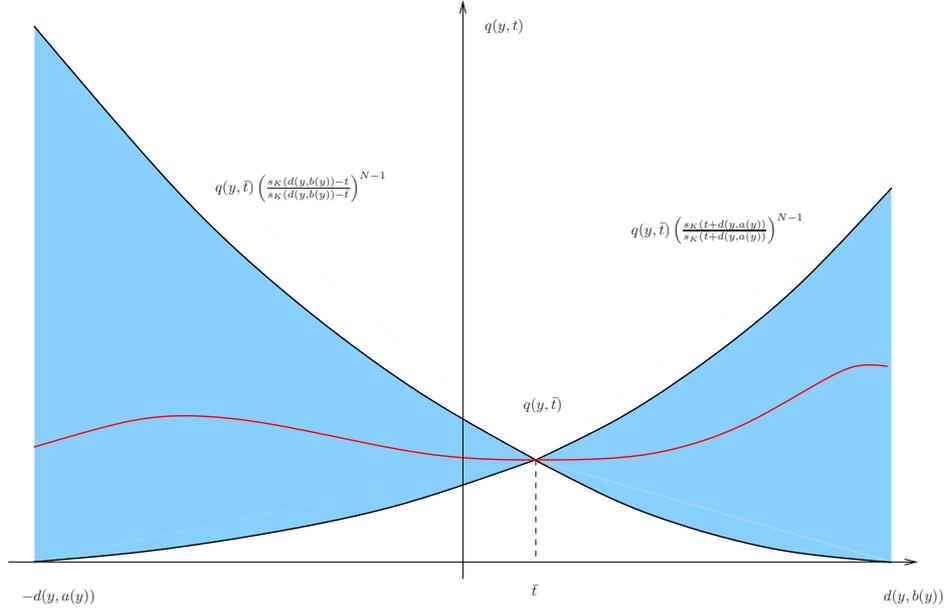}}}
\caption{The region where $q(y,t)$ takes values.}
\end{figure}

Since $\mu \ll \eta$, it follows that also the densities of the conditional probabilities of $\mu$ are absolutely continuous 
w.r.t. $\haus^{1}$, and therefore we have the following corollary.

\begin{corollary}\label{C:conclu}
Let $(X,d,\eta)$ satisfies $MCP(K,N)$, let $\mu,\nu \in \mathcal{P}(X)$ with $\mu \ll \eta$, then there exists a $\mu$-measurable map $T:X \to X$ such that 
$T_{\sharp}\mu =\nu$ and 
\[
\int d(x,T(x))\mu(dx) = \min_{\pi \in \Pi(\mu,\nu)} \int d(x,y) \pi(dxdy).
\]
\end{corollary}

We can obtain additional regularity of the conditional probabilities $\eta_y$ under $MCP(K,N)$: in particular we deduce that the conclusion of Corollary \ref{C:regudual} holds and if the support of $\mu$ and $\nu$ are compact sets with empty intersection the statements of Lemma \ref{L:normalcurr} and Remark \ref{R:abscurr} are true.

\begin{lemma}
\label{L:mhdregul}
The marginal densities
\[
\big( -d(a(y),y),d(y,b(y) \big) \ni t \mapsto q(y,t) \in \R^+
\]
are strictly positive Lipschitz continuous for $m$-a.e. $y \in S$, and for some constant $C>0$
\[
\TV \big( q(y,\cdot) \big) \leq \frac{C}{d(a(y),b(y))}.
\]
\end{lemma}

\begin{proof}
From \eqref{E:bigreg} it follows immediately that the function $q(y,t) > 0$ and Lipschitz continuous for $t \in (-d(y,a(y)),d(y,b(y)) )$ and 
$m$-a.e. $y$.
By differentiating it follows that 
\begin{equation}\label{E:der}
-(N-1)  \frac{s'_{K}( d(g(y, t) , b(y)))}{s_{K}( d(g(y,t) , b(y))) }  \leq \frac{q'(y,t)}{ q(y,t)} \leq 
(N-1)   \frac{s'_{K}( d(g(y, t) , a(y)))}{s_{K}( d(g(y,t) , a(y))) }.
\end{equation}
In particular $q(y,\cdot)$ is Lipschitz.

For notational convenience let us assume that $d(a(y),y)=d(y,b(y))=l$.
From \eqref{E:bigreg} one can prove that
\[ 
q(y,t) \geq q(y,0)  \cdot
\begin{cases}
\displaystyle \frac{s_{K}(l-t)}{s_{K}(l)}, & t\geq 0 \crcr
\displaystyle \frac{s_{K}(-l+t)}{s_{K}(-l)}, & t\leq 0
\end{cases}
\]
Since $\int q(y,t) dt=1$ it follows that 
\[
q(y,0) \leq c_{K}(d(a(y),b(y))),
\]
where
\[
c_{K}(t) := \frac{s_{k}( t/2 )^{N-1}}{2} \bigg(  \int_{0}^{t/2} s_{K}(\tau)^{N-1} d\tau \bigg)^{-1} \leq \frac{C}{t},
\]
being $C$ a constant depending only on $K$.

To show that 
\[
\int_{-l}^{l} |q' (y,t)| dt < + \infty,
\]
it is enough to prove 
\[
\int_{-l}^{0} |q' (y,t)| dt < + \infty,
\]
From \eqref{E:der} it follows
\[
\omega'(y,t) := q'(y,t)  +(N-1)  \frac{s'_{K}( l-t)}{s_{K}( l) }q(y,0) \geq 0 
\]
so that 
$$
\TV \big( \omega(y,\cdot) ) \leq \bigg(1 + (N-1)\bigg( \frac{s_{K}(2l)}{s_{K}(l)} -1  \bigg) \bigg) q(y,0).
$$
Hence
\begin{align*}
\TV \big( q(y,\cdot), (-\ell,0] \big) \leq&~ \TV \big( \omega(y,\cdot), (-\ell,0] \big) + \TV \bigg( (N-1) \frac{s'_{K}}{s_{K}}, (-\ell,0] \bigg) q(0,y) \crcr 
\leq&~ \TV \big( \omega(y,\cdot), (-\ell,0] \big) + (N-1) \frac{s'_{K}(2 l)}{s_{K}(l)} q(0,y) \crcr
\leq&~ \bigg( 1 + 2\bigg( \frac{s_{K}(2l)}{s_{K}(l)} - 1 \bigg) q(y,0).
\end{align*}
Collecting all the estimates, we get
\[
\TV \big( q(y,\cdot) \big) \leq 2 \bigg(1 + 2\bigg( \frac{s_{K}(2l)}{s_{K}(l)} -1  \bigg) c_{K}(2l).
\]
\end{proof}

In general, the current $\dot g$ is not normal, as one can easily verify in $\mathbb T^2$ with the standard distance.

\section{Examples}
\label{S:examples}

We end this paper with some examples which shows how the different hypotheses of Section \ref{ss:Metric} enter into the analysis. In the following we denote the standard Euclidean scalar product in $\R^d$ as $\cdot$ and the standard distance in $\mathbb{T}^d$ by $|\cdot|$. We will also denote points by $p = (x,y,z,\dots) \in \R^d$, and $\alpha$ a fixed constant in $[0,1] \setminus \Q$.

\begin{example}[Non strongly consistent disintegration along rays]
\label{Ex:nonsection}
Consider the metric space
\[
(X,d) = \big( \mathbb{T}^2, |\cdot| \big)
\]
and the l.s.c. distance in the local chart $X = \{ (x,y) : 0 \leq x,y < 1 \}$
\[
d_L(p_1,p_2) :=
\begin{cases}
|x_1 - x_2 + i| & y_1 - y_2 = \alpha (x_1 - x_2) + i \alpha + n \crcr
+\infty & \text{otherwise}
\end{cases}
\]
for $i,n \in \Z$.
The sets $D_L$ are given by
\[
D_L(p_1) = \Big\{ (x,y) :   y = y_1 + \alpha (x - x_1 + i) \mod 1, i \in \N \Big\},
\]
so that it is easy to see that the partition $\{D_L(p)\}_{p \in X}$ does not yield a strongly consistent disintegration. Since $t \mapsto (t \mod 1, \alpha t \mod 1)$ is a continuous not locally compact geodesic, Condition \eqref{Cond:XdL5} is not verified in this system.

Consider the measures $\mu = \mathcal{L}^2 \llcorner_{\mathbb{T}}$ and the map $T : (x,y)\mapsto (x,y + \alpha \mod 1)$: being $\mu$ invariant w.r.t. translations, one has $T_\sharp \mu = \mu$, and moreover
\[
\int d_L(x,T(x)) \mu(dx) = 1.
\]

If we consider points $(p_i,(x_i, y_i + \alpha \mod 1))$, $i = 1,\dots,I$, then the only case for which $d_L(p_{i+1},p_i) < +\infty$ is when $p_{i+1} = (x_i + t \mod 1, y_i + \alpha t \mod 1)$ for some $t \in \R$, i.e. they belong to the geodesic
\[
\R \ni t \mapsto (x_i + t \mod 1, y_i + \alpha t \mod 1) \in X.
\]
Hence, to prove $d_L$-cyclical monotonicity, it is sufficient to consider path which belongs to a single geodesic, where $d_L$ reduces to the the one dimensional length:
\[
d_L \big( (x,y), (x + t \mod 1, y + \alpha t \mod 1) \big) = |t|.
\]
Since translations in $\R$ are cyclically monotone w.r.t. the absolute value, we conclude that $T$ is $d_L$-cyclically monotone.

The fact that the optimal rays coincide with the sets $D_L$ yields that the disintegration is not strongly consistent, in particular there is not a Borel section up to a saturated negligible set. Note that every transference plan which leaves the common mass in the same place has cost $0$, so that this example shows the necessity of Condition \eqref{Cond:XdL5} for Proposition \ref{P:ortho}.

\begin{figure}
\label{Fi:mongem1}
\psfrag{x}{$x$}
\psfrag{y}{$y$}
\psfrag{DL(P)}{$D_L(p)$}
\psfrag{T2}{$\mathbb{T}^2$}
\psfrag{L2}{$\mathcal{L}^2 \llcorner_{\mathbb{T}^2}$}
\psfrag{p}{$p$}
\psfrag{T(p)}{$T(p)$}
\centerline{\resizebox{9cm}{9cm}{\includegraphics{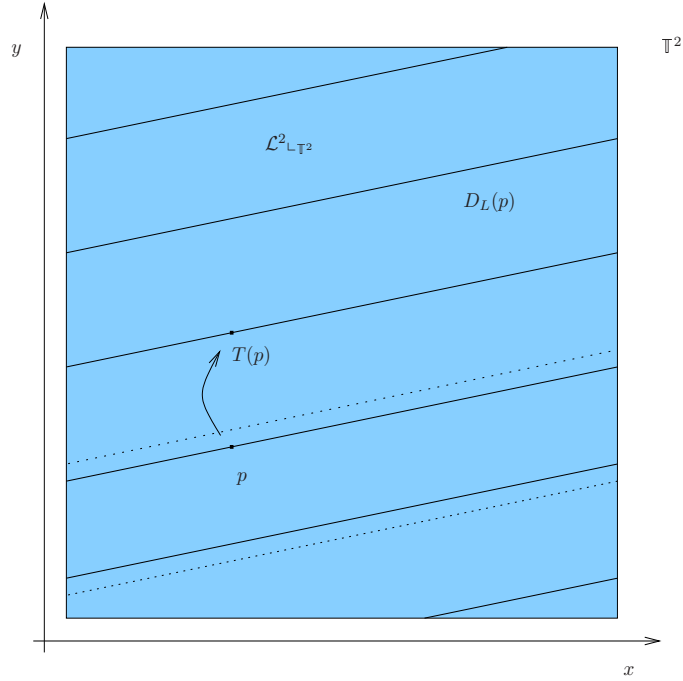}}}
\caption{The metric space of Example \ref{Ex:nonsection}}
\end{figure}

\end{example}

\begin{example}[Non optimality of transport map]
\label{Ex:nooptir}
Consider countable copies of the manifolds $\mathbb{T}^2$: we denote them in local coordinates by
\[
C := \big\{ (x,y) : 0 \leq x,y < 1 \big\}, \quad C^i := \big\{ (x^i,y^i) : 0 \leq x^i,y^i < 1 \big\},\ i \in \Z \setminus \{0\}.
\]
With this fixed choice of coordinates, identify the points $(x,0) \equiv (x^i,0)$ if $0 \leq x = x^i < 1$. In other words, we glue the sets $C$, $C^i$, $i \in \Z \setminus \{0\}$, along a maximal circle $S$, which will be written in local coordinates by
\[
S  = \big\{ \theta : 0 \leq \theta < 1 \big\}.
\]
The space $X$ is the set obtained with this procedure.

In the following points we need to divide $C$ into two parts: with the same coordinates as above, we set
\[
C^- := \big\{ (x,y) \in C : 0 \leq x < 1, 0 \leq y \leq 1/2 \big\}, \quad C^+ := \big\{ (x,y) \in C : 0 \leq x < 1, 1/2 < y < 1 \big\}.
\]

{\it Definition of $d$ and $d_L$.} The distance $d$ is defined as follows:
\[
d(p_1,p_2) = \min \Big\{ |p_1 - p_2|, |(x_1,y_1) - (\theta,0)| + |(x_2,y_2) -  (\theta,0)|, \theta \in S^1 \Big\}.
\]
Note that $p_1 - p_2$ can be computed only when the points belong to the same component. It is fairly easy to see that $(X,d)$ is a compact set, in particular Polish.

The distance $d_L$ is defined as follows: if $\gamma : [0,1] \to X$ is a $d$-Lipschitz map, then set
\[
L(\gamma) := \int_0^1 \omega(\gamma(t),\dot \gamma(t)) dt, \quad \omega(p,v) :=
\begin{cases}
|\dot \gamma| & p \in C^-, \dot \gamma \cdot (-1,\alpha) = 0 \crcr
4 |\dot \gamma| & p \in C^+ \setminus S, \dot \gamma \cdot (-1,\alpha) = 0 \crcr
|\dot \gamma| & p \in C^i \setminus S, \dot \gamma \cdot (-1,i\alpha) = 0 \crcr
+\infty & \text{otherwise}
\end{cases}
\]
In other words, the Lipschitz path with finite length are a countable union of segments in $C$ or $C^i$, $i \in \Z \setminus \{0\}$ with slope $(\alpha,1)$, $(i\alpha,1)$, respectively. The distance $d_L$ is defined then by
\[
d_L(p_1,p_2) := \inf \Big\{ L(\gamma): \gamma \in \text{Lip}([0,1],X), \gamma(0)=p_1, \gamma(1)=p_2 \Big\}.
\]

{\it Study of the distance $d_L$.} To study the distance $d_L$, observe that it is enough to analyze the induced distance on $S^1$. We consider the length of the return map on $S$ depending on which sets we are moving on:
\begin{enumerate}
\item if we take the path $\theta \to \theta+\alpha$ along $C$, then its length is $\frac{5}{2} \sqrt{1+\alpha^2}$;
\item if we take the path $\theta \to \theta+i\alpha$ along $C^i$, then its length is $\sqrt{1+(i\alpha)^2}$.
\end{enumerate}
In particular, geodesics starting from $S$ and ending in some $C^i$, $i \in \Z \setminus \{0\}$, never take values in $C \setminus S$. Due to the invariance w.r.t. translations $(x,y) \mapsto (x+\alpha \mod 1,y)$, it is sufficient to study the structure the metric space $(D_L((0,0)),d_L)$.

The set $D_L((0,0))$ is the set $\{y = \alpha x + z \alpha \mod 1, z \in \Z\}$ in each component $C$, $C^i$, $i \in \Z \setminus \{0\}$. The metric $d_L \llcorner_{D_L((0,0))}$ is obtained as follows: given two points $(p_1,p_2)$, we can connect them using a path on the same component or by connecting each of them to points $\theta_1$, $\theta_2$ of $S$, and using one of the $C^i$ to connect these last points.

It follows that $(D_L((0,0)),d_L)$ is geodesic, and a more careful analysis shows that $d_L$ is actually l.s.c.. Moreover, the fact that $\sqrt{\cdot}$ is subadditive yields that there are not geodesic of infinite length: in particular all the assumptions listed on Page \pageref{P:assumpDL} are satisfied.

{\it Transport problem.} Define the sets
\[
A := \bigg\{ (x,y) \in C: 0 \leq x < 1, \frac{1}{2} < y < \frac{5}{8} \bigg\}, \quad B := \bigg\{ (x,y) \in C: 0 \leq x < 1, \frac{7}{8} < y < 1 \bigg\},
\]
and the measures $\mu := \mathcal{L}^2 \llcorner_A$, $\nu := \mathcal{L}^2 \llcorner_B$. Consider the two maps defined on $A$
\[
\begin{array}{ccccc}
T^+ &:& A &\to& B \crcr
&& (x,y) &\mapsto& T^+(x,y) := \Big( x + \frac{3}{8} \alpha \mod 1, y + \frac{3}{8} \Big)
\end{array}
\]
\[
\begin{array}{ccccc}
T^- &:& A &\to& B \crcr
&& (x,y) &\mapsto& T^-(x,y) := \Big( x - \frac{5}{8} \alpha \mod 1, y + \frac{3}{8} \Big)
\end{array}
\]
It is standard to show that $T^\pm_\sharp \mu = \nu$.

Let $(p_i,T^+(p_i)) \in \text{graph}(T^+)$, $i = 1,\dots,I$: from the definition of $d_L$, $d_L(p_{i+1},T^+(p_i))$ can be either equal to $d_L(p_i,T^+(p_i))$ or greater than $3 \sqrt{1+\alpha^2}$ (by taking the path along $C^- \cup C^1$). Since $d_L(p_i,T^+(p_i)) = \frac{3}{2} \sqrt{1+\alpha^2}$, it follows that $T^+$ is $d_L$-cyclically monotone.

However, one has
\[
\int d_L(x,T^-(x)) \mu(dx) =  \frac{1}{8}\sqrt{1 + \alpha^{2}} < \frac{3}{2} \frac{1}{8}\sqrt{1 + \alpha^{2}} = \int d_L(x,T^+(x)) \mu(dx).
\]
Hence the $d_L$-cyclical monotonicity is not sufficient for optimality. Note that Assumption \ref{A:NDE} is verified.

\begin{figure}
\label{Fi:mongemetri2}
\psfrag{D}{$C$}
\psfrag{Ci}{$C^i$}
\psfrag{t}{$\theta$}
\psfrag{t+a}{$\theta+\alpha$}
\psfrag{S}{$S$}
\psfrag{t+ia}{$\theta+i\alpha$}
\psfrag{D-}{$C^-$}
\psfrag{D+}{$C^+$}
\psfrag{T-}{$T^-$}
\psfrag{T+}{$T^+$}
\psfrag{mu}{$\mu$}
\psfrag{nu}{$\nu$}
\centerline{\resizebox{14cm}{5cm}{\includegraphics{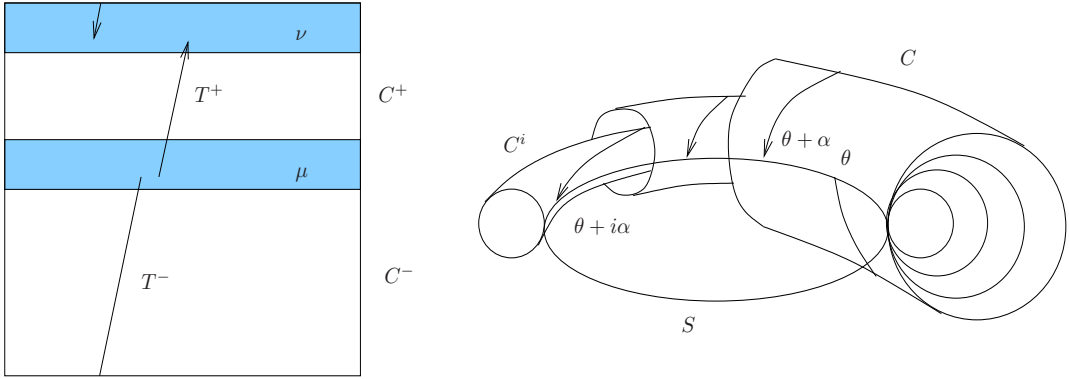}}}
\caption{The metric space of Example \ref{Ex:nooptir}.}
\end{figure}

\end{example}

\appendix

\section{Notation}
\label{S:notation}

\begin{tabbing}
\hspace{4cm}\=\kill
$P_{i_1\dots i_I}$ \> projection of $x \in \Pi_{k=1,\dots,K} X_k$ into its $(i_1,\dots,i_I)$ coordinates, keeping order
\\
$\mathcal{P}(X)$ or $\mathcal{P}(X,\Omega)$ \> probability measures on a measurable space $(X,\Omega)$
\\
$\mathcal{M}(X)$ or $\mathcal{M}(X,\Omega)$ \> signed measures on a measurable space $(X,\Omega)$
\\
$f \llcorner_A$ \> the restriction of the function $f$ to $A$
\\
$\mu \llcorner_A$ \> the restriction of the measure $\mu$ to the $\sigma$-algebra $A \cap \Sigma$
\\
$\mathcal{L}^d$ \> Lebesgue measure on $\R^d$
\\
$\mathcal{H}^k$ \> $k$-dimensional Hausdorff measure
\\
$\Pi(\mu_1,\dots,\mu_I)$ \> $\pi \in \mathcal{P}(\Pi_{i=1}^I X_i, \otimes_{i=1}^I \Sigma_i)$ with marginals $(P_i)_\sharp \pi = \mu_i \in \mathcal{P}(X_i)$
\\
$\mathcal{I}(\pi)$ \> cost functional \eqref{E:Ifunct}
\\
$c$ \> cost function $ : X \times Y \mapsto [0,+\infty]$
\\
$\mathcal{I}$ \> transportation cost \eqref{E:Ifunct}
\\
$\phi^c$ \> $c$-transform of a function $\phi$ \eqref{E:ctransf}
\\
$\partial^c \f$ \> $d$-subdifferential of $\f$ \eqref{E:csudiff}
\\
$\Phi_c$ \> subset of $L^1(\mu) \times L^1(\nu)$ defined in \eqref{E:Phicset}
\\
$J(\phi,\psi)$ \> functional defined in \eqref{E:Jfunct}
\\
$C_b$ or $C_b(X,\R)$ \> continuous bounded functions on a topological space $X$
\\
$(X,d)$ \> Polish space
\\
$(X,d_L)$ \> non-branching geodesic separable metric space
\\
$D_L(x)$ \> the set $\{y : d_L(x,y) < +\infty\}$
\\
$L(\gamma)$ \> length of the Lipschitz curve $\gamma$, Definition \ref{D:lengthstr}
\\
$\gamma_{[x,y]}(t)$ \> geodesics $\gamma : [0,1] \to X$ such that $\gamma(0) = x$, $\gamma(1) = y$
\\
$\gamma_{(x,y)}$, $\gamma_{[x,y]}$ \> open, closed geodesics \eqref{E:opeclgeo}
\\
$B_r(x)$ \> open ball of center $x$ and radius $r$ in $(X,d)$
\\
$B_{r,L}(x)$ \> open ball of center $x$ and radius $r$ in $(X,d_L)$
\\
$\mathcal{K}(X)$ \> space of compact subsets of $X$
\\
$d_H(A,B)$ \> Hausdorff distance of $A$, $B$ w.r.t. the distance $d$
\\
$A_x$, $A^y$ \> $x$, $y$ section of $A \subset X \times Y$ \eqref{E:sectionxx}
\\
$\mathcal{B}$, $\mathcal{B}(X)$ \> Borel $\sigma$-algebra of $X$ Polish
\\
$\Sigma^1_1$, $\Sigma^1_1(X)$ \> the pointclass of analytic subsets of Polish space $X$, i.e.~projection of Borel sets
\\
$\Pi^1_1$ \> the pointclass of coanalytic sets, i.e.~complementary of $\Sigma^1_1$
\\
$\Sigma^1_n$, $\Pi^1_n$ \> the pointclass of projections of $\Pi^1_{n-1}$-sets, its complementary
\\
$\Delta^1_n$ \> the ambiguous class $\Sigma^1_n \cap \Pi^1_n$
\\
$\mathcal{A}$ \> $\sigma$-algebra generated by $\Sigma^{1}_{1}$
\\
$\mathcal{A}$-function \> $f : X \to \R$ such that $f^{-1}((t,+\infty])$ belongs to $\mathcal A$
\\
$h_\sharp \mu$ \> push forward of the measure $\mu$ through $h$, $h_\sharp \mu(A) = \mu(h^{-1}(A))$
\\
$\textrm{graph}(F)$ \> graph of a multifunction $F$ \eqref{E:graphF}
\\
$F^{-1}$ \> inverse image of multifunction $F$ \eqref{E:inverseF}
\\
$F_x$, $F^y$ \> sections of the multifunction $F$ \eqref{E:sectionxx}
\\
$\mathrm{Lip}_1(X)$ \> Lipschitz functions with Lipschitz constant $1$
\\
$\Gamma'$ \> transport set \eqref{E:gGamma}
\\
$G$, $G^{-1}$ \> outgoing, incoming transport ray, Definition \ref{D:Gray}
\\
$R$ \> set of transport rays \eqref{E:Rray}
\\
$\mathcal{T}$, $\mathcal{T}_e$ \> transport sets \eqref{E:TR0}
\\
$a,b : \mathcal{T}_e \to \mathcal{T}_e$ \> endpoint maps \eqref{E:endpoint0}
\\
$\mathcal{Z}_{m,e}$, $\mathcal Z_m$ \> partition of the transport set $\Gamma$ \eqref{E:Zkije}, \eqref{E:Zkij}
\\
$\mathcal S$ \> cross-section of $R \llcorner_{\mathcal T \times \mathcal T}$
\\
$g = g^+ \cup g^-$ \> ray map, Definition \ref{D:mongemap}
\\
$A_t$ \> evolution of $A \subset \mathcal{Z}_{k,i,j}$ along geodesics \eqref{E:At}
\\
$\dot g$ \> current on $(X,d)$ corresponding to the flow along geodesics, Definition \ref{D:dotgamma}
\\
$\partial \dot g$ \> boundary of the current $\dot \gamma$ \eqref{E:boundgamma}
\\
$H$\> be the set of all geodesics as a subset of $\textrm{Lip}_1([0,1],X)$
\\
$e_{t}(\gamma)$\> evaluation map \eqref{E:evalua}
\\
$\Xi$\> dynamical transference plan
\\
$s_{K}$\>  the map defined in \eqref{E:sk}
\\
$A^{s,\tilde H}$\> the evolution of the set $A$ along the geodesics $\tilde H \subset H$ \eqref{E:evodue}
\\
$H(G)$\> the set of geodesics used by $G$ \eqref{E:geodue}

\end{tabbing}

\bibliography{biblio}

\end{document}